\documentclass[11pt,twoside, a4paper, english, reqno]{amsart}
\usepackage{graphicx, amsmath, varioref, amscd, amssymb,color, bm, stmaryrd, amsthm}
\usepackage{fancybox}
\usepackage[pdftex, colorlinks=true, backref]{hyperref}

\numberwithin{equation}{section}
\allowdisplaybreaks

\newtheorem{theorem}{Theorem}[section]
\newtheorem{lemma}[theorem]{Lemma}
\newtheorem{corollary}[theorem]{Corollary}

\newtheorem{proposition}[theorem]{Proposition}
\theoremstyle{definition}
\newtheorem{definition}[theorem]{Definition}

\theoremstyle{remark}
\newtheorem{remark}[theorem]{Remark}

\newcommand{\Div}{\operatorname{div}}

\newcommand{\Curl}{\operatorname{curl}}
\newcommand{\Grad}{\nabla}
\newcommand{\vr}{\varrho}

\newcommand{\vc}[1]{{\bm{#1}}}

\newcommand{\weak}{\rightharpoonup}

\newcommand{\weakstar}{\overset{\star}\rightharpoonup}

\newcommand{\dist}{\operatorname{dist}}

\newcommand{\norm}[1]{\left\Vert#1\right\Vert}

\newcommand{\Set}[1]{\left\{#1\right\}}
\newcommand{\jump}[1]{\left\llbracket #1\right\rrbracket}

\newcommand{\vrho}{\varrho}
\newcommand{\inb}{\in_{\text{b}}}
\newcommand{\Dt}{\Delta t}
\newcommand{\R}{\mathbb{R}}
\newcommand{\weakto}{\rightharpoonup}

\newcommand{\Om}{\ensuremath{\Omega}}

\newcommand{\pOm}{\ensuremath{\partial\Omega}}
\newcommand{\Dom}{(0,T)\times\Omega}

\newcommand{\eps}{\epsilon}

\newcommand{\up}{\operatorname{Up}}
\newcommand{\ov}{\widehat}

\begin{document}

%\title[A convergent method for  compressible Navier-Stokes]{Convergence of a combined \\ Finite Element - Discontinuous Galerkin method  for the Compressible \\ Navier-Stokes equations} 
\title[A convergent method for  compressible Navier-Stokes]
{A convergent FEM-DG method for \\ the compressible Navier-Stokes equations} 

\author[Karper]{Trygve K. Karper}\thanks{This research is supported by the Research Council of Norway (proj. 205738)}

\address[Karper]{\newline
Center for Scientific Computation and Mathematical Modeling, University of Maryland, College Park, MD 20742}
\email[]{\href{karper@gmail.com}{karper@gmail.com}}
\urladdr{\href{http://folk.uio.no/~trygvekk}{folk.uio.no/\~{}trygvekk}}

\date{\today}



\maketitle
\begin{abstract}
This paper presents a new numerical method
for the compressible Navier-Stokes equations governing the flow of an ideal isentropic gas. To approximate the continuity equation,
the method utilizes a discontinuous Galerkin discretization on piecewise constants and a basic upwind flux. 
For the momentum equation, the method is a new combined discontinuous Galerkin 
and finite element method approximating the velocity in the Crouzeix-Raviart finite element space.
While the diffusion operator is discretized in a standard fashion,
the convection and time-derivative are discretized using discontinuous Galerkin on the element average  
 velocity and a Lax-Friedrich type flux. Our main result is convergence of the method to a global weak solution
as discretization parameters go to zero. The convergence analysis
constitutes a numerical version of the existence analysis of Lions and Feireisl.

\end{abstract}
\setcounter{tocdepth}{1}
\tableofcontents

\section{Introduction}
In this paper, we will construct 
a convergent numerical method for  
 the compressible Navier-Stokes equations:
\begin{align}
	\vr_t + \Div(\vr u) &= 0 \qquad\qquad\qquad\quad \text{ in $(0,T)\times \Om$},\label{eq:cont}\\
	(\vr u)_t + \Div(\vr u\otimes u)  &= - \Grad p(\vr) + \Delta u,\quad \text{ in $(0,T)\times \Om$}, \label{eq:moment}
\end{align}
where $\Om \subset \R^3$ is an open, bounded, domain with Lipschitz boundary $\pOm$,
and $T>0$ is a given finite final time.
The unknowns are the fluid density $\vr$
and vector velocity $u$, while the operator $\otimes$ denotes the tensor 
product of two vectors ($(v \otimes v)_{i,j} = v_iv_j$). The mechanism driving the flow is the pressure $p(\vr)$
which is assumed to be that of an ideal isentropic gas (constant entropy);
\begin{equation*}
	p(\vr) = a\vr^\gamma.
\end{equation*}
To close the system of equations \eqref{eq:cont} - \eqref{eq:moment},
 standard no-slip boundary condition is assumed
\begin{equation}\label{eq:boundary}
	u|_{\pOm} = 0,
\end{equation}
together with initial data
\begin{equation}\label{eq:init}
	\vr(0,\cdot) = \vr_0 \in L^{\gamma+1}(\Om),\qquad  \vr u(0, \cdot)  = m_0 \in L^\frac{3}{2}(\Om).
\end{equation}

From the point of view of applications, 
the system \eqref{eq:cont} - \eqref{eq:moment} is the simplest 
form of the equations governing the flow of an ideal viscous and isentropic gas \cite{Batchelor, Serrin}.
In the available engineering literature, the reader can find a variaty of flows 
for which the assumption of constant entropy (reversibility) is a reasonable approximation. 
However, while  viscosity in \eqref{eq:moment} is modeled by the Laplace operator ($\Delta u$),
 practical applications will make use of a more appropriate
stress tensor, the simplest being that of a Newtonian fluid
with constant coefficients
\begin{equation*}
	\Div \mathbb{S} = \mu\Div \left(\Grad u + \Grad u^T\right) + \lambda \Grad \Div u.
\end{equation*}
Note that our diffusion is a special case of $\mathbb{S}$. Indeed, $\mathbb{S}$ reduces to the Laplace operator when $\mu = 1$ and $\lambda = \frac{1}{2}$.
The analysis in this paper can be generalized to  all relevant cases 
of constant non-vanishing $\mu$ and $\lambda$. However, this comes at the cost of 
more terms as the chosen finite element space (Crouzeix-Raviart) does not satisfy Korn's inequality \cite{Mardal},
but with a negligible gain in terms of ideas and novelty. The more physically 
relevant case where $\mu$ and $\lambda$ are functions of the density of the form $\vr^\frac{\gamma-1}{2}$
is not included in the available existence theory (cf.~\cite{cc, Jenssen})
and there does not seem to be any obvious way of including it here either.

In terms of physical applicability of the forthcoming results,
a more pressing issue is the equation of state for the pressure.
For purely technical reasons, we will be forced to require that
 $\gamma$ is greater than the spatial dimension;
$$
\gamma > 3.
$$ 
This is a severe restriction on $\gamma$ with no physical applications (to the author's knowledge).
Kinetic theory predicts a value of $\gamma$ depending on the specific 
gas in question: $\sim \frac{5}{3}$ for monoatomic gas (e.g~helium), $\sim \frac{7}{5}$ for diatomic 
gas (e.g~ air), and creeping towards one for more complex molecules and/or higher temperatures. 
It should be noted that global existence is only known for $\gamma > \frac{3}{2}$
and hence only in the case of monoatomic gas (see \cite{Feireisl} and the references therein). 
In this paper, the restriction on $\gamma$ 
seems absolutely necessary to prove convergence of the method, but is not required for stability.
In fact, the strict condition on $\gamma$ is related to the numerical diffusion introduced by
the method and is in perfect analogy to the problems encountered in a vanishing diffusion limit for
 the continuity equation \eqref{eq:cont} 
(see for instance \cite{Feireisl}). 
That being said, it might be calming that the upcoming analysis  can be easily extended to 
pressures of the form
\begin{equation*}
	p(\vr) = a\vr^{\gamma_1} + \kappa \vr^{\gamma_2}, \quad \gamma_1 \in (1,2), \quad \gamma_2 > 3,
\end{equation*}
where $\kappa > 0$ may be chosen arbitrarily small. Hence, for all practical purposes, 
the numerical method presented herein is convergent for cases 
very close to the physically relevant ones.

The numerical literature contains a vast body of  methods for compressible fluid flows. 
Many of them are widely applied and constitute an indispensable 
tool in a variety of disciplines such as engineering, meteorology, or astrophysics. 
While the complexity and range of applicability of numerical methods for compressible 
viscous fluids is increasing, the convergence properties of any of these methods 
have thus far remained unknown.
In fact, prior to this paper, there have not been reported any 
convergence results for numerical approximations of compressible viscous flow 
in more than one spatial dimension.
In one dimension, the available results are due to David Hoff and collaborators \cite{Chen:2000yq, Zarnowski:1991uq, Zhao:1994fk, Zhao:1997qy} (see also \cite{Jenssen}) 
and concerns the equations posed in Lagrangian variables and 
relies on the {\small \sc 1D} existence theory which have yet to be extended to multiple dimensions.
That being said, in the recent years there have appeared a number of convergent methods 
 for simplified versions of \eqref{eq:cont}-\eqref{eq:moment}. 
In \cite{Eymard, Gallouet2, Gallouet1}, convergent finite volume and 
finite element methods for the stationary compressible Stokes equations 
was developed. Simultaneously, in \cite{Karlsen1, Karlsen2, Karlsen3}, K. Karlsen and the author developed 
convergent finite element methods for the non-stationary compressible Stokes 
equations. In the upcoming analysis, we will utilize ideas from all of these recent papers.

Let us now discuss the choice of numerical method.
For the approximation of the continuity equation, we will use the standard 
upwind finite volume method. However, we will formulate this method 
as a discontinuous Galerkin method where the density $\vr$ 
is approximated by piecewise constants. For the velocity we will use 
the Crouzeix-Raviart finite element space. The method and some its properties 
was originally inspired by \cite{Walkington:2005jl, Liu:2007fe} (cf. \cite{Karlsen1}), but variants 
can be found several places in the literature  (see for instance \cite{Gallouet:2006lr}).
For the momentum equation, we are leaning on the knowledge gained through 
\cite{Eymard, Gallouet2, Gallouet1, Karlsen1, Karlsen2, Karlsen3}, 
from which it is evident that proving convergence for any numerical method is a hard task.
In particular, to develop a numerical analogy of the continuous existence 
theory, it seems necessary that the numerical discretization of the diffusion 
and pressure respects orthogonal Hodge decompositions (see Section \ref{sec:hodge}
for an explanation) in some form.

The distinctively new and completely original feature of the upcoming method   is the 
discretization of the material derivative in the momentum equation. 
Our discretization will be of the form
\begin{equation*}
	\begin{split}
		&(\vr u)_t + \Div (\vr  u\otimes u) \\
		&\qquad \approx \int_\Om \partial_t (\vr_h \widehat u_h) v_h~dx - \sum_{ \Gamma} \int_{\Gamma} \up(\vr u\otimes \ov u)\jump{\ov v_h}_\Gamma ~dS(x),
		\quad \forall v_h,
		\end{split}
\end{equation*}	
where $\widehat{v_h}$ denotes the $L^2$ projection of $v_h$ onto piecewise constants.
The operator $\up(\vr u \otimes u)$ is the specific upwind flux which we shall use (see Section \ref{sec:num} for precision). 
Hence, the material derivative is discretized using discontinuous Galerkin \emph{with the same polynomial order} 
as the continuity discretization. 
This stands in contrast to the diffusion and pressure terms which are solved with the Crouzeix-Raviart element space 
and hence in particular with piecewise constant divergence $\Div v_h$ matching the density (and pressure) space.
In the language of finite differences, this means that the pressure and diffusion is solved using staggered grid 
while the material derivative are solved at the same nodal values as the density.
The great benefit of this approach is that it solves the long-lasting problem 
of incorporating both the hyperbolic nature of the material derivative and the  
nature of the diffusion-pressure coupling. In particular, our main result yields stability and convergence 
for all Mach and Reynolds numbers.

By proving convergence of a numerical method
we will in effect also give an alternative
existence proof for global weak solutions to the equations \eqref{eq:cont}-\eqref{eq:moment}. 
While the first global existence result for  incompressible 
Navier-Stokes  was achieved by Leray more than 
80 years ago, a similar result for compressible flow
was obtained by P-L. Lions in the mid 90s. In the celebrated book \cite{Lions:1998ga}, 
Lions obtains 
weak solutions of \eqref{eq:cont}-\eqref{eq:moment} 
as the a.e weak limit of a sequence of approximate solutions.
Consequently, from the point of view of analysis, we will in 
this paper perform similar analysis to that of \cite{Lions:1998ga},
but for a numerical method. That is, we will develop a numerical 
analogy of the continuous existence theory.
The key difficulty when passing to a limit is 
presented by the non-linear pressure $p(\vr)$. Specifically, 
 compactness of $\vr$ is needed, while the available 
estimates provides no form of continuity of $\vr$.
In all current proofs of existence, the necessary compactness is 
derived using renormalized solutions together with a remarkable 
sequential continuity result for the quantity $p(\vr) - \Div u$. 
The result provides a.e convergence of the density, but gives 
no insight into continuity properties of $\vr$.
In the original proof, Lions needed that
 $\gamma > \frac{9}{5}$. The existence theory was further developed by Feireisl 
and the requirement lowered to $\gamma > \frac{3}{2}$. This seems to be optimal
in the absence of pointwise bounds on the density. However, it is interesting
that this still does not include the case of air in three dimensions.
The reader is strongly encouraged to consult \cite{Straskraba} for a thorough and well-written introduction to 
the mathematical theory of solutions to \eqref{eq:cont}-\eqref{eq:moment}.

\subsection*{Organization of the paper:}
In the next section, we will go through some preliminary 
knowledge where we attempt to make clear the solution concept, basic compactness 
results, and the finite element theory used in the analysis.
Then, in Section \ref{sec:num}, we present the numerical method, 
give some basic properties of the method, and state the main convergence result (Theorem \ref{thm:main}).
We will then move on to Section \ref{sec:stability} in which we establish 
stability of the method and draw conclusions in terms of uniform integrability.
The following section, Section \ref{sec:operators}, is a
fundamental section  where we will estimate the weak
error (in a weak norm) of the transport operators in the discretization. 
The material contained in this section will be used ubiquitously in the convergence analysis. 
In Section \ref{sec:higher} we prove higher integrability 
of the density. That is, the density enjoys more integrability than the energy 
estimate provides. Then, in Section \ref{sec:weak}, we will pass to the limit 
in the method and conclude that the limit is almost a global weak solution.
It will then only remain to prove strong convergence of the density approximation.
In Section \ref{sec:flux}, we establish the fundamental ingredient in 
the proof of density compactness; - weak sequential continuity of the effective 
viscous flux. Finally, in Section \ref{sec:strong} we prove strong convergence 
of the density and conclude the main result (Theorem \ref{thm:main}).
The paper ends with an appendix section containing the proof of well-posedness for the numerical method.

\section{Preliminary material}
The purpose of this section is to state
some  results that  will be needed in the upcoming 
convergence analysis.

\subsection{Weak and renormalized solutions}

\begin{definition}\label{def:weak}
We say that a pair $(\vr, u)$ is a weak solution 
of \eqref{eq:cont} - \eqref{eq:moment} with initial condition \eqref{eq:init}
and boundary condition \eqref{eq:boundary} provided:
\begin{enumerate}
	\item The continuity equation \eqref{eq:cont} holds in the sense of distributions
	\begin{equation*}
		\int_0^T\int_\Om \vr(\phi_t + u \cdot \Grad \phi)~dxdt = -\int_\Om \vr_0 \phi(0,\cdot)~dx,
	\end{equation*}	
	for all $\phi \in C^\infty_0([0,T) \times \overline{\Om})$.
	\item The momentum equation \eqref{eq:moment} holds in the sense
	\begin{equation*}
		\begin{split}
					&\int_0^T\int_\Om -\vr u v_t - \vr u \otimes u : \Grad v - p(\vr)\Div v + \Grad u \Grad v~dxdt\\
					&\qquad= \int_\Om m_0 v(0,\cdot)~dx,
		\end{split}
	\end{equation*}
	for all $v \in [C_0^\infty([0,T) \times \Om)]^3$.
	\item The energy inequality holds
	\begin{equation}\label{eq:cont-energy}
		\begin{split}
				&\sup_{t \in (0,T)}\int_\Om \frac{\vr u^2}{2} + \frac{1}{\gamma-1}p(\vr)~dx \\
				&\qquad + \int_0^T\int_\Om |\Grad u|^2~dxdt 
				 \leq \int_\Om \frac{m_0^2}{2\vr_0 } + \frac{1}{\gamma-1}p(\vr_0)~dx.
		\end{split}
	\end{equation}
\end{enumerate}
\end{definition}

\begin{definition}[Renormalized solutions]
\label{renormlizeddef}
Given $u \in L^2(0,T; W^{1,2}_{0}(\Omega))$, we 
say that $\varrho\in  L^\infty(0,T;L^\gamma(\Omega))$ 
is a renormalized solution of \eqref{eq:cont} if 
$$
B(\vrho)_t + \Div \left(B(\vrho)u\right) + b(\vrho)\Div u = 0,
$$
in $\mathcal{D}'\left([0,T)\times \overline{\Om}\right)$
for any $B\in C[0,\infty)\cap C^1(0,\infty)$ with $B(0)=0$ and $b(\vrho) := B'(\vrho)\vrho - B(\vrho)$.
\end{definition}

We shall need the following well-known lemma \cite{Lions:1998ga} stating that 
square-integrable weak solutions $\vrho$ are also renormalized solutions.

\begin{lemma}
\label{lemma:feireisl}
Suppose $(\varrho, u)$ is a weak solution according to Definition \ref{def:weak}.
If $\varrho \in L^2((0,T)\times \Omega))$, then $\vrho$ is a renormalized solution 
according to Definition \ref{renormlizeddef}.
\end{lemma}

\subsection{Compactness results}
In the analysis, we shall need a number of 
compactness results. 

\begin{lemma}\label{lem:timecompactness}
Let $X$ be a separable Banach space, and suppose $v_n\colon [0,T]\to
X^\star$, $n=1,2,\dots$, is a sequence for which 
$\norm{v_n}_{L^\infty([0,T];X^\star)}\le C$, for some constant $C$ independent of $n$. 
Suppose the sequence $[0,T]\ni t\mapsto \langle v_n(t),\Phi \rangle_{X^\star,X}$, $n=1,2,\dots$, 
is equi-continuous for every $\Phi$ that belongs to a dense subset of $X$.  
Then $v_n$ belongs to $C\left([0,T];X^\star_{\mathrm{weak}}\right)$ for every
$n$, and there exists a function $v\in 
C\left([0,T];X^\star_{\mathrm{weak}}\right)$ such that along a 
subsequence as $n\to \infty$ there holds $v_n\to v$ in 
$C\left([0,T];X^\star_{\mathrm{weak}}\right)$.
\end{lemma}

To obtain strong compactness of the density 
approximation, we will utilize the following lemma.
\begin{lemma}\label{lem:prelim} 
Let $O$ be a bounded open subset of $\R^M$, $M\ge 1$.  
Suppose $g\colon \R\to (-\infty,\infty]$ is a lower semicontinuous 
convex function and $\Set{v_n}_{n\ge 1}$ is a sequence of 
functions on $O$ for which $v_n\weakto v$ in $L^1(O)$, $g(v_n)\in L^1(O)$ for each 
$n$, $g(v_n)\weakto \overline{g(v)}$ in $L^1(O)$. Then 
$g(v)\le \overline{g(v)}$ a.e.~on $O$, $g(v)\in L^1(O)$, and
$\int_O g(v)\ dy \le \liminf_{n\to\infty} \int_O g(v_n) \ dy$. 
If, in addition, $g$ is strictly convex on an open interval
$(a,b)\subset \R$ and $g(v)=\overline{g(v)}$ a.e.~on $O$, 
then, passing to a subsequence if necessary, 
$v_n(y)\to v(y)$ for a.e.~$y\in \Set{y\in O\mid v(y)\in (a,b)}$.
\end{lemma}

In what follows, we will often obtain a priori estimates for a sequence $\Set{v_n}_{n\ge 1}$ 
that we write as ``$v_n\inb X$'' for some functional space $X$. What this really means is that 
we have a bound on $\norm{v_n}_X$ that is independent of $n$.
The following lemma is taken from \cite{Karlsen1}.
\begin{lemma}\label{lemma:aubinlions}
Given $T>0$ and a small number $h>0$, write 
$(0,T] = \cup_{k=1}^M(t_{k-1}, t_{k}]$ with $t_{k} = hk$ and $Mh = T$. 
Let $\{f_{h}\}_{h>0}^\infty$, $\{g_{h}\}_{h>0}^\infty $ be two sequences such that:
\begin{enumerate}
\item{} the mappings $t \mapsto g_{h}(t,x)$ and $t\mapsto f_{h}(t,x)$ 
are constant on each interval $(t_{k-1}, t_{k}],\ k=1, \ldots, M$.

\item{}$\{f_{h}\}_{h>0}$ and $\{g_{h}\}_{h>0}$ converge weakly to 
$f$ and $g$ in $L^{p_{1}}(0,T;L^{q_{1}}(\Om))$ and 
$L^{p_{2}}(0,T;L^{q_{2}}(\Om))$, respectively, 
where $1 < p_{1},q_{1}< \infty$ and
$$
\frac{1}{p_{1}} + \frac{1}{p_{2}} = \frac{1}{q_{1}} + \frac{1}{q_{2}} = 1.
$$

\item{} the discrete time derivative satisfies
$$
\frac{g_{h}(t,x) - g_{h}(t-h,x)}{h} \in_{b} L^1(0,T;W^{-1,1}(\Om)).
$$

\item{}$\|f_{h}(t,x) - f_{h}(t,x-\xi)\|_{L^{p_{2}}(0,T;L^{q_{2}}(\Om))} 
\rightarrow 0$ as $|\xi|\rightarrow 0$, uniformly in $h$.
\end{enumerate}

Then, $g_{h}f_{h} \weak gf$ in the sense of distributions on $\Dom$.
\end{lemma}

\subsection{Finite element spaces and some basic properties} \label{sec:CRM}

Let $E_h$ denote a shape regular tetrahedral mesh of $\Om$. Let 
$\Gamma_h$ denote 
the set of faces in $E_h$. We will approximate the density 
in the space of piecewise constants on $E_{h}$ and 
denote this space by $Q_{h}(\Om)$. For the approximation of the velocity we 
use the Crouzeix--Raviart element space \cite{Crouzeix:1973qy}:
\begin{equation}\label{def:CR}
	\begin{split}
		V_{h}(\Om) &= \Bigg{\{}v_h \in L^2(\Om) : 
		~v_{h}|_{E} \in \mathbb{P}_{1}^3(E), \ 
		\forall E \in E_{h}, \\
		&\qquad \qquad\qquad\qquad
		\int_{\Gamma}\jump{v_{h}}_\Gamma\ dS(x) = 0, \ 
		\forall \Gamma \in \Gamma_{h}\Bigg\},		
	\end{split}
\end{equation}
where $\jump{\cdot}_\Gamma$ denotes the jump across a face $\Gamma$.
To incorporate the boundary condition, we let the 
degrees of freedom of $V_{h}(\Om)$ vanish at the boundary.
Consequently, the finite element method is nonconforming 
in the sense that the velocity approximation 
space is not a subspace of the corresponding 
continuous space, $W_0^{1,2}(\Om)$.

For the purpose of analysis, we shall also 
need the div-conforming N{\'e}d{\'e}lec finite element 
space of first order and kind \cite{Nedelec:1980ec, Nedelec:1986bv}
\begin{equation}\label{def:nedelec}
	\begin{split}
		\mathcal{N}_h(\Om) = \Bigg\{v_h, ~\Div v_h \in L^2(\Om): 
		~ &v_h|_E \in \mathbb{P}_0^3 \oplus \mathbb{P}_0^1 \vc{x}, ~\forall E \in E_h, \\
		&\int_\Gamma \jump{v\cdot \nu}\ dS(x)= 0, 
		~\forall \Gamma \in \Gamma_h \Bigg\}.
	\end{split}
\end{equation}
We introduce the canonical interpolation operators
\begin{equation}\label{def:interpolant}
	\begin{split}
		\Pi_h^V: W^{1,2}_0(\Om) & \mapsto V_h(\Om), \\
		\Pi_h^Q: L^2(\Om) &\mapsto Q_h(\Om), \\
		\Pi_h^\mathcal{N}: W^{1,2}(\Om) &\mapsto \mathcal{N}_h(\Om),
	\end{split}
\end{equation}
defined by
\begin{equation}\label{eq:opdef}
	\begin{split}
		\int_\Gamma \Pi_h^V v_h \ dS(x) & 
		= \int_\Gamma v_h \ dS(x), 
		\quad\forall \Gamma \in \Gamma_h, \\
		\int_\Gamma \Pi_h^\mathcal{N}v_h \cdot \nu~ dS(x) &= \int_\Gamma v_h \cdot \nu~ dS(x), \quad\forall \Gamma \in \Gamma_h,\\
		\int_E \Pi_h^Q \phi~dx &= \int_E \phi \ dx, 
		\quad \forall E \in E_h.
	\end{split}
\end{equation}
Then, by virtue of \eqref{eq:opdef} and Stokes' theorem,
\begin{equation}\label{eq:commute}
	\Div \Pi_{h}^\mathcal{N} v = \Div_{h} \Pi_{h}^Vv = \Pi_{h}^Q \Div v,
	\qquad  
	\Curl_{h} \Pi_{h}^Vv = \Pi_{h}^Q \Curl v, 
\end{equation}
for all $v \in W_0^{1,2}(\Om)$. Here, $\Curl_h$ and $\Div_h$ denote 
the curl and divergence operators, respectively, taken inside each element.

Let us now state some basic properties of the finite element spaces. 
We start by recalling from \cite{Brezzi:1991lr,Crouzeix:1973qy, Nedelec:1980ec} 
a few interpolation error estimates. 

\begin{lemma}\label{lemma:interpolationerror} % -- LEMMA
There exists a constant $C>0$, depending only on the 
shape regularity of $E_h$, such 
that for any $1\leq p\leq\infty$,
\begin{equation*}
	\begin{split}
		&\|\Pi_h^Q \phi-\phi \|_{L^p(E)} 
		\leq Ch\|\Grad \phi\|_{L^p(E)}, \\
		&\|\Pi_{h}^V v-v\|_{L^p(E)} 
		+h\|\Grad (\Pi_{h}^V v-v)\|_{L^p(E)} 
		\leq ch^s\|\Grad^sv\|_{L^p(E)},\quad s = 1,2, \\
		& \|\Pi_h^\mathcal{N}v - v\|_{L^p(E)}
		+ h\|\Div (\Pi_h^\mathcal{N} v - v)\|_{L^p(E)} 
		\leq ch^s\|\Grad^sv\|_{L^p(E)},\quad s = 1,2, \\
	\end{split}
\end{equation*}
for all $\phi \in W^{1,p}(E)$ and $v \in W^{s,p}(E)$.
\end{lemma}

By scaling arguments, the trace theorem, and 
the Poincar\'e inequality, we obtain 

\begin{lemma}\label{lemma:toolbox}
For any $E \in E_{h}$ and $\phi \in W^{1,2}(E)$, 
we have
\begin{enumerate}
\item{}trace inequality,
$$
\|\phi\|_{L^2(\Gamma)} \leq 
ch_{E}^{-\frac{1}{2}}\left(\|\phi\|_{L^2(E)}
+h_{E}\|\Grad \phi\|_{\vc{L}^2(E)}\right), 
\quad \forall 
\Gamma \in \Gamma_{h}\cap \partial E.
$$
\item{} Poincar\'e inequality,
$$
\left\|\phi - \frac{1}{|E|}\int_{E}\phi\ dx \right\|_{L^2(E)} 
\leq Ch_{E}\|\Grad \phi\|_{\vc{L}^2(E)}.
$$
\end{enumerate}
In both estimates, $h_E$ is the diameter of the element $E$.
\end{lemma}

Scaling arguments and the equivalence of finite dimensional norms yields
the classical inverse estimate (cf.~\cite{Brenner}):
\begin{lemma}\label{lemma:inverse}
There exists a positive constant $C$, depending only on the shape 
regularity of $E_h$, such that for $1\leq q,p \leq \infty$ 
and $r= 0,1$,
\begin{equation*}
	\norm{\phi_h}_{W^{r,p}(E)} 
	\leq C h^{-r + \min\{0, \frac{3}{p}-\frac{3}{q}\}}
	\norm{\phi_h}_{L^q(E)}, 
\end{equation*}
for any $E \in E_h$ and all polynomial 
functions $\phi_h \in \mathbb{P}_k(E)$, $k=0,1,\ldots$.
\end{lemma}

Since the Crouzeix-Raviart element space is not a subspace 
of $W^{1,2}$, it is not a priori clear that functions in this space 
are compact in $L^p$, $p<6$. However, from the Sobolev inequality
on each element and an interpolation argument 
we obtain the needed result.

\begin{lemma}\label{lem:translation}
For $u_h \in V_h(\Om)$, let $p$ satisfy $2 \leq p < 6$ and determine $a$ such that
$$
\frac{1}{p} = \frac{a}{2} + \frac{(1-a)}{6}.
$$
The following space translation estimate holds
\begin{equation*}
	\|u_h(\cdot) - u_h(\cdot -\xi)\|_{L^p(\Om\setminus \Set{x: \dist(x, \partial \Om)}< |\xi|)} \leq C\left(h^2 + |\xi|^2\right)^\frac{a}{2} \|\Grad u_h\|_{L^2(\Om)},
\end{equation*}
where $C$ is independent of $h$ and $\xi$.
\end{lemma}
\begin{proof}
From the work of Stummel \cite{Stummel:1980fk}, we have that
\begin{equation}\label{stummel1}
	\|u_h(\cdot) - u_h(\cdot -\xi)\|_{L^2(\Om)} \leq C\left(h^2 + |\xi|^2\right)^\frac{1}{2}\|\Grad_h u_h\|_{L^2(\Om)}.
\end{equation}
The standard Sobolev inequality gives
\begin{equation}\label{stummel2}
	\|u_h(\cdot) - u_h(\cdot -\xi)\|_{L^6(\Om)} \leq C\|\Grad_h u_h\|_{L^2(\Om)}.
\end{equation}
Hence, the proof follows from interpolation between \eqref{stummel1} and \eqref{stummel2}.

\end{proof}

Finally, we recall the following well-known property 
of the Crouzeix-Raviart element space. 

\begin{lemma}\label{lem:amazing}
For any $u_h \in V_h(\Om)$ and $v \in W^{1,2}_0(\Om)$,
\begin{equation}\label{eq:magic1}
	\int_\Om \Grad_h u_h \Grad_h \left(\Pi_h^V v - v\right)~dx = 0.
\end{equation}
\end{lemma}
\begin{proof}
By direct calculation, we obtain
\begin{equation*}
	\begin{split}
		&\int_\Om \Grad_h u_h \Grad_h (\Pi_h^V v-v)~dx \\
		&\qquad  = -\sum_E\int_E \Delta u_h(\Pi_h^V v - v)~dx + \int_{\partial E}(\Grad u_h\cdot \nu)(\Pi_h^V v - v)~dS(x).
	\end{split}
\end{equation*}
Now, since $u_h$ is linear on each element $\Delta u_h = 0$ and $\Grad u_h$ is constant. 
Moreover, since the normal vector $\nu$ is constant on each face of the element, we have that
\begin{equation*}
	\begin{split}
		&\int_\Om \Grad_h u_h \Grad_h (\Pi_h^V v-v)~dx \\
		&\qquad  = \sum_{E}\sum_{\Gamma \subset \partial E}\left(\int_{\Gamma }(\Pi_h^V v - v)~dS(x)\right)(\Grad u_h\cdot \nu) = 0,
	\end{split}
\end{equation*}
by definition of the interpolation operator $\Pi_h^V$. Hence, we have proved \eqref{eq:magic1}.

\end{proof}

\section{Numerical method and main result}\label{sec:num}

For a given timestep $\Dt>0$, we divide the time interval $[0,T]$ in 
terms of the points $t^k=k\Dt$, $k=0,\dots,M$, where we assume $M\Dt=T$. 
To discretize space, we let $\{E_{h}\}_{h}$ be a shape-regular 
family of tetrahedral meshes of $\Omega$, where $h$ is the maximal diameter.
It will be a standing assumption that $h$ and $\Delta t$ are 
related like $\Delta t = c h$ for some constant $c$. 
% By shape-regular we mean the existence of a constant 
% $\kappa > 0$ such that every $E \in E_{h}$ contains a ball of radius 
% $\lambda_{E} \geq \frac{h_{E}}{\kappa}$, where $h_{E}$ is the diameter of $E$. 
We also let $\Gamma_{h}$ denote the set of faces in $E_{h}$. 

The functions that are piecewise constant with 
respect to the elements of a mesh $E_{h}$ are 
denoted by $Q_h(\Om)$ and by $V_h(\Om)$ we denote
the Crouzeix--Raviart finite element space \eqref{def:CR} formed on $E_h$.
To incorporate the boundary condition, we let the 
degrees of freedom of $V_h(\Om)$ vanish at the boundary: 
$$
\int_\Gamma v_h~dS(x) = 0
\quad \forall \Gamma \in \Gamma_h\cap \partial \Om, 
\quad \forall v_h \in V_h(\Om).
$$

We will need some additional notation 
to accommodate discontinuous Galerkin discretization. 
Related to the boundary $\partial E$ of an element $E$, we write $f_{+}$ 
for the trace of the function $f$ taken within the element $E$, 
and $f_{-}$ for the trace of $f$ from the outside. 
Related to a face $\Gamma$ shared between two 
elements $E_{-}$ and $E_{+}$, we will write $f_{+}$ for 
the trace of $f$ within $E_{+}$, and $f_{-}$ for the trace
of $f$ within $E_{-}$. Here $E_{-}$ and $E_{+}$ are 
defined such that $\nu$ points from $E_{-}$ to $E_{+}$, where $\nu$ is 
fixed as one of the two possible 
normal components on  $\Gamma$.
The jump of $f$ across the face $\Gamma$ is denoted $\jump{f}_{\Gamma}= f_{+} - f_{-}$.

To pose the method, and in the convergence analysis, we shall need the canonical interpolation operators 
\eqref{def:interpolant}. In fact, we shall need the operators $\Pi_h^Q$ and
$\Pi_h^\mathcal{N}$ to such an extent that we introduce the convenient notation 
\begin{equation}\label{def:tilde}
	\widetilde v = \Pi_h^\mathcal{N}v, \qquad \widehat \phi = \Pi_h^Q \phi.
\end{equation} 

To discretize the convective operator $\Div(\vr u)$ in the continuity equation \eqref{eq:cont}, 
we will utilize a standard upwind method in 
the degrees of freedom of $u_h$. The upwind discretization 
is defined as follows
\begin{equation}\label{def:up}
	\begin{split}
		\up(\vr u)|_{\Gamma} 
		 &= \vr_- \left(\frac{1}{|\Gamma|}\int_\Gamma u_h \cdot \nu ~ dS(x)\right)^+ 
		\!\!+ \vr_+ \left(\frac{1}{|\Gamma|}\int_\Gamma u_h \cdot \nu ~ dS(x)\right)^-,\\
		 &= \vr_- (\widetilde u_h \cdot \nu)^+ + \vr_+(\widetilde u_h \cdot \nu)^-, \quad \forall \Gamma \in \Gamma_h,
	\end{split}
\end{equation}
where we have used the notation \eqref{def:tilde}.

For the convective operator $\Div(\vr u \otimes u)$
in the momentum equation \eqref{eq:moment} we will use the following 
Lax-Friedrich type upwind flux
\begin{equation}\label{def:upp}
	\begin{split}
		\up(\vr u \otimes \ov u) &= \up^+(\vr u)\ov u_+ + \up^-(\vr u)\ov u_-.
	\end{split}
\end{equation}
Observe that this operator is posed for the average value of $u_h$ over each element.  
This is non-standard and has to the author's knowledge not been studied 
previously. We are now ready to pose the method.
\begin{table}
	\begin{center}
	\begin{tabular}{l c l}
%		\hline
		\\
		$E_h$& - & the mesh. \\
		$E$ & - & an element in the mesh. \\
		$\partial E$ & - &  the boundary of $E$. \\
		$\Gamma$ & - & a face in the mesh. \\
		$\Gamma_h$ & - & all faces in the mesh. \\
		$Q_h(\Om)$ & - & the space of piece constant scalars on $E_h$.\\
		$V_h(\Om)$ & - & the Crouzeix-Raviart vector space on $E_h$.\\
		$\mathcal{N}_h(\Om)$ & - & the div conforming N{\'e}d{\'e}lec space \\
		&&of first order and kind. \\
		$\Pi_h^Q$ & - & the $L^2$ projection operator onto $Q_h$. \\
		$\Pi_h^V$ & - & the canonical interpolation operator onto $V_h$.\\
		$\Pi_h^N$ & - & the canonical interpolation operator onto $\mathcal{N}_h$.\\
		$\widehat f$ & - &  $\Pi_h^Q f$ (the piecewise constant projection).\\
		$\widetilde v$ & - & $\Pi_h^N v$ (the N{\'e}d{\'e}lec interpolation).\\
		$f^+$ & - &  $\max\{f,0\}$.\\
		$f^-$ & - &  $\min\{f,0\}$.\\
		$f_+|_{\partial E}$ & - & the 
					trace of $f$ taken from within $E$. \\
		$f_-|_{\partial E}$ & - & the 
						trace of $f$ taken from outside $E$. \\
		$f_+|_{\Gamma}$ & - & the 
					trace of $f$ taken against the normal vector $\nu$. \\
		$f_-|_{\Gamma}$ & - & the 
					trace of $f$ taken in the direction of $\nu$. \\
		$\jump{f}_\Gamma$ & - & $f_+ - f_-$.\\
		$\up(\vr u)|_{\partial E}$ & - & $\vr_+(\widetilde u_h \cdot \nu)^+ + \vr_-(\widetilde u_h \cdot \nu)^-$.\\
		$\up(\vr u)|_{\Gamma}$ & - & $\vr_-(\widetilde u_h \cdot \nu)^+ + \vr_+(\widetilde u_h \cdot \nu)^-$. \\
		$\up(\vr u \otimes \widehat u)_{\partial E}$ & - & $\up^+(\vr u)\ov u_+ + \up^-(\vr u)\ov u_-$.\\
		$\up(\vr u \otimes \widehat u)_{\Gamma}$ & - & $\up^-(\vr u)\ov u_+ + \up^+(\vr u)\ov u_-$.\\
		\\
%		\hline
	\end{tabular}
	\end{center}
	\caption{Notation}
%	\label{label}
\end{table}

\begin{definition}[Numerical method]\label{def:scheme}
Let $\vr_0 \in L^{\gamma+1}(\Om)$ and $m_0 \in L^\frac{3}{2}(\Om)$ be given initial data and assume that 
$T>0$ is a given finite final time. 
Define the numerical initial data
\begin{equation}\label{def:init}
	\vr_h^0 = \Pi_h^Q (\vr_0 + \kappa h), \qquad u_h^0 = \Pi_h^V\left[\frac{m_0}{\vr_0 + \kappa h}\right],
\end{equation}
where $\kappa$ is a small positive number. Determine sequentially 
	\begin{equation*}
		(\vr_h^k, u_h^k) \in Q_h(\Om)\times V_h(\Om), \quad k=1, \ldots, M,
	\end{equation*}
	satisfying, for all $q_h \in Q_h(\Om)$,
	\begin{equation}\label{fem:cont}
		\begin{split}
			\int_\Om \frac{\vr^k_h - \vr_h^{k-1}}{\Delta t} q_h~dx 
			&- \sum_\Gamma \int_{\Gamma} \up^k(\vr u)\jump{q_h}_\Gamma~dS(x) \\
			&\qquad + h^{1-\epsilon}\sum_\Gamma \int_{\Gamma}\jump{\vr^k_h}_\Gamma\jump{q_h}_\Gamma~dS(x) = 0,
		\end{split}
	\end{equation}
	and for all $v_h \in V_h(\Om)$,
	\begin{equation}\label{fem:moment}
		\begin{split}
			&\int_\Om \frac{\vr^k_h\ov u^k_h - \vr_h^{k-1}\ov u^{k-1}_h}{\Delta t} v_h~dx - \sum_{ \Gamma} \int_{\Gamma} \up^k(\vr u\otimes \ov u)\jump{\ov v}_\Gamma ~dS(x) \\
			&\qquad \qquad + \int_\Om \Grad_h u^k_h \Grad_h v_h - p(\vr_h^k)\Div v_h~dx \\
			&\qquad \qquad \qquad+ h^{1-\epsilon}\sum_E \int_{\partial E} \left(\frac{\ov u^k_- + \ov u^k_+}2\right) \jump{\vr^k_h}\jump{\ov v_h}_\Gamma~dS(x) = 0,
			\end{split}
	\end{equation}	
	where $\eps > \frac{1}{6}$ should be chosen as small as possible.
\end{definition}

For the purpose of analysis, we will need to extend the pointwise-in-time numerical solution $(\vr_h^k, u_h^k)$, $k =0, \ldots, M$, 
to a piecewise constant in time. For this purpose, we will use the following definition
\begin{equation}\label{def:schemeII}
	(\vr_h, u_h)(t, \cdot) = (\vr_h^k, u_h^k), \quad t \in [t^{k}, t^{k+1}), \quad k=0, \ldots, M.
\end{equation}

\begin{remark}
The terms involving $h^{1-\eps}$ in \eqref{fem:cont} - \eqref{fem:moment} 
are needed for purely technical reasons. In particular, they 
are not needed to obtain stability of the method or to prove convergence 
of the continuity method \eqref{fem:cont}.
See Section \ref{sec:hodge}
and \cite{Gallouet2} for more on why they are needed. 

\end{remark}

\subsection{The method is well-defined}

Since the numerical method is nonlinear and implicit,
it is not trivial that it is actually well-defined. 
In addition, the transport operators in the momentum
equation is posed for the element average velocity 
and hence does not provide a full set of equations
in themselves. In fact, it is only due to the diffusion that
the system has a sufficient number of equations 
for the degrees of freedom of $u_h$.
We shall prove the existence of a numerical solution 
through a topological degree argument. 
The proof is very similar to that of \cite{Gallouet:2006lr, Karlsen1} and 
is for this reason deferred to the appendix.

\begin{proposition}\label{pro:defined}
For each fixed $h > 0$, there exists 
a solution 
\begin{equation*}
	(\vr_h^k, u_h^k) \in Q_h(\Om)\times V_h(\Om),\quad \vr^k_h(\cdot) > 0, \quad k=1, \ldots, M,
\end{equation*}
to the numerical method posed in Definition \ref{def:scheme}.
\end{proposition}

In the upcoming analysis, we will need that the density solution is positive. 
For this purpose, we recall the following result from \cite{Karlsen1} (see also \cite{Gallouet:2006lr}):
\begin{lemma} 
\label{lemma:vrho-props}
Fix any $k=1,\dots,M$ and suppose $\varrho^{k-1}_{h} 
\in Q_{h}(\Om)$, $u^k_{h} \in V_{h}(\Om)$ are given bounded functions. 
Then the solution $\varrho^{k}_{h} \in Q_{h}(\Om)$ of the discontinuous 
Galerkin method \eqref{fem:cont} satisfies
$$
\min_{x \in \Omega}\varrho_{h}^k 
\geq \min_{x \in \Omega}\varrho_{h}^{k-1}\left(\frac{1}{1 + 
\Delta t \|\Div_h u^k_{h}\|_{L^\infty(\Omega)}}\right).
$$
Consequently, if $\varrho^{k-1}_{h}(\cdot)>0$, then $\varrho^{k}_{h}(\cdot)>0$.
\end{lemma}

\subsection{Main result}
Our main result is that the numerical method converges to
a weak solution of the compressible Navier-Stokes equations \eqref{eq:cont} - \eqref{eq:init}.

\begin{theorem}\label{thm:main}
Suppose $\gamma >3$, $T>0$ is a given finite final time, and that the initial data $(\vr_0, m_0)$ satisfies
\begin{equation*}
	\int_\Om \frac{m_0^2}{2\vr_0} + \frac{1}{\gamma - 1}p(\vr_0)~dx \leq C, \quad \vr_0 \in L^\infty(0,T;L^{\gamma+1}(\Om)).
\end{equation*}
Let $\{(\vr_h, u_h)\}_{h>0}$ be a sequence of numerical solutions constructed according 
to Definition \ref{def:scheme} and \eqref{def:schemeII} with $\Delta t = c h$. Along a subsequence as $h \rightarrow 0$, 
\begin{align*}
	u_h &\weak u \text{ in $L^2(0,T;L^6(\Om))$},\\
	\Grad_h u_h &\weak \Grad u \text{ in $L^2(0,T;L^2(\Om))$}\\
	\vr_h  u_h, ~ \vr_h\widetilde u_h &\weak \vr u \text{ in } L^2(0,T;L^{\frac{6\gamma}{3+\gamma}}(\Om)), \nonumber\\
	\vr_h \widehat u_h &\weak \vr u \text{ in } C(0,T;L^{\frac{2\gamma}{\gamma+1}}(\Om)). \\
	\vr_h \widetilde u_h \otimes \widehat u_h &\weak \vr u \otimes u \text{ in }L^2(0,T;L^{\frac{3\gamma}{3+\gamma}}(\Om)), \nonumber \\
	\vr_h \rightarrow \vr \text{ a.e}& \text{ and in $L_\text{loc}^p((0,T)\times \Om)$, $p < \gamma + 1$},
\end{align*}
where $(\vr, u)$ is a weak solution of the isentropic compressible Navier-Stokes equations \eqref{eq:cont} - \eqref{eq:moment} in the sense of Definition \ref{def:weak}.
\end{theorem}

The proof of Theorem \ref{thm:main} will be developed in the remaining 
sections of this paper. The final conclusion will come in Section \ref{sec:proof}.

\section{Energy and stability}\label{sec:stability}
In this section we will prove that our method 
is stable and satisfies a numerical analogy 
of the continuous energy inequality \eqref{eq:cont-energy}. 
Both the stability estimate and large parts of 
the subsequent convergence analysis relies 
on a renormalized type identity derived 
from the continuity scheme \eqref{fem:cont}.
The proof of this identity can be found in \cite{Karlsen1, Karlsen2}. 
We shall utilize the following form:
\begin{lemma}\label{lem:renorm}
Let $(\vr^h, u^h)$ solve the continuity scheme \eqref{fem:cont}. Then, 
for any $B \in C^2(\R_+)$ with $B'' \geq 0$, there holds
\begin{equation*}\label{eq:renorm}
	\begin{split}
	&		\int_\Om \frac{B(\vr^k_h)-B(\vr_h^{k-1})}{\Delta t}~dx +(\Delta t)^{-1}\int_\Om B''(\vr_\ddagger)(\vr_h^{k} - \vr_h^{k-1})^2~ dx\\
	&\qquad \qquad+ \sum_\Gamma \int_\Gamma B''(\vr_\dagger)\jump{\vr^k_h}^2\left(|\widetilde{u^k_h}\cdot\nu|+ h^{1-\eps}\right)~dS(x) \\
	&\qquad \qquad
	 \leq -\int_\Om (\vr B'(\vr) - B(\vr))\Div u_h~dx,
	\end{split}
\end{equation*}
where $\vr_\ddagger$ and $\vr_\dagger$ are some numbers in the range $[\vr^{k-1}, \vr^{k}]$ and 
$[\vr^k_-, \vr^k_+ ]$, respectively.
\end{lemma}

We now prove our main stability result.
%The following estimate is the numerical version of the energy inequality  \eqref{eq:cont-energy}. 

\begin{proposition}\label{prop:energy}
For given $\Delta t$, $h>0$, let $(u_h^k, \vr_h^k)$, $k=0,\ldots,M$ 
be the numerical approximation of \eqref{eq:cont}-\eqref{eq:moment} 
in the sense of Definition \ref{def:scheme}. Then,
\begin{equation}\label{eq:energy}
	\begin{split}
		&\max_{m=1, \ldots, M}\int_\Om \frac{\vr_h^m |\ov u^m_h|^2}{2} + \frac{1}{\gamma -1}p(\vr^m_h)~dx  \\
		&\qquad\qquad\quad  + \Delta t \sum_{k=1}^M\int_\Om |\Grad_h u_h^k|^2~dx + \Delta t\sum_{i=1}^5\sum_{k=1}^M D^k_i \\
	 	&\qquad \leq \int_\Om \frac{\vr_h^0 |\ov u^0_h|^2}{2} + \frac{1}{\gamma -1}p(\vr^0_h)~dx,
	\end{split}
\end{equation}
where the numerical diffusion terms are given by 
\begin{equation*}
	\begin{split}
		D^k_1 &= \sum_\Gamma \int_{\Gamma}P''(\vr_\dagger)\jump{\vr^k_h}_\Gamma^2|u_h^k\cdot \nu|~dS(x), \\
		D^k_2 &=\sum_\Gamma \int_{\Gamma}\left|\up^k(\vr u)\right|\jump{\ov u^k_h}_\Gamma^2~dS(x), \\
		D^k_3 &= h^{1-\eps}\sum_\Gamma \int_0^T\int_\Gamma \vr_{\dagger\dagger} \jump{\vr_h}_\Gamma^2~dS(x)dt,\\
		D^k_4 &=  (\Delta t)^{-1}\int_\Om P''(\vr_\ddagger)(\vr^k_h - \vr_h^{k-1})^2~dx, \\
		D^k_5 &=  (\Delta t)^{-1}\int_\Om \vr_h^{k-1}|u^k_h - u_h^{k-1}|^2~dx.
	\end{split}
\end{equation*}
\end{proposition}

\begin{proof}
	Let $v_h = u^k_h$ in the momentum scheme \eqref{fem:moment}, to obtain
	\begin{equation}\label{en:beg}
		\begin{split}
			&\int_\Om \frac{\vr_h^k\ov u_h^k - \vr_h^{k-1}\ov u_h^{k-1}}{\Delta t} u_h^k~dx + \int_\Om |\Grad_h u_h^k|^2~dx\\
			& \qquad \quad = - \sum_{E} \int_{\partial E} \up^k(\vr u\otimes \ov u)\ov u^k_+~dS(x) + \int_\Om p(\vr_h^k)\Div u_h^k~dx \\
			&\qquad \qquad + h^{1-\epsilon}\sum_E \int_{\partial E} \left(\frac{\ov u^k_- + \ov u^k_+}2\right) \jump{\vr^k_h}\ov u^k_h~dS(x).
		\end{split}
	\end{equation}
	From Lemma \ref{lem:renorm} with $B(z) = \frac{1}{\gamma-1}p(z)$, we see that the pressure term
	\begin{equation}\label{en:pres}
		\begin{split}
			&\int_\Om p(\vr_h^k)\Div u_h^k~dx \\
			&\qquad \geq - \int_\Om \frac{1}{\gamma-1}\frac{p(\vr_h^k) - p(\vr_h^{k-1})}{\Delta t}~dx - D_1^k-D_3^k- D_4^k.
		\end{split}
	\end{equation}
	
 	Next, we turn our attention to the first term after the equality in \eqref{en:beg}.
	From the definition of $\up(\vr u\otimes \ov u)$, we have that
	\begin{align}\label{eq:tr1}
			&\sum_{E} \int_{\partial E} \up^k(\vr u\otimes \ov u)\ov u^k_+~dS(x) \nonumber\\
			&\quad\qquad = \sum_{E} \int_{\partial E}\left(\up^+(\vr u)\ov u^k_+ + \up^-(\vr u)\ov u^k_-\right)\ov u^k_+~dS(x) \nonumber \\
			&\quad\qquad = \sum_{E} \int_{\partial E}\up^k(\vr u)\frac{(\ov u^k_+)^2}{2} - \up^-(\vr u)\frac{(\ov u^k_+)^2}{2} \\
			&\quad\qquad\qquad \qquad + \up^+(\vr u)\frac{(\ov u^k_+)^2}{2} + \up^-(\vr u)\ov u^k_- \ov u^k_+ dS(x). \nonumber
	\end{align}	
	 By setting $q_h = (1/2)(\ov u_h^k)^2$ in the continuity scheme \eqref{fem:cont}, we see that the first 
	term after the equality in \eqref{eq:tr1} appears 
	\begin{align}\label{eq:tr2}
				&\int_\Om \frac{\vr_h^k - \vr_h^{k-1}}{\Delta t} \frac{|\ov u^k_h|^2}{2}~dx \nonumber\\
				&\qquad = -\sum_{E} \int_{\partial E}\up^k(\vr u)\frac{(\ov u^k_+)^2}{2}~dS(x) %\nonumber\\
				 + h^{1-\eps}\sum_E\int_{\partial E}\jump{\vr_h^k}\frac{(\ov u^k_+)^2}{2}~dS(x) \nonumber\\
				&\qquad = -\sum_{E} \int_{\partial E}\up^k(\vr u)\frac{(\ov u^k_+)^2}{2}~dS(x) \\
				&\qquad \qquad\qquad
				 + h^{1-\eps}\sum_E\int_{\partial E}\jump{\vr_h^k}\left(\frac{\ov u^k_+ + \ov u^k_-}{2}\right)\ov u^k_+~dS(x). \nonumber
	\end{align}
	To see the contribution of the second term in \eqref{eq:tr1}, we first recall that
	\begin{equation*}
		\left.\up^+(\vr u)\right|_{\partial E^+} = -\left.\up^-(\vr u)\right|_{\partial E^-}, \qquad \partial E^+ \cap \partial E^- = \Gamma,
	\end{equation*}
	 since the normal vector has opposite signs.
	Using this, we obtain 
	\begin{equation*}
		\begin{split}
			&\sum_{E}\int_{\partial E}\up^+(\vr u)\frac{(\ov u^k_+)^2}{2}- \up^-(\vr u)\frac{(\ov u^k_+)^2}{2}+ \up^-(\vr u)\ov u^k_- \ov u^k_+ ~dS(x) \\
			&= \sum_\Gamma \int_\Gamma \up^+(\vr u)\frac{(\ov u_+^k)^2}{2} - \up^-(\vr u)\frac{(\ov u^k_-)^2}{2} - \up^-(\vr u)\frac{(\ov u^k_+)^2}{2}\\
			&\qquad \qquad  - \up^+(\vr u)\frac{(\ov u^k_-)^2}{2} + \up^-(\vr u)\ov u^k_- \ov u^k_+ - \up^+(\vr u) \ov u^k_+ \ov u^k_-~dS(x)\\
			&= \sum_\Gamma \int_\Gamma |\up^k(\vr u)|\left(\frac{(\ov u^k_+)^2}{2} + \frac{(\ov u^k_-)^2}{2} - \ov u^k_+ \ov u^k_-  \right)~dS(x) \equiv D^k_2.
		\end{split}
	\end{equation*}
	By applying this together with \eqref{eq:tr2} in \eqref{eq:tr1}, we discover
	\begin{equation}\label{en:last}
		\begin{split}
				&-\sum_{E} \int_{\partial E} \up^k(\vr u\otimes \ov u)\ov u^k_+~dS(x) \\
				&\qquad + h^{1-\eps}\sum_E\int_{\partial E}\jump{\vr_h^k}\left(\frac{\ov u^k_+ + \ov u^k_-}{2}\right)\ov u^k_+~dS(x) \\
				&\qquad = \int_\Om \frac{\vr_h^k - \vr_h^{k-1}}{\Delta t} \frac{|\ov u^k_h|^2}{2}~dx - D^k_2.
		\end{split}
	\end{equation}
	
	Now, setting \eqref{en:last} and \eqref{en:pres} into \eqref{en:beg} reveals
	\begin{equation*}
		\begin{split}
			&\int_\Om \frac{\vr_h^k\ov u_h^k - \vr_h^{k-1}\ov u_h^{k-1}}{\Delta t} u_h^k - \frac{\vr_h^k - \vr_h^{k-1}}{\Delta t} \frac{|\ov u^k_h|^2}{2} 
			+ \frac{1}{\gamma-1}\frac{p(\vr_h^k) - p(\vr_h^{k-1})}{\Delta t}~dx \\
			&\quad + \int_\Om |\Grad_h u^k_h|^2 ~dx + D_1^k + D_2^k + D_3^k + D_4^k \leq 0.
		\end{split}
	\end{equation*}
	A simple calculation gives
	\begin{equation*}
		\begin{split}
			&\int_\Om \frac{\vr_h^k\ov u_h^k - \vr_h^{k-1}\ov u_h^{k-1}}{\Delta t} u_h^k - \frac{\vr_h^k - \vr_h^{k-1}}{\Delta t}\frac{|\ov u^k_h|^2}{2}~dx \\
			&\qquad = \int_\Om \frac{\vr_h^k |u_h^k|^2 - \vr_h^{k-1}|u_h^{k-1}|^2}{2\Delta t} + \frac{\vr_h^{k-1}|u_h^k - u_h^{k-1}|^2}{\Delta t}~ dx \\
			& \qquad \equiv \int_\Om\frac{\vr_h^k |u_h^k|^2 - \vr_h^{k-1}|u_h^{k-1}|^2}{2\Delta t}~dx + D_5^k.
		\end{split}
	\end{equation*}
	Consequently, by combining the two previous inequalities, 
	\begin{equation*}
		\begin{split}
			&\int_\Om\frac{\vr_h^k |u_h^k|^2 - \vr_h^{k-1}|u_h^{k-1}|^2}{2\Delta t}~dx
			+ \frac{p(\vr_h^k) - p(\vr_h^{k-1})}{(\gamma - 1)\Delta t}~dx \\
			&\quad + \int_\Om |\Grad_h u^k_h|^2~dx + \sum_{i=1}^5 D_i^k \leq 0.
		\end{split}
	\end{equation*}
	We conclude by multiplying with $\Delta t$ and summing over $k=1, \ldots, M$.
\end{proof}

Observe that the energy estimate does not provide $L^\infty$ control in time on $\vr_h |u_h|^2$.
Instead, we only gain this control on the projection $\vr_h |\widehat u_h|^2$. 
The following corollary is an immediate consequence of 
the energy estimate (Proposition \ref{prop:energy}) 
and the H\"older inequality (cf.~\cite{Feireisl}).

\begin{corollary}\label{cor:energy}
Under the conditions of Proposition \ref{prop:energy},
\begin{align*}
	\vr_h \inb L^\infty(0,T; L^\gamma(\Om)), \quad p(\vr_h) \inb L^\infty(0,T;L^1(\Om))\\
	u_h \inb L^2(0,T;L^6(\Om)), \quad\Grad_h u_h \inb L^2(0,T;L^2(\Om)),
\end{align*}
\begin{equation*}
	\vr_h \widehat u_h \inb L^\infty(0,T;L^{m_\infty}(\Om)), \quad \vr_h u_h  \inb L^2(0,T;L^{m_2}(\Om)),
\end{equation*}
\begin{equation*}
	\vr_h |\widehat u_h|^2 \in L^\infty(0,T;L^1(\Om)), \quad \vr_h |u_h|^2 \inb L^2(0,T;L^{c_2}(\Om)),
\end{equation*}
where the exponents are given by (since $\gamma > 3$)
\begin{align*}
	m_\infty = \frac{2\gamma}{\gamma + 1} > \frac{3}{2}, \qquad m_2 = \frac{6\gamma}{3 + \gamma} > 3,
	\qquad c_2 = \frac{3\gamma}{3 + \gamma} > \frac{3}{2}.
\end{align*}
\end{corollary}

\section{Estimates on the numerical operators}\label{sec:operators}
To prove convergence of the numerical method, our strategy 
will be to adapt the continuous existence theory 
to the numerical setting. We will succeed with
this by controlling the weak error of the numerical 
operators relative to their continuous counterparts. 
The purpose of this section, is to derive the needed error estimates. 
% It should be noted that proving such estimates for 
% any given method is difficult. This is actually 
% the only reason for our lower order discretization 
% of the convective part in the momentum equation.

For notational convenience, let us define
\begin{equation*}
	D_t^h f = \frac{f(t) - f(t-\Delta t)}{\Delta t},
\end{equation*}
and observe that this satisfies 
\begin{equation*}
	D_t^h \vr_h(t) = \frac{\vr^k_h - \vr_h^{k-1}}{\Delta t} \quad t \in [t^{k}, t^{k+1}).
\end{equation*}

\subsection{The convective discretizations}
We begin by deriving identities for the distributional error 
of the numerical convection operators.

\begin{lemma}\label{lem:transport}
Fix 
two functions $\phi \in C_0^\infty(\Om)$, $v \in [C_0^\infty(\Om)]^d$. The numerical transport operators 
in \eqref{fem:cont} and \eqref{fem:moment} satisfies the following identities
	\begin{align}
				\sum_\Gamma \int_{\Gamma} \up(\vr u)~\jump{\Pi_h^Q \phi}_\Gamma~dS(x) 
				 &= \int_\Om \vr_h \widetilde u_h \Grad \phi~dx + P_1(\phi), \label{eq:cont-trans} \\
			\sum_{\Gamma} \int_{\Gamma} \up(\vr u\otimes \ov u)\jump{\widehat{\Pi_h^V v}}_\Gamma~dS(x) 
			& = \int_\Om \vr_h \widetilde u_h\otimes \ov u_h:\Grad  v~dx + \sum_{i=2}^4 P_i(v) \label{eq:moment-trans}, 
	\end{align}
	where the error functionals $P_i$, $i=1,\ldots, 4$, are 
	\begin{align*}
		P_1(\phi) & = \sum_E \int_{\partial E} \jump{\vr_h}_{\partial E}(\widetilde u_h \cdot \nu)^-(\Pi_h^Q\phi - \phi)~dS(x), \\
		P_2(v) & = \sum_E \int_{\partial E} \jump{\vr_h}_{\partial E}(\widetilde u_h \cdot \nu)^-\ov u_h \left(\widehat{\Pi_h^V v} - v\right)~dS(x),\\
		P_3(v) & = \sum_E \int_{\partial E}\up_-(\vr u)\jump{\ov u_h}_{\partial E}\left(\widehat{\Pi_h^V v} - v\right)~dS(x), \\
		P_4(v) &= \sum_E \int_{E}\vr_h \Div u_h \ov u_h \left(\Pi_h^V v - v\right)~dx.
	\end{align*}
\end{lemma}

\begin{proof}
Using the continuity of $\up(\vr u)$ and $\phi$ across edges, we calculate
\begin{equation*}
	\begin{split}
		&\sum_\Gamma \int_{\Gamma} \up(\vr u)~\jump{\Pi_h^Q \phi}_\Gamma~dS(x) \\
		&\quad = -\sum_E \int_{\partial E} \up(\vr u)~\Pi_h^Q \phi~dS(x) \\
		&\quad = -\sum_E \int_{\partial E} \up(\vr u)~\left(\Pi_h^Q \phi - \phi \right)~dS(x) \\
		&\quad = -\sum_E \int_{\partial E} \left(\vr_+(\widetilde u_h \cdot \nu)^+ + \vr_- (\widetilde u_h \cdot \nu)^-\right)~\left(\Pi_h^Q \phi - \phi \right)~dS(x),
	\end{split}
\end{equation*}
where the last identity is the definition of $\up(\vr u)$.
Next, we add and subtract to deduce
\begin{equation*}
\begin{split}
		&\sum_\Gamma \int_{\Gamma} \up(\vr u)~\jump{\Pi_h^Q \phi}_\Gamma~dS(x) \\
		&\quad = -\sum_E \int_{\partial E} \vr_+(\widetilde u_h \cdot \nu)\left(\Pi_h^Q \phi - \phi \right) \\
		&\qquad \qquad \qquad	+ (\vr_- - \vr_+) (\widetilde u_h \cdot \nu)^-\left(\Pi_h^Q \phi - \phi \right)~dS(x) \\
		&\quad = -\sum_E \int_E \Div\left(\vr_h \widetilde u_h\left(\Pi_h^Q \phi - \phi \right)\right)~dS(x) + P_1(\phi) \\
		&\quad = \int_\Om \vr_h  \widetilde u_h \Grad \phi~dx - \sum_E \int_E \vr_h \Div u_h\left(\Pi_h^Q \phi - \phi \right)~dx + P_1(\phi).
	\end{split}
\end{equation*}	
We conclude \eqref{eq:cont-trans} by recalling that  $\Div_h u_h$ is constant on each element and hence the second term is zero.

To derive \eqref{eq:moment-trans}, we apply the definition \eqref{def:upp} of $\up(\vr u \otimes \ov u)$ and 
add to subtract to obtain
\begin{equation*}
	\begin{split}
		&\sum_{\Gamma} \int_{\Gamma} \up(\vr u\otimes \ov u)\jump{\widehat{\Pi_h^V v}}_\Gamma~dS(x)  \\
		&\quad = -\sum_{ E} \int_{\partial E} \up(\vr u\otimes \ov u)~\left(\widehat{\Pi_h^V v}- v\right)~dS(x) \\
		&\quad = -\sum_{ E} \int_{\partial E} \left(\up_+(\vr u)\ov u_+ + \up_-(\vr u)\ov u_-\right)~\left(\widehat{\Pi_h^V v}- v\right)~dS(x) \\
		&\quad = -\sum_{ E} \int_{\partial E} \left(\up(\vr u)\ov u_+ + \up_-(\vr u)(\ov u_-- \ov u_+)\right)~\left(\widehat{\Pi_h^V v}- v\right)~dS(x) \\
		&\quad = -\sum_{ E} \int_{\partial E} \up(\vr u)\ov u_+\left(\widehat{\Pi_h^V v}- v\right)~dS(x) + P_3(v).
\end{split}
\end{equation*}
To proceed, we apply the definition of $\up(\vr u)$ \eqref{def:up}
and add and subtract to obtain
\begin{equation*}
	\begin{split}
		&\sum_{\Gamma} \int_{\Gamma} \up(\vr u\otimes \ov u)\jump{\widehat{\Pi_h^V v}}_\Gamma~dS(x)  \\
		&\quad = -\sum_{ E} \int_{\partial E} \left(\vr_+ (\widetilde u_h\cdot \nu)^+ + \vr_-(\widetilde u_h \cdot \nu)^-\right)\ov u_+\left(\widehat{\Pi_h^V v}- v\right)~dS(x) + P_3(v) \\
		&\quad = -\sum_{ E} \int_{\partial E} \vr_+ (\widetilde u_h\cdot \nu)\ov u_+\left(\widehat{\Pi_h^V v}- v\right) \\
		&\qquad \qquad \qquad + (\vr_--\vr_+)(\widetilde u_h \cdot \nu)^-\ov u_+\left(\widehat{\Pi_h^V v}- v\right)~dS(x) + P_3(v) \\
		&\quad = -\sum_{ E} \int_{\partial E} \vr_+ (\widetilde u_h\cdot \nu)\ov u_+\left(\widehat{\Pi_h^V v}- v\right)~dS(x) + P_2(v) + P_3(v) \\
	\end{split}
	\end{equation*}
We can then apply the divergence theorem to the first term to obtain
	\begin{equation*}
		\begin{split}
		&\sum_{\Gamma} \int_{\Gamma} \up(\vr u\otimes \ov u)\jump{\widehat{\Pi_h^V v}}_\Gamma~dS(x)  \\
		&\quad = \int_\Om \vr_h \widetilde u_h\otimes \ov u_h:\Grad  v~dx -\sum_E \int_E\vr_h \Div u_h \ov u_h \left(\Pi_h^V v - v\right)~dx\\
		&\quad \qquad \qquad \quad + P_2(v) + P_3(v) \\
		&\quad = \int_\Om \vr_h \widetilde u_h\otimes \ov u_h:\Grad  v~dx +P_4(v)+ P_2(v) + P_3(v),
	\end{split}
\end{equation*}	
which is \eqref{eq:moment-trans}.
\end{proof}

Identities like \eqref{eq:cont-trans} and \eqref{eq:moment-trans} 
can be derived for any numerical method. The difficult part is to 
control the error terms, in our case is given by $P_i$, $i=1, \ldots, 4$.
The following proposition provides sufficient control on the error terms.
Note in particular the integrability required of the test-functions. 

\begin{proposition}\label{pro:transport}
Let $(\vr_h^k, u_h^k)$, $k=1, \ldots, M$ be the numerical solution obtained using 
the scheme \eqref{fem:cont} - \eqref{fem:moment}. Let $(\vr_h, u_h)$ be 
the piecewise constant extension of $(\vr_h^k, u_h^k)$, $k=1, \ldots, M$ in time to all of $[0, M\Delta t]$ 
(i.e~\eqref{def:schemeII}). 
Then, the $P_i$, $i=1, \ldots, 4$ in Proposition \ref{lem:transport} are also 
piecewise constant in time and there exists a constant $C > 0$, independent of $h$ and $\Delta t$, 
such that
	\begin{align}
		\int_0^T |P_1(\phi)|~dt &\leq h^{\frac{1}{2}-\min\{3\frac{4-\gamma}{4\gamma},0\}}C\|\Grad \phi\|_{L^4(0,T;L^\frac{12}{5}(\Om)}, \label{eq:p1}\\
		\int_0^T|P_2(v)|~dt &\leq h^{\frac{1}{2}-\min\{3\frac{4-\gamma}{4\gamma},0\}}C\|\Grad v\|_{L^\infty(0,T;L^3(\Om))}, \label{eq:p2}\\
		\int_0^T |P_3(v)|~dt &\leq h^\frac{1}{2}C\|\Grad v\|_{L^4(0,T;L^\frac{6\gamma}{5\gamma-6}(\Om))} \label{eq:p3}, \\
		\int_0^T |P_4(v)|~dt &\leq h^\frac{2(\gamma-9/4)}{\gamma}C\|\Grad v\|_{L^\infty(0,T;L^\gamma(\Om))} \label{eq:p4}.
	\end{align}
	In particular, for $\gamma > 3$, we have that
	\begin{equation}\label{eq:pcombined}
		\begin{split}
				\int_0^T |P_1(\phi)|~dt & \leq h^\frac{1}{4}C\|\Grad \phi\|_{L^4(0,T;L^\frac{12}{5}(\Om)},\\
				 \int_0^T |P_2(v)| + |P_3(v)|~dt &\leq Ch^\frac{1}{4}\|\Grad v\|_{L^\infty(0,T;L^\gamma(\Om))}.
		\end{split}
	\end{equation}

\end{proposition}

\begin{proof}
We will prove one inequality at the time. 

\vspace{0.2cm}
\noindent
\textit{1. Bound on $P^h_1$:} An application of the Cauchy-Schwartz inequality yields
\begin{equation}\label{eq:p1start}
	\begin{split}
		\int_0^T|P_1^h|~dt & := \int_0^T\left|\sum_E \int_{\partial E} \jump{\vr_h}_{\partial E}(u_h \cdot \nu)^-(\Pi_h^Q\phi - \phi)~dS(x)\right|~dt \\
		&\leq \left(\int_0^T\sum_E \int_{\partial E}\jump{\vr_h}_{\partial E}^2 |u_h \cdot \nu|~dS(x)dt\right)^\frac{1}{2}\\
		& \qquad \qquad \times\left(\int_0^T\sum_E \int_{\partial E} |u_h|(\Pi_h^Q\phi - \phi)^2~dS(x)dt\right)^\frac{1}{2} \\
		&:= \sqrt{I_1}\times \sqrt{I_2}.
	\end{split}
\end{equation}	
By setting $B(z) = \frac{1}{2}z^2$ in Lemma \ref{lem:renorm} and integrating in time, we obtain 
\begin{equation}\label{eq:p11}
	\begin{split}
		I_1 &\leq \frac{1}{2}\int_\Om (\vr_h)^2~dx + \int_0^T\int_\Om (\vr_h)^2 \Div_h u_h~dxdt \\
			&\leq  C\|\vr_h\|_{L^\infty(0,T;L^4(\Om))}^2 + \sqrt{T}\|\vr_h\|_{L^\infty(0,T;L^4(\Om)}^2\|\Div_h u_h\|_{L^2(0,T;L^2(\Om))}\\
			&\leq C\|\vr_h\|_{L^\infty(0,T;L^4(\Om))}^2(1+\sqrt{T}).
	\end{split}
\end{equation}
By applying the Trace Lemma \ref{lemma:toolbox} and the H\"older inequality, we deduce
\begin{equation}\label{eq:p12}
	\begin{split}
		I_2 &\leq h^{-1}\int_0^T\int_\Om |u_h|(\Pi_h^Q\phi - \phi)^2~dxdt \\
			&\leq h^{-1}C\int_0^T\|u_h\|_{L^6(\Om)}\\
			&\qquad \qquad \qquad \times \left(\|\Pi_h^Q \phi - \phi\|_{L^\frac{12}{5}(\Om)}^2 + h^2\|\Grad \phi\|_{L^4(0,T;L^\frac{12}{5}(\Om))}^2\right)dt \\
			& \leq hC\|u_h\|_{L^2(0,T;L^6(\Om))}\|\Grad \phi\|_{L^4(0,T;L^\frac{12}{5}(\Om))}^2,
	\end{split}
\end{equation}
where the last inequality is the error estimate on the interpolation error of $\Pi_h^Q$ 
in $L^p$ (see Lemma \ref{lemma:interpolationerror}).

Consequently, by applying \eqref{eq:p11} and \eqref{eq:p12} in \eqref{eq:p1start}, we see that
\begin{equation}\label{eq:p13}
	\begin{split}
		\int_0^T|P_1|~dt \leq h^\frac{1}{2}C\|\vr_h\|_{L^\infty(0,T;L^4(\Om))}\|\Grad \phi\|_{L^4(0,T;L^\frac{12}{5}(\Om))}.
	\end{split}
\end{equation}
From the standard inverse inequality (Lemma \ref{lemma:inverse}),  we have that
\begin{equation}\label{eq:p1inv}
	\|\vr_h\|_{L^\infty(0,T;L^4(\Om))} \leq h^{\max\{-3\frac{4-\gamma}{4\gamma},0\}}C\|\vr_h\|_{L^\infty(0,T;L^\gamma(\Om))},
\end{equation}
where the last inequality follows from $\gamma > 3$. Applying this in \eqref{eq:p13} finally yields
\begin{equation*}
	\int_0^T|P_1|~dt \leq h^{\frac{1}{2}-\min\{3\frac{4-\gamma}{4\gamma},0\}}C\|\vr_h\|_{L^\infty(0,T;L^\gamma(\Om))}\|\Grad \phi\|_{L^4(0,T;L^\frac{12}{5}(\Om))}.
\end{equation*}
Since Corollary \ref{cor:energy} provides $\vr_h \inb L^\infty(0,T;L^\gamma(\Om))$, we can conclude \eqref{eq:p1}.

\vspace{0.5cm}
\noindent
\textit{2. Bound on $P^h_2$:}
An application of the Cauchy-Schwartz inequality yields
\begin{equation}\label{eq:p2start}
	\begin{split}
		\int_0^T|P_2|~dt & := \int_0^T\left|\sum_E \int_{\partial E} \jump{\vr_h}_{\partial E}(u_h \cdot \nu)^-\ov u_h(\Pi_h^V v - v)~dS(x)\right|~dt \\
		&\leq \left(\int_0^T\sum_E \int_{\partial E}\jump{\vr_h}_{\partial E}^2 |u_h \cdot \nu|~dS(x)dt\right)^\frac{1}{2}\\
		& \qquad \times\left(\int_0^T\sum_E \int_{\partial E} |u_h||\ov u_h|(\Pi_h^Vv - v)^2~dS(x)dt\right)^\frac{1}{2} \\
		&=: \sqrt{I_1}\times \sqrt{I_2}.
	\end{split}
\end{equation}	
From \eqref{eq:p11} and \eqref{eq:p1inv}, we have that
\begin{equation}\label{eq:p21}
	\sqrt{I_1} \leq h^{\max\{-3\frac{4-\gamma}{4\gamma},0\}}C\|\vr_h\|_{L^\infty(0,T;L^\gamma(\Om))}.
\end{equation}
Using the Trace inequality in Lemma \ref{lemma:toolbox}, the H\"older inequality, the stability of $\Pi_h^Q$ in $L^6$, and 
the interpolation error estimate for $\Pi_h^V$ in $L^3$ (Lemma \ref{lemma:interpolationerror}), we deduce
\begin{equation}\label{eq:p22}
	\begin{split}
		I_2 &\leq h^{-1}\int_0^T \left(\int_\Om |u_h|^3\left|\Pi_h^Q u_h\right|^3~dx\right)^\frac{1}{3}\\
		&\quad \qquad \qquad \times \left(\|\Pi_h^V v - v\|_{L^3(\Om)}^2 + h^2C\|\Grad v\|^2_{L^3(\Om)}\right)~dt \\
		&\leq h^{-1}C\int_0^T \|u_h\|_{L^6(\Om)}\|u_h\|_{L^6(\Om)} \\
		&\quad \qquad \qquad \times \left(\|\Pi_h^V v - v\|_{L^3(\Om)}^2 + h^2C\|\Grad v\|^2_{L^3(\Om)}\right)~dt \\
		& \leq hC\|u_h\|_{L^2(0,T;L^6(\Om))}^2\|\Grad v\|_{L^\infty(0,T;L^3(\Om))}^2,
	\end{split}
\end{equation}
where the term involving $u_h$ is bounded by Corollary \ref{cor:energy}.
 
Now, setting \eqref{eq:p21} and \eqref{eq:p22} in \eqref{eq:p2start} yields
\begin{equation*}
	\int_0^T|P_2|~dt \leq h^{\frac{1}{2}-\min\{3\frac{4-\gamma}{4\gamma},0\}}C\|\Grad v\|_{L^\infty(0,T;L^3(\Om))},
\end{equation*}
which concludes our proof of \eqref{eq:p2}.

\vspace{0.5cm}
\noindent
\textit{3. Bound on $P^h_3$:}
An application of the Cauchy-Schwartz inequality yields
\begin{equation*}
	\begin{split}
		&\int_0^T |P_3(v)|~dt \\
		&\qquad = \int_0^T \left|\sum_E \int_{\partial E} \up_-(\vr u)\jump{\ov u_h}_{\partial E}\left(\Pi_h^Q\Pi_h^V v - v \right)~dS(x)\right|~dt \\
		&\qquad\leq \left(2\sum_\Gamma \int_0^T\int_\Gamma \left|\up(\vr u)\right|\jump{\ov u_h}^2~dS(x)dt \right)^\frac{1}{2}\\
		&\qquad\qquad\quad \times\left(\sum_E\int_0^T\int_{\partial E}\left|\up(\vr u)\right|\left(\Pi_h^Q\Pi_h^V v - v \right)^2~dS(x)dt \right)^\frac{1}{2}\\
		&:= I_1\times I_2.
	\end{split}
\end{equation*}
From the energy estimate (Proposition \ref{prop:energy}), we have that
\begin{equation*}
	I_1 = \left(\sum_{k=1}^M \Delta t D^k_2\right)^\frac{1}{2} \leq C,
\end{equation*}
and hence it only remains to bound $I_2$.

Using the definition \eqref{def:up} of $\up(\vr u)$ and the H\"older inequality
\begin{equation*}
	\begin{split}
		I_2^2 & =  \sum_E\int_0^T\int_{\partial E}\left|\up(\vr u)\right|\left(\Pi_h^Q\Pi_h^V v - v \right)^2~dS(x)dt \\
		&\leq  \sum_E\int_0^T\int_{\partial E}(\vr_+ + \vr_-)|u_h \cdot \nu|\left([\Pi_h^Q\Pi_h^V v]_+ - v \right)^2~dS(x)dt \\
		&\leq  \sum_E \int_0^T \left(\|\vr_+ u_h\|_{L^\frac{6\gamma}{6+\gamma}(\partial E)} + \|\vr_- u_h\|_{L^\frac{6\gamma}{6+\gamma}(\partial E)}\right) \\
		&\qquad \qquad \quad \times \left\|[\Pi_h^Q\Pi_h^V v] - v\right\|_{L^\frac{6\gamma}{5\gamma-6}(\partial E)}^2~dt.
		\end{split}
\end{equation*}
To proceed, we apply the trace Lemma \ref{lemma:toolbox} followed by the inverse estimate (Lemma \ref{lemma:inverse})
\begin{equation*}
	\begin{split}
		I_2^2 &\leq  2h^{-1}C\|\vr_h u_h\|_{L^2(0,T;L^\frac{6\gamma}{6+\gamma}(\Om))}  \\
		&\qquad \times \left(\left\|\Pi_h^Q\Pi_h^V v - v\right\|_{L^4(0,T;L^\frac{6\gamma}{5\gamma-6}(\Om))}^2 + h^2\|\Grad v\|_{L^4(0,T;L^\frac{6\gamma}{5\gamma-6}(\Om))}^2\right) \\
		&\leq 2h^{-1}C\|\vr_h u_h\|_{L^2(0,T;L^\frac{6\gamma}{6+\gamma}(\Om))}\\
		&\qquad \qquad \times \left(\left\|\Pi_h^Q\Pi_h^V v - \Pi_h^V v\right\|_{L^4(0,T;L^\frac{6\gamma}{5\gamma-6}(\Om))}^2\right. \\
		&\qquad \quad \qquad \qquad + \left.  \left\|\Pi_h^V v - v\right\|_{L^4(0,T;L^\frac{6\gamma}{5\gamma-6}(\Om))}^2+ h^2\|\Grad v\|_{L^4(0,T;L^\frac{6\gamma}{5\gamma-6}(\Om))}^2\right) \\
		&\leq hC4\|\vr_h u_h\|_{L^2(0,T;L^\frac{6\gamma}{6+\gamma}(\Om))}\|\Grad v\|_{L^4(0,T;L^\frac{6\gamma}{5\gamma-6}(\Om))}^2,
	\end{split}
\end{equation*}
where we in the last inequality have used both the interpolation error of $\Pi_h^Q$ and $\Pi_h^V$ (see Lemma \ref{lemma:interpolationerror}). 
In particular, we have used the following estimate 
\begin{equation*}
	\begin{split}
		\left\|\Pi_h^Q\Pi_h^V v - \Pi_h^V v\right\|_{L^\frac{6\gamma}{5\gamma-6}(\Om)}^2 
		&\leq h^2C\sum_E \|\Grad \Pi_h^V v\|_{L^\frac{6\gamma}{5\gamma-6}(\Om)}^2\\
		&\leq h^2C\|\Grad v\|_{L^\frac{6\gamma}{5\gamma-6}(\Om)}^2.		
	\end{split}
\end{equation*}
Now, from Corollary \ref{cor:energy} we have that $\|\vr_h u_h\|_{L^2(0,T;L^\frac{6\gamma}{6+\gamma}(\Om))} \leq C$, 
and hence we can actually conclude that 
\begin{equation*}
	\int_0^T |P_3(v)|~dt \leq h^\frac{1}{2}C\|\Grad v\|_{L^4(0,T;L^\frac{6\gamma}{5\gamma-6}(\Om))},
\end{equation*}
which is \eqref{eq:p3}.

\vspace{0.5cm}
\noindent
\textit{4. Bound on $P^h_4$:}
By direct calculation using the H\"older inequality
\begin{equation*}
	\begin{split}
		&\int_0^T |P_4(v)|~dt \\
		&\qquad \leq \|\Div_h u_h\|_{L^2(0,T;L^2(\Om))}\left(\int_0^T\int_\Om |\vr_h \ov u_h|^2\left(\Pi_h^Vv -v\right)^2~dxdt \right)^\frac{1}{2} \\
		&\qquad \leq hC \|\vr_h \ov u_h\|_{L^2(0,T;L^\frac{2\gamma}{\gamma-2}(\Om))}\|\Grad v\|_{L^\infty(0,T;L^\gamma(\Om))},
	\end{split}
\end{equation*}
where we have used that $\Div_h u_h \inb L^2(0,T;L^2(\Om))$ and the interpolation error Lemma \ref{lemma:interpolationerror}.
Now, we apply the inverse estimate in Lemma \ref{lemma:inverse} to obtain
\begin{equation*}
	\begin{split}
			h\|\vr_h \ov u_h\|_{L^2(0,T;L^\frac{2\gamma}{\gamma-2}(\Om))} &\leq h^\frac{2(\gamma-9/4)}{\gamma}C\|\vr_h \ov u_h\|_{L^2(0,T;L^\frac{6\gamma}{3+\gamma}(\Om))}.
	\end{split}
\end{equation*}
We conclude \eqref{eq:p4} by recalling that the last term is bounded by Corollary \ref{cor:energy}.
\end{proof}

\subsection{The material momentum transport operator}
In our proof of convergence we will need an identity like \eqref{eq:moment-trans} 
 for the discretization of both terms in the discrete material 
transport operator ($(\vr u)_t + \Div (\vr u \otimes u)$) combined. 
To obtain the desired estimates, we will rely on the following 
weak time-continuity result:
\begin{lemma}\label{lem:timederiv}
Let $(\vr_h, u_h)$ satisfy the energy estimate \eqref{eq:energy}. Then,
	\begin{equation*}
		\|\vr_h u_h - (\vr_h u_h)(-\Delta t)\|_{L^\frac{6}{5}(\Delta t,T;L^\frac{3}{2}(\Om))} \leq (\Delta t)^\frac{1}{\gamma}C
	\end{equation*}
	As a consequence,
	\begin{equation*}
		\|D_t (\vr_h u_h)\|_{L^\frac{6}{5}(\Delta t,T;L^\frac{3}{2}(\Om))} \leq (\Delta t)^{\frac{1-\gamma}{\gamma}}C.
	\end{equation*}
\end{lemma}

\begin{proof}
From the energy estimate \eqref{eq:energy}, we have that
\begin{equation*}
	\sum_k \int_\Om P''(\vr_\ddagger)[\vr_h^k - \vr_h^{k-1}]~dx \leq C,
\end{equation*}
where $P''(\vr_\ddagger)$ is determined as the remainder in a Taylor expansion
and the mean value theorem. In particular, a simple calculation gives
\begin{equation*}
	\begin{split}
		\int_\Om P''(\vr_\ddagger) [\vr_h^k - \vr_h^{k-1}]^2~dx \geq \nu(\gamma)\int_\Om[\vr_h^k - \vr_h^{k-1}]^\gamma~dx,
	\end{split}
\end{equation*}
where $\nu(\gamma)$ only depend on $\gamma$. Hence, 
\begin{equation}\label{lions}
	\sum_k \|\vr_h^k - \vr_h^{k-1}\|_{L^\gamma(\Om)}^\gamma \leq C.
\end{equation}
By adding and subtracting, we have that
\begin{equation*}
	\begin{split}
		&\|\vr^k_h \ov u^k_h - \vr^{k-1}_h \ov u^{k-1}_h\|_{L^\frac{3}{2}(\Om)} \\
		&= \|\vr^{k-1}_h[\ov u_h^k - \ov u_h^{k-1}] + \ov u_h^k[\vr_h^k - \vr_h^{k-1}]\|_{L^\frac{3}{2}(\Om)} \\
		&\leq \|\vr_h^{k-1}\|_{L^3(\Om)}^\frac{1}{2}\left\|\sqrt{\vr_h^{k-1}}[\ov u_h^k - \ov u_h^{k-1}]\right\|_{L^2(\Om)} \\
		&\qquad \quad C\|\ov u_h^k\|_{L^6(\Om)}\|\vr^k_h - \vr^{k-1}_h\|_{L^\gamma(\Om)}
	\end{split}
\end{equation*}
We then integrate in time, apply several applications 
of the H\"older inequality, and utilize \eqref{lions} to deduce
\begin{equation*}
	\begin{split}
		&\sum_k\Delta t\|\vr^k_h \ov u^k_h - \vr^{k-1}_h \ov u^{k-1}_h\|_{L^\frac{3}{2}(\Om)}^\frac{6}{5} \\
		&\qquad \leq \Delta t\sum_k \|\vr_h^{k-1}\|_{L^3(\Om)}^\frac{6}{10}\left\|\sqrt{\vr_h^{k-1}}[\ov u_h^k - \ov u_h^{k-1}]\right\|^\frac{6}{5}_{L^2(\Om)} \\
		&\qquad \quad + C\Delta t\sum_k\|\ov u_h^k\|^\frac{6}{5}_{L^6(\Om)}\|\vr^k_h - \vr^{k-1}_h\|_{L^\gamma(\Om)}^\frac{6}{5} \\
		&\qquad \leq (\Delta t)^\frac{3}{5} \|\vr_h\|_{L^\infty(0,T;L^\gamma(\Om))}^\frac{2}{5}\left(\sum_k \Delta t\right)^\frac{2}{5}\\
		&\qquad \qquad \qquad \times\left(\sum_k\left\|\sqrt{\vr_h^{k-1}}[\ov u_h^k - \ov u_h^{k-1}]\right\|^2_{L^2(\Om)}\right)^\frac{3}{5} \\
		&\qquad \qquad + \Delta tC\left(\sum_k \|u_h^k\|_{L^6(\Om)}^2\right)^\frac{3}{5}\left(\sum_k \|\vr_h^k - \vr_h^{k-1}\|_{L^\gamma(\Om)}^3\right)^\frac{2}{5} \\
		&\qquad \leq (\Delta t)^\frac{3}{5}C T^\frac{2}{5} + (\Delta t)^\frac{6}{5\gamma}C\left(\sum_k \Delta t\right)^\frac{2(\gamma-3)}{5\gamma}\left(\sum_k \|\vr_h^k - \vr_h^{k-1}\|_{L^\gamma(\Om)}^\gamma\right)^\frac{6}{5\gamma}\\
		&\qquad \leq C\left((\Delta t)^\frac{3}{5} + (\Delta t)^\frac{6}{5\gamma} \right).
	\end{split}
\end{equation*}
This concludes the proof.
\end{proof}

Using the previous lemma, we 
are now ready to prove the following 
result and error bound.

\begin{lemma}\label{lem:consistency}
Let $(\vr_h, u_h)$ be the numerical solution
obtained through Definition \ref{def:scheme} and \eqref{def:schemeII}.
Then, for all $v \in L^\infty(0,T;W^{1,6}(\Om))$,
\begin{equation*}
	\begin{split}
		&\int_0^T\int_\Om D_t^h(\vr_h \ov u_h)  \Pi_h^Vv~dx + \sum_{E} \int_0^T\int_{\partial E} \up(\vr u\otimes \ov u)\ov{ \Pi_h^Vv_+}~dS(x)dt \\
		& \qquad = \int_0^T D_t^h(\vr_h \ov u_h) v - \vr_h \widetilde u_h \otimes \ov u_h:\Grad v~dxdt + F(v),
	\end{split}
\end{equation*}
where the reminder is bounded as
\begin{equation}\label{eq:fbound}
	|F(v)| \leq h^\frac{1}{\gamma}C\|\Grad v\|_{L^\infty(0,T;L^3(\Om))},
\end{equation}
with constant independent of $h$ and $\Delta t$.
\end{lemma}
\begin{proof}
Using \eqref{eq:moment-trans}, we have the identity
\begin{equation*}
	\begin{split}
		&\int_0^T\int_\Om D_t^h(\vr_h \ov u_h)  \Pi_h^V v~dx + \sum_{E} \int_0^T\int_{\partial E} \up(\vr u\otimes \ov u)\widehat{ \Pi_h^V v_+}~dS(x)dt \\
		& \qquad = \int_0^T D_t^h(\vr_h \ov u_h)  \Pi_h^V v\ - \vr_h u_h \otimes \ov u_h:\Grad  v~dxdt + \sum_{i=2}^4 \int_0^T P_i\left( v\right)~ dt \\
		&\qquad = \int_0^T D_t^h(\vr_h \ov u_h) v - \vr_h u_h \otimes \ov u_h:\Grad v~dxdt  \\
		&\qquad \qquad +   \sum_{i=2}^4 \int_0^T P_i\left( v\right)~ dt +\int_0^T D_t^h(\vr_h \ov u_h) \left(\Pi_h^V v - v\right)~ dxdt.
	\end{split}
\end{equation*}
We thus define $F(v)$ by
\begin{equation*}\label{eq:theF}
	\begin{split}
		\left|F(v) \right| &:= \left|\sum_{i=2}^4 \int_0^T P_i~ dt +\int_0^T D_t^h(\vr_h \ov u_h) \left(\Pi_h^V v - v\right)~ dxdt \right| \\
		&\leq \sum_{i=2}^4\int_0^T |P_i|~ dt + h\|\Grad v\|_{L^\infty(0,T;L^3(\Om))}\int_0^T\left\|D_t^h(\vr_h \ov u_h)\right\|_{L^\frac{3}{2}(\Om)}~dt \\
		&\leq \sum_{i=2}^4\int_0^T |P_i|~ dt + h(\Delta t)^{-1}(\Delta t)^\frac{1}{\gamma}\|\Grad v\|_{L^\infty(0,T;L^3(\Om))} \\
		&=\sum_{i=2}^4\int_0^T |P_i|~ dt + h^\frac{1}{\gamma}\|\Grad v\|_{L^\infty(0,T;L^3(\Om))},
	\end{split}
\end{equation*}
where we have used Lemma \ref{lem:timederiv} and $h = a\Delta t$.

Finally, by applying \eqref{eq:pcombined} to bound the $P_i$ terms, we obtain 
the desired \eqref{eq:fbound}.

\end{proof}

\subsection{The artificial stabilization terms}
To prove convergence of the numerical method
we will need to prove that the artificial stabilization terms 
converges to zero. Moreover, we will need that these terms 
are small in a suitable Lebesgue space.

\begin{lemma}\label{lem:stupid}
If $(\vr_h, u_h)$ is the numerical solution obtained 
by Definition \ref{def:scheme} and \eqref{def:schemeII}, 
the artificial stabilization terms satisfies
\begin{align}
		&\left|h^{1-\eps}\sum_\Gamma \int_0^T\int_\Gamma \jump{\vr_h}_\Gamma\jump{\Pi_h^Q \phi}_\Gamma~dS(x)dt\right| \nonumber\\
		&\qquad\quad
		\leq h^\frac{11 -6\eps}{12}C\|\Grad \phi\|_{L^2(0,T;L^2(\Om))},	\label{eq:stupid1} \\
		&\left|h^{1-\epsilon}\sum_\Gamma \int_0^T\int_{\Gamma} \left(\frac{\ov u_- + \ov u_+}2\right) \jump{\vr_h}_\Gamma\jump{\ov{ \Pi_h^Vv}}_\Gamma~dS(x)dt\right| \nonumber\\
		&\qquad\quad
		 \leq h^\frac{13-6\eps}{12}C\|\Grad v\|_{L^\infty(0,T;L^3(\Om))}, \label{eq:stupid2}
\end{align}
for all sufficiently smooth $\phi$ and $v$.
\end{lemma}
\begin{proof}
We will begin by proving \eqref{eq:stupid1}.
By direct calculation, using the Cauchy-Schwartz inequality, we deduce
\begin{align}\label{stupid:1}
		&h^{1-\eps}\sum_\Gamma \int_0^T\int_\Gamma \jump{\vr_h}_\Gamma\jump{\Pi_h^Q \phi}_\Gamma~dS(x)dt \nonumber\\
		&\qquad = -h^{1-\eps} \sum_E \int_0^T \int_{\partial E}\jump{\vr_h}_\Gamma\left(\Pi_h^Q \phi \right)~dS(x)dt \nonumber\\
		&\qquad = -h^{1-\eps} \sum_E \int_0^T \int_{\partial E}\jump{\vr_h}_\Gamma\left(\Pi_h^Q \phi - \phi\right)~dS(x)dt \nonumber\\
		&\qquad\leq h^{1-\eps}\left(\sum_\Gamma \int_0^T \int_\Gamma \jump{\vr_h}_\Gamma^2~dS(x)dt\right)^\frac{1}{2} \\
		&\qquad \qquad \times \left(\sum_E \int_{\partial E}(\Pi_h^Q \phi - \phi)^2~dS(x)dt\right)^\frac{1}{2}\nonumber\\
		&\qquad \leq h^\frac{1}{2}h^{1-\eps}\left(\sum_\Gamma \int_0^T \int_\Gamma \jump{\vr_h}_\Gamma^2~dS(x)dt\right)^\frac{1}{2}\|\Grad \phi\|_{L^2(0,T;L^2(\Om))},\nonumber
\end{align}
where the last inequality comes from the trace theorem (Lemma \ref{lemma:toolbox})
and the interpolation error estimate for $\Pi_h^Q$.
Now, to bound the jump term, we apply Lemma \ref{lem:renorm} with $B(\vr) = \vr^2$
to obtain
\begin{equation}\label{stupid:2}
	\begin{split}
			&h^{1-\eps}\sum_\Gamma\int_0^T\int_\Gamma \jump{\vr_h}^2_\Gamma~dS(x)dt \\
			&\qquad \leq \int_\Om \vr_0^2~dx + \int_0^T\int_\Om \vr_h^2 \Div u_h~dxdt \\
			&\qquad \leq \|\vr_0\|_{L^2(\Om)}^2 + C\|\vr_h\|^2_{L^\infty(0,T;L^4(\Om))}\|\Div u_h\|_{L^2(0,T;L^2(\Om))} \\
			&\qquad \leq \|\vr_0\|_{L^2(\Om)}^2 + h^{\max\{-\frac{3(\gamma-4)}{2\gamma},0\}}C\|\vr_h\|_{L^\infty(0,T;L^\gamma(\Om))}^2 \\
			&\qquad \leq \|\vr_0\|_{L^2(\Om)}^2 + h^{-\frac{1}{6}}C\|\vr_h\|_{L^\infty(0,T;L^\gamma(\Om))}^2 \leq h^{-\frac{1}{6}}C,
	\end{split}
\end{equation}
where we have applied the inverse estimate (Lemma \ref{lemma:inverse}), the energy bound,
 and that $\gamma > 3$. Finally, we apply \eqref{stupid:2} in \eqref{stupid:1} to discover
\begin{equation*}
	\begin{split}
		&h^{1-\eps}\sum_\Gamma \int_0^T\int_\Gamma \jump{\vr_h}_\Gamma\jump{\Pi_h^Q \phi}_\Gamma~dS(x)dt \\
		&\leq h^{\frac{1}{2}}h^{\frac{1-\eps}{2}}\left(h^{1-\eps}\sum_\Gamma\int_0^T\int_\Gamma \jump{\vr_h}^2_\Gamma~dS(x)dt\right)^\frac{1}{2}\|\Grad \phi\|_{L^2(0,T;L^2(\Om))} \\
		&\leq h^{1 - \frac{\eps}{2}-\frac{1}{12}}C\|\Grad \phi\|_{L^2(0,T;L^2(\Om))},
	\end{split}
\end{equation*}
which is \eqref{eq:stupid1}.

Next, we prove \eqref{eq:stupid2}. 
Since $\int_\Gamma \jump{\Pi_h^V v}_\Gamma~dS(x) = 0$, we have the identity
\begin{equation*}
	\begin{split}
		&-h^{1-\epsilon}\sum_\Gamma \int_0^T\int_{\Gamma} \left(\frac{\ov u_- + \ov u_+}2\right) \jump{\vr_h}_\Gamma\jump{\ov{ \Pi_h^Vv}}_\Gamma~dS(x)dt \\
		&= h^{1-\epsilon}\sum_E \int_0^T\int_{\partial E} \left(\frac{\ov u_- + \ov u_+}2\right) \jump{\vr_h}_\Gamma\left(\ov{ \Pi_h^Vv} -  \Pi_h^V v\right)~dS(x).
	\end{split}
\end{equation*}
An application of the H\"older inequality and several applications of the trace theorem (Lemma \ref{lemma:toolbox}) 
allow us to deduce
\begin{align}\label{stupid:3}
		&\left|h^{1-\epsilon}\sum_\Gamma \int_0^T\int_{\Gamma} \left(\frac{\ov u_- + \ov u_+}2\right) \jump{\vr_h}_\Gamma\jump{\ov{ \Pi_h^Vv}}_\Gamma~dS(x)dt\right|  \nonumber\\
		&\quad \leq h^{1-\epsilon}\left(\sum_\Gamma \int_0^T\int_\Gamma \jump{\vr_h}_\Gamma^2~dS(x)\right)^\frac{1}{2}\nonumber\\
		&\qquad \qquad\times
		\sup_t\left(\sum_E \int_{\partial E}\left|\ov{ \Pi_h^Vv} -  \Pi_h^V v\right|^3~dS(x)\right)^\frac{1}{3}\\
		&\qquad \qquad \times 
		\left(\int_0^T\left(\sum_E\int_{\partial E}\left|\frac{\ov u_- + \ov u_+}2\right|^6~dS(x)~\right)^\frac{2}{6}~dt\right)^\frac{1}{2} \nonumber\\
		&\quad \leq h^{1-\eps}C\left(\sum_\Gamma \int_0^T\int_\Gamma \jump{\vr_h}_\Gamma^2~dS(x)\right)^\frac{1}{2} \nonumber\\
		&\qquad \qquad \times \left(h^{-\frac{1}{3}}\|\widehat{\Pi_h^V v} - \Pi_h^V v\|_{L^\infty(0,T;L^3(\Om))}+ h^\frac{2}{3}\|\Grad v\|_{L^\infty(0,T;L^3(\Om))}\right) \nonumber\\
		&\qquad \qquad \times Ch^{-\frac{1}{6}}\|\widehat u_h\|_{L^2(0,T;L^6(\Om))}.\nonumber
\end{align}
Next, we apply \eqref{stupid:2} and the fact that $\gamma > 3$ to \eqref{stupid:3} to get
\begin{equation*}
	\begin{split}
		&\left|h^{1-\epsilon}\sum_\Gamma \int_0^T\int_{\Gamma} \left(\frac{\ov u_- + \ov u_+}2\right) \jump{\vr_h}_\Gamma\jump{\ov{ \Pi_h^Vv}}_\Gamma~dS(x)dt\right|  \\
		&\qquad \leq h^{\frac{1-\eps}{2}+\frac{2}{3} - \frac{1}{6} - \frac{1}{12}}C\|\Grad v\|_{L^\infty(0,T;L^3(\Om))} \\
		&\qquad = h^\frac{13-6\eps}{12}C\|\Grad v\|_{L^\infty(0,T;L^3(\Om))},
	\end{split}
\end{equation*}	
which is \eqref{eq:stupid2}.
\end{proof}

\subsection{Weak time control}
We conclude this section by proving $h$-independent bounds 
on the discrete time derivates in the numerical method. 
\begin{lemma}\label{lem:weaktime}
	Under the conditions of the previous lemma,
	\begin{equation}\label{eq:dtvr}
		 D_t \vr_h \inb L^\frac{4}{3}(0,T;W^{-1,\frac{3}{2}}(\Om)),
	\end{equation}
	\begin{equation}\label{eq:dtm}
		D_t (\vr_h u_h)\inb L^1(0,T;W^{-1, \frac{3}{2}}(\Om)).
	\end{equation}
\end{lemma}

\begin{proof}
Let $\phi \in C_0^\infty([0,T)\times \Om)$ be arbitrary and set $\Pi_h^Q \phi$ as test-function
in the continuity scheme \eqref{fem:cont} to obtain
\begin{equation}\label{weak:1}
	\begin{split}
		&\int_0^T\int_\Om D^h_t(\vr_h)\phi~dxdt \\
		&= \sum_\Gamma \int_0^T\int_\Gamma\up(\vr u)\jump{\Pi_h^Q \phi}_\Gamma + h^{1-\eps}\jump{\vr_h}_\Gamma\jump{\Pi_h^Q \phi}_\Gamma~dS(x)dt  \\
		&= \int_0^T\int_\Om \vr_h \tilde u_h \Grad \phi~dxdt + P_1(\phi) \\
		&\qquad \qquad + h^{1-\eps}\sum_\Gamma\int_0^T\int_\Gamma \jump{\vr_h}_\Gamma\jump{\Pi_h^Q \phi}_\Gamma~dS(x)dt,
	\end{split}
\end{equation}
where the last equality is \eqref{eq:cont-trans}. From Proposition \ref{pro:transport},  have that
\begin{equation}\label{weak:2}
	|P_1(v)| \leq h^\frac{1}{4}C\|\Grad \phi\|_{L^4(0,T;L^\frac{12}{5}(\Om))}.
\end{equation}
Moreover, from Lemma \ref{lem:stupid}, we have that
\begin{equation}\label{weak:3}
	h^{1-\eps}\sum_\Gamma\int_0^T\int_\Gamma\jump{\vr_h}_\Gamma\jump{\Pi_h^Q \phi}_\Gamma~dS(x)dt \leq h^{\frac{11-6\eps}{12}}\|\Grad \phi\|_{L^2(0,T;L^2(\Om))}.
\end{equation}
Hence, by applying \eqref{weak:2}, \eqref{weak:3}, and the H\"older inequality to \eqref{weak:1}, we conclude
\begin{equation*}
	\begin{split}
			&\left|\int_0^T\int_\Om D^h_t(\vr_h)\phi~dxdt\right| \\
			& \qquad \leq \|\vr_h  u_h\|_{L^2(0,T;L^3(\Om))}\|\Grad \phi\|_{L^2(0,T;L^\frac{3}{2}(\Om))} \\
			&\qquad\qquad + h^\frac{1}{4}C\|\Grad \phi\|_{L^4(0,T;L^\frac{12}{5}(\Om))} + h^{\frac{11-6\eps}{12}}\|\Grad \phi\|_{L^2(0,T;L^2(\Om))}.
	\end{split}
\end{equation*}
We conclude \eqref{eq:dtvr} by recalling that $\phi$ was chosen arbitrarily.

Next, let $v \in [C_0^\infty([0,T)\times \Om)]^d$ be an arbitrary vector and set $v_h = \Pi^V_h v$ 
as test-function in the momentum scheme \eqref{fem:moment} to obtain
\begin{equation*}
	\begin{split}
		&\int_0^T\int_\Om D_t^h(\vr_h u_h)v_h~dxdt + \sum_E \int_0^T\int_{\partial E}\up(\vr u\otimes u)\ov  v_h dS(x)dt \\
		& \qquad  +\int_0^T\int_\Om \Grad_h u_h \Grad_h v_h- p(\vr_h)\Div_h v_h~dxdt \\
 		& \qquad \qquad +h^{1-\epsilon}\sum_E \int_{\partial E} \left(\frac{\ov u_- + \ov u_+}2\right) \jump{\vr_h}_{\partial E}\ov v_h~dS(x) = 0.
	\end{split}
\end{equation*}
Now, to the first two terms we apply the identity in Lemma \ref{lem:consistency} and reorder terms to obtain
\begin{equation}\label{weak:start}
	\begin{split}
		&\int_0^T\int_\Om D_t^h (\vr_h u_h)v~dxdt \\
		&\qquad= \int_0^T\int_\Om \vr_h \tilde u_h\otimes \ov u_h:\Grad v~dxdt + F(v)\\
		&\qquad\qquad\quad +\int_0^T\int_\Om \Grad_h u_h \Grad_h v_h- p(\vr_h)\Div_h v_h~dxdt \\
		&\qquad\qquad\quad + h^{1-\epsilon}\sum_E \int_0^T\int_{\partial E} \left(\frac{\ov u_- + \ov u_+}2\right) \jump{\vr_h}_{\partial E}\ov v_h~dS(x)dt.
	\end{split}
\end{equation}
From Lemma \ref{lem:consistency}, we also have the bound
\begin{equation}\label{weak:thef}
	|F(v)| \leq h^\frac{1}{4}C\|\Grad v\|_{L^\infty(0,T;L^3(\Om))}.
\end{equation}
Next, we invoke Lemma \ref{lem:stupid} to obtain the bound
\begin{equation}\label{weak:stupid}
	\begin{split}
		&\left|h^{1-\epsilon}\sum_E \int_0^T\int_{\partial E} \left(\frac{\ov u_- + \ov u_+}2\right) \jump{\vr_h}_{\partial E}\ov v_h~dS(x)dt\right| \\
		&\qquad \leq h^\frac{13-6\epsilon}{12}C\|\Grad v\|_{L^\infty(0,T;L^3(\Om))}.
	\end{split}
\end{equation}
By applying \eqref{weak:thef} and \eqref{weak:stupid}, together with the 
H\"older inequality in \eqref{weak:start}, we deduce
\begin{equation*}
	\begin{split}
		\left|\int_0^T\int_\Om D_t^h (\vr_h u_h)v~dxdt\right| 
		&\leq \|\vr_h |u_h|^2\|_{L^2(0,T;L^\frac{3}{2}(\Om))}\|\Grad \phi\|_{L^2(0,T;L^3(\Om))} \\
		&\quad + h^\frac{1}{4}C\|\Grad v\|_{L^\infty(0,T;L^3(\Om))} \\
		&\quad + \|p(\vr_h)\|_{L^\infty(0,T;L^\gamma(\Om))}\|\Div v\|_{L^1(0,T;L^\frac{3}{2}(\Om))}\\
		&\quad + h^\frac{13-6\eps}{12}C\|\Grad v\|_{L^\infty(0,T;L^3(\Om))}.
	\end{split}
\end{equation*}
From this we easily conclude \eqref{eq:dtm}.

\end{proof}

\section{Higher integrability on the density}\label{sec:higher}
In order to pass to the limit in the pressure, 
we will need higher integrability on the density. 
That is, from the energy estimate (Corollary \ref{cor:energy}) we only know that
$$
p(\vr_h)  \inb L^\infty(0,T;L^1(\Om)),
$$
and $L^1$ is not weakly closed. Hence, it is not clear 
that $p(\vr_h)$ converges to an integrable function. 
To prove higher integrability, we shall make use
of the operator $\mathcal{A}^i[\cdot]:L^p(\Om) \mapsto W^{1,p}(\Om)$
\begin{equation*}
	\mathcal{A}^i[q] = \left. \frac{d}{dx_i}\Delta^{-1}\left[\Pi^{E} q\right]\right|_{\Om}, \quad i=1,\ldots, 3.
\end{equation*}
Here, $\Pi_E$ is the extension by zero operator to all of $\R^n$
 and $|_\Om$ denotes the restriction to $\Om$. The $\Delta^{-1}$ operator
is the usual convolution with the Newtonian potential
\begin{equation*}
	\Delta^{-1}\phi = -\lambda \int_{\R^3}\frac{\phi(y)}{|x-y|}~dy, \quad \lambda > 0.
\end{equation*}
Using $\mathcal{A}_i$, we define two operators $\mathcal{A}^\Grad$ and $\mathcal{A}^\text{div}$
acting on scalars and vectors, respectively.  For a scalar $q$ and a vector $v = [v_1, v_2, v_3]^T$
they are defined
\begin{equation}\label{def:aop}
	\mathcal{A}^\Grad[q] = 
	\begin{pmatrix}
		\mathcal{A}^1[q] \\
		\mathcal{A}^2[q] \\
		\mathcal{A}^3[q] \\
	\end{pmatrix}, \qquad 
		\mathcal{A}^\text{div}[v] = \mathcal{A}^1[v_1] + \mathcal{A}^2[v_2] + \mathcal{A}^3[v_3].
\end{equation}
By direct calculation, one easily verifies the following lemma.
\begin{lemma}\label{lem:duality}
For any two $f \in L^p(\Om)$ and $g \in L^q(\Om)$ with $\frac{1}{p} + \frac{1}{q} = 1$ 
and $1 < p,q < \infty$, the following identity holds
	\begin{equation*}
		\int_\Om v\mathcal{A}^\Grad[g\phi]~dx = -\int_\Om \mathcal{A}^\text{div}[f]\phi g~dx, \quad \forall \phi \in C_0^\infty(\Om).
	\end{equation*}
Moreover, there is a constant $C$ such that,
\begin{equation*}
	\|A^i[f]\|_{L^p(\Om)} + \|\Grad A^i[f]\|_{L^q(\Om)}  \leq C\|f\|_{L^q(\Om)}, \quad p < \frac{3q}{3-q}.
\end{equation*}
\end{lemma}

We are now ready to prove higher integrability 
of the numerical density.
\begin{proposition}\label{pro:higher}
Let $(\vr_h, u_h)$ be the numerical approximation constructed 
through Definition \ref{def:scheme} and \eqref{def:schemeII}.
The following integrability estimate holds
	\begin{equation*}
		\vr_h \inb L^{\gamma + 1}_\text{loc}([0,T]\times \Om).
	\end{equation*}
\end{proposition}

\begin{proof} 
Let $\phi \in C_0^\infty(\Om)$ be arbitrary and define the test-functions
\begin{equation*}
	v = \phi\mathcal{A}^\Grad[\vr_h \phi], \qquad v_h = \Pi_h^V v.
\end{equation*}
Since $\vr_h$ is piecewise constant in time, so is $v$ and hence also $v_h$.
Moreover, since $\phi$ vanishes at the boundary, the degrees of freedom of $v_h$ 
is zero at the boundary.
As a consequence, $v_h$ is a valid test-function in the momentum scheme \eqref{fem:moment}.

Note that the energy estimate, Lemma \ref{lem:duality}, and the H\"older inequality provides the bound
\begin{equation}\label{eq:dv}
	\|\Grad v\|_{L^\infty(0,T;L^p(\Om))} \leq C\|\phi\|^2_{W^{1,\infty}(\Om)}\|\vr_h\|_{L^\infty(0,T;L^\gamma(\Om))} \leq C\|\phi\|^2_{W^{1,\infty}(\Om)},
\end{equation}
for any $p \leq \gamma$.

Now, by applying $v_h$ as test-function in \eqref{fem:moment}, integrating in time, and reordering terms, 
we obtain
\begin{equation}\label{high:start}
	\begin{split}
		\int_0^T\int_\Om p(\vr_h)\Div_h v_h~dxdt & =  I_1 + I_2 + I_3,
	\end{split}
\end{equation}
where
\begin{align*}
	I_1 &= \int_0^T\int_\Om \Grad_h u_h\Grad_h  v_h~dxdt, \\
	I_2 &= h^{1-\eps}\sum_\Gamma \int_0^T\int_\Gamma \left(\frac{\ov u_+ +\ov u_-}{2} \right)\jump{\vr_h}_\Gamma\jump{\ov v_h}_\Gamma~dS(x)dt, \\
	I_3 &= \int_0^T\int_\Om D_t^h(\vr_h \ov u_h)v_h~dxdt-\sum_\Gamma \int_0^T \int_{\Gamma}\up(\vr  u \otimes \ov u)\jump{\ov v_h}_\Gamma~dS(x)dt, \\
\end{align*}
Before we start deriving bounds for $I_1$, $I_2$, and $I_3$, let 
us first consider the term on the left-hand side of \eqref{high:start}.
 For this purpose, recall from Section 2 that the finite element spaces 
are chosen such that
\begin{equation*}
	\Div_h \Pi_h^V v = \Pi_h^Q \Div v.
\end{equation*}
Hence, we have that
\begin{equation}\label{eq:thecalc}
	\begin{split}
		\int_\Om p(\vr_h)\Div_h v_h~dx &= \int_\Om p(\vr_h)\Pi_h^Q \left[\Div\left( \phi\mathcal{A}^\Grad[\phi \vr_h]\right)\right]~dx \\ 
		&= \int_\Om p(\vr_h)\Div\left( \phi\mathcal{A}^\Grad[\phi \vr_h]\right)~dx \\
		&=\int_\Om p(\vr_h)\Grad \phi \cdot \mathcal{A}^\Grad[\phi \vr_h] + \phi^2 p(\vr_h)\vr_h~dx.
	\end{split}
\end{equation}
By setting this expression in \eqref{high:start}, we see that
\begin{equation}\label{high:start2}
	\int_0^T\int_\Om \phi^2 p(\vr_h)\vr_h~dx = I_1 + I_2 + I_3 + I_4, 
\end{equation}
where now
\begin{equation*}
	I_4 = - \int_\Om p(\vr_h)\Grad \phi \cdot \mathcal{A}^\Grad[\phi \vr_h]~dx.
\end{equation*}
Thus, the proof follows provided we can bound 
 $I_1$, $I_2$, $I_3$, and $I_4$.

\vspace{0.2cm}
\noindent
\textit{1. Bounds on $I_1$ and $I_2$:}	
Using the Cauchy-Schwartz inequality and the 
energy estimate (Corollary \ref{cor:energy}), 
we have that
\begin{equation}\label{high:i3}
	\begin{split}
		|I_1| &\leq 2\|\Grad_h u_h\|_{L^2(0,T,L^2(\Om))}\|\Grad_h v_h\|_{L^2(0,T;L^2(\Om))} \\
		&\leq C\|\Grad v\|_{L^2(0,T;L^2(\Om))} \leq C\|\phi\|_{W^{1,\infty(\Om)}}^2,
	\end{split}
\end{equation}
where the last inequality is \eqref{eq:dv}.	

To bound the $I_2$ term, we apply Lemma \ref{lem:stupid} to obtain
\begin{equation}\label{high:i2}
	|I_2| \leq h^\frac{1}{2}\|\Grad v\|_{L^\infty(0,T;L^\gamma(\Om))}\leq h^\frac{1}{2}C\|\phi\|_{W^{1,\infty}(\Om)}^2,
\end{equation}
where again the last inequality is  \eqref{eq:dv}.

\vspace{0.2cm}
\noindent
\textit{2. Bound on $I_3$:}	
From Lemma \ref{lem:consistency}, we have the identity
\begin{equation}\label{high:i1start}
	\begin{split}
		I_3 &= \int_0^T\int_\Om D_t^h(\vr_h \ov u_h)v - \vr_h \widetilde u_h \otimes \ov u_h:\Grad v~dxdt + F(v),
	\end{split}
\end{equation}
where $F(v)$ is bounded by \eqref{eq:fbound} and \eqref{eq:dv} as
\begin{equation*}\label{high:i1f}
	\begin{split}
		|F(v)| &\leq h^\frac{1}{4}C\|\Grad v\|_{L^\infty(0,T;L^\gamma(\Om))} 
			  	\leq h^\frac{1}{4}C\|\phi\|_{W^{1,\infty}(\Om)}^2.
	\end{split}
\end{equation*}
It remains to bound the two other terms in $I_3$. Let us begin with 
the easiest term. For this purpose, we apply the H\"older inequality 
and \eqref{eq:dv} to deduce
\begin{equation}\label{high:i11}
	\begin{split}
		&\left|\int_0^T\int_\Om \vr_h \widetilde u_h \otimes \ov u_h:\Grad v~dxdt\right| \\
		&\qquad \leq \|\vr_h \widetilde u_h \ov u_h\|_{L^2(0,T;L^\frac{\gamma}{\gamma-1}(\Om))}\|\Grad v \|_{L^\infty(0,T;L^\gamma(\Om))} \\
		&\qquad \leq C \|\phi\|_{W^{1,\infty}(\Om)}^2\|\vr_h \widetilde u_h \ov u_h\|_{L^2(0,T;L^\frac{\gamma}{\gamma-1}(\Om))} 
		\leq C\|\phi\|_{W^{1,\infty}(\Om)}^2,
	\end{split}
\end{equation}
where the last bound comes from Corollary \ref{cor:energy} using that
$$
\frac{3\gamma}{3+\gamma} \geq \frac{\gamma}{\gamma -1} \quad \text{for}\quad \gamma \geq 3.
$$
The remaining term in $I_3$ is more complicated. By summation by parts in time followed
by Lemma \ref{lem:duality}, we calculate
\begin{align}
		&\int_0^T\int_\Om D_t^h(\vr_h \ov u_h)v~dxdt \nonumber\\
		&\qquad = \Delta t \sum_{k=1}^M \int_\Om \frac{\vr_h^k \ov u_h^k - \vr_h^{k-1} \ov u_h^{k-1}}{\Delta t}v^k~dx \nonumber\\
		&\qquad= - \Delta t \sum_{k=1}^M \int_\Om \vr_h^{k-1} \ov u_h^{k-1}\frac{v^k- v^{k-1}}{\Delta t}~dx \nonumber\\
		&\qquad\qquad \qquad \qquad- \int_\Om \vr_h^0\ov u_h^0 v^1~dx 
		 + \int_\Om \vr_h^M \ov u_h^M v^M~dx \nonumber.
\end{align}
To this identity, we apply the definition of $v^k$ and Lemma \ref{lem:duality} and write
\begin{align}
		&\int_0^T\int_\Om D_t^h(\vr_h \ov u_h)v~dxdt \nonumber\\
		&\qquad=-\Delta t \sum_{k=1}^M \int_\Om \vr_h^{k-1} \ov u_h^{k-1} \phi \mathcal{A}^\Grad\left[\phi \frac{\vr^k_h - \vr^{k-1}_h}{\Delta t}\right]~dx \nonumber\\
		&\qquad \qquad\qquad \qquad - \int_\Om \vr_h^0\ov u_h^0 v^1~dx 
		 +\int_\Om \vr_h^M \ov u_h^M v^M~dx \nonumber\\
		&\qquad=-\Delta t \sum_{k=1}^M \int_\Om \mathcal{A}^\text{div}\left[\phi \vr_h^{k-1} \ov u_h^{k-1}\right] \phi \frac{\vr^k_h - \vr^{k-1}_h}{\Delta t}~dx \label{high:i12}\\
		&\qquad \qquad\qquad \qquad- \int_\Om \vr_h^0\ov u_h^0 v^1~dx 
		 +\int_\Om \vr_h^M \ov u_h^M v^M~dx.\nonumber
\end{align}
The last two terms are easily bounded since \eqref{eq:dv} and $\gamma > 3$ gives that $v \inb L^\infty(0,T;L^\infty(\Om))$
and hence
\begin{equation*}
	\begin{split}
		&\left|- \int_\Om \vr_h^0\ov u_h^0 v^1~dx  +\int_\Om \vr_h^M \ov u_h^M v^M~dx\right|
		\leq C\|\phi\|^2_{L^\infty(\Om)}\|\vr_h u_h\|_{L^\infty(0,T;L^1(\Om))},
	\end{split}
\end{equation*}
where the last term is bounded by Corollary \ref{cor:energy}. To bound the other term in \eqref{high:i12}
we apply the continuity scheme \eqref{fem:cont} with 
$$
q_h = \Pi_h^Q\left[\mathcal{A}^\text{div}\left[\phi \vr_h^{k-1} \ov u_h^{k-1}\right] \phi\right]
$$
which gives
\begin{equation*}
	\begin{split}
		&-\Delta t \sum_{k=1}^M \int_\Om \mathcal{A}^\text{div}\left[\phi \vr_h^{k-1} \ov u_h^{k-1}\right] \phi \frac{\vr^k_h - \vr^{k-1}_h}{\Delta t}~dx \\
		&\qquad = \Delta t \sum_{k=1}^M\sum_\Gamma \int_\Gamma\up^k(\vr u)\jump{\Pi_h^Q\left[\mathcal{A}^\text{div}\left[\phi\vr_h^{k-1} \ov u_h^{k-1}\right] \phi\right]}_\Gamma~dS(x) \\
		&\qquad = \Delta t \sum_{k=1}^M\int_\Om \vr^k_h \widetilde u^k_h \Grad \left(\mathcal{A}^\text{div}\left[\phi \vr_h^{k-1} \ov u_h^{k-1}\right] \phi\right)~dx\\
		&\qquad \qquad \qquad \qquad 
		+ P_1\left(\mathcal{A}^\text{div}\left[\phi \vr_h^{k-1} \ov u_h^{k-1}\right] \phi\right),
	\end{split}
\end{equation*}
where the last equality is \eqref{eq:cont-trans} in Lemma \ref{lem:transport}. Now, by applying Proposition \ref{pro:transport} (i.e \eqref{eq:pcombined})
and the H\"older inequality in the previous identity, we obtain
\begin{equation}\label{high:i13}
	\begin{split}
		&\left|-\Delta t \sum_{k=1}^M \int_\Om \mathcal{A}^\text{div}\left[\phi\vr_h^{k-1} \ov u_h^{k-1}\right] \phi \frac{\vr^k_h - \vr^{k-1}_h}{\Delta t}~dx\right|\\
		&\quad \leq 2\|\vr_h \widetilde u_h\|_{L^2(0,T;L^2(\Om))}\|\phi\|^2_{W^{1,\infty}(\Om)}\left\|\mathcal{A}^\text{div}\left[\vr_h \ov u_h\right]\right\|_{L^2(0,T;W^{1,2}(\Om))} \\
		&\qquad \qquad \qquad + h^\frac{1}{4}C\|\phi\|^2_{W^{1,\infty}}\left\|\mathcal{A}^\text{div}\left[\vr_h \ov u_h\right]\right\|_{L^4(0,T;W^{1,3}(\Om))} \\
		&\quad \leq C\|\phi\|^2_{W^{1,\infty}(\Om)}\left(1 +h^\frac{1}{4}\|\vr_h \ov u_h\|_{L^4(0,T;L^3(\Om))} \right),
	\end{split}
\end{equation}
where the last inequality follows from the properties of $\mathcal{A}^i$ (Lemma \ref{lem:duality}) 
together with Corollary \ref{cor:energy} giving $\vr_h \widetilde u_h \inb L^2(0,T;L^2(\Om))$. Next, we apply 
the inverse estimate in Lemma \ref{lemma:inverse} to bound the last term
\begin{equation*}
	h^\frac{1}{4}\|\vr_h \ov u_h\|_{L^4(0,T;L^3(\Om))} \leq h^\frac{1}{4}(\Delta t)^{-\frac{1}{4}}C\|\vr_h \ov u_h\|_{L^2(0,T;L^3(\Om))},
\end{equation*}
which is bounded by Corollary \ref{cor:energy}. By applying this in \eqref{high:i13}, using that $h = a \Delta t$,  and 
setting the result into \eqref{high:i12} we obtain
\begin{equation*}
	\begin{split}
		\left|\int_0^T\int_\Om D_t^h(\vr_h \ov u_h)v~dxdt \right|
		\leq C\|\phi\|^2_{W^{1,\infty}(\Om)}.
	\end{split}
\end{equation*}
This, together with \eqref{high:i11} and \eqref{high:i1start}, yields
\begin{equation}\label{high:i1}
	|I_3| \leq C\|\phi\|^2_{W^{1,\infty}(\Om)}.
\end{equation}

\vspace{0.2cm}
\noindent
\textit{3. Bound on $I_4$:}
To bound the $I_4$, we will use that $\gamma > 3$ 
yielding $W^{1,\gamma}$ embedded in $L^\infty$.  
The resulting calculation is 
\begin{equation}\label{high:i4}
	\begin{split}
		|I_4| &=  \left|\int_\Om p(\vr_h)\Grad \phi \cdot \mathcal{A}^\Grad[\phi \vr_h]~dx\right| \\
		&\leq C\|p(\vr_h)\|_{L^\infty(0,T;L^1(\Om))}\|\Grad \phi\|_{L^\infty(\Om)}\|\mathcal{A}^\Grad[\phi \vr_h]\|_{L^\infty(0,T;L^\infty(\Om))} \\
		&\leq C\|\Grad \phi\|_{L^\infty(\Om)}\|\mathcal{A}^\Grad[\phi \vr_h]\|_{L^\infty(0,T;W^{1,\gamma}(\Om))} 
		\leq C\|\Grad \phi\|_{L^\infty(\Om)}^2,
	\end{split}
\end{equation}
where the last inequality is derived as in \eqref{eq:dv}. Here, we have also used Corollary \ref{cor:energy}
which tell us that $p(\vr_h) \inb L^\infty(0,T;L^1(\Om))$.

\vspace{0.2cm}
\noindent
\textit{4. Conclusion:}	
By applying \eqref{high:i3}, \eqref{high:i2}, \eqref{high:i1}, and \eqref{high:i4} 
in \eqref{high:start2}, we obtain
\begin{equation*}
	\int_0^T\int_\Om \phi^2 p(\vr_h)\vr_h~ dxdt \leq C\|\phi\|_{W^{1,\infty}(\Om)}^2.
\end{equation*}
Since $p(\vr_h)= a\vr_h^\gamma$ and $\phi \in C_0^\infty(\Om)$ can be chosen arbitrary,
this concludes the proof.

\end{proof}

\section{Weak convergence}\label{sec:weak}
In this section, we will pass 
to the limit in the numerical method 
and conclude that the limit of $(\vr_h, u_h)$
is almost a weak solution of the compressible 
Navier-Stokes equations. What remains in order to prove 
the main theorem is that $p(\vr_h) \weak p(\vr)$, which 
will be the topic of the ensuing sections.

Our starting point is that Corollary \ref{cor:energy} and Proposition \ref{pro:higher}
allow us to assert the existence of functions 
\begin{equation*}
	\begin{split}
		\vr \in L^\infty(0,T;L^\gamma(\Om)), \quad  u \in L^2(0,T;W^{1,2}_0(\Om)),
	\end{split}
\end{equation*}
and a subsequence $h_j \rightarrow 0$ such that
\begin{equation}\label{eq:conv1}
	\begin{split}
		\vr_h &\weakstar \vr \text{ in }L^\infty(0,T;L^\gamma(\Om)), \\
		u_h &\weak u \text{ in }L^2(0,T;L^2(\Om)), \\
		\Grad_h u_h & \weak \Grad u \text{ in }L^2(0,T;L^2(\Om)).
	\end{split}
\end{equation}
From Corollary \ref{cor:energy}, we immediately obtain the integrability
\begin{equation*}
	\vr u \in L^\infty(0,T; L^{m_\infty}(\Om))\cap L^2(0,T;L^{m_2}(\Om)),
\end{equation*}
\begin{equation*}
	\vr u \otimes u \in L^\infty(0,T;L^1(\Om))\cap L^2(0,T;L^{c_2}(\Om)),
\end{equation*}
with exponents
\begin{equation*}
	m_\infty = \frac{2\gamma}{\gamma + 1}> \frac{3}{2}, \quad m_2 = \frac{6\gamma}{3+\gamma} > 3, \quad c_2 = \frac{3\gamma}{3+\gamma}> \frac{3}{2}.
\end{equation*}
Moreover, Lemma \ref{lem:timecompactness} together with Lemma \ref{lem:weaktime} provides the bounds
\begin{equation*}
\vr_h \inb C(0,T;L^\gamma(\Om)), \qquad	\vr_h u_h \inb C(0,T;L^{m_\infty}(\Om)).
\end{equation*}

In the following lemma, we prove that the above convergences are sufficient 
to pass to the limit in both nonlinear terms involving $u$.

\begin{lemma}\label{lem:product}
Given the convergences \eqref{eq:conv1},
	\begin{align}
		\vr_h  u_h, ~ \vr_h\tilde u_h &\weak \vr u \text{ in } L^2(0,T;L^{m_2}(\Om)), \nonumber\\
		\vr_h \widehat u_h &\weak \vr u \text{ in } C(0,T;L^{m_\infty}(\Om)).\label{eq:conv2} \\
		\vr_h \tilde u_h \otimes \widehat u_h &\weak \vr u \otimes u \text{ in }L^2(0,T;L^{c_2}(\Om)), \nonumber
	\end{align}
\end{lemma}
\begin{proof}
Lemma \ref{lem:translation} tell us that $u_h$ is spatially compact in $L^2(0,T;L^p(\Om))$ for any
$p < 6$. From Lemma \ref{lem:weaktime} we have that $D_t^h \vr_h \inb L^\frac{4}{3}(0,T;W^{-1,\frac{3}{2}}(\Om))$.
As a consequence, we can apply Lemma \ref{lemma:aubinlions}, with $g_h = \vr_h$ and $f_h = u_h$, to obtain
\begin{equation*}
	\vr_h  u_h \weak \vr u \text{ in } L^2(0,T;L^{m_2}(\Om)).
\end{equation*}
To conclude weak convergence of $\vr_h \widetilde u_h$ and $\vr_h \ov u_h$, we write 
\begin{equation*}
	\vr_h \widetilde u_h = \vr_h u_h + \vr_h (\Pi_h^\mathcal{N} u_h - u_h), \qquad \vr_h \ov u_h = \vr_h u_h + \vr_h (\Pi_h^Q u_h - u_h).
\end{equation*}
The interpolation error estimate (Lemma \ref{lemma:interpolationerror}) tell us that
\begin{equation*}
	\|\Pi_h^\mathcal{N} u_h - u_h\|_{L^2(0,T;L^2(\Om))} +  \|\Pi_h^Q u_h - u_h\|_{L^2(0,T;L^2(\Om))} \leq C\|\Grad_h u_h\|_{L^2(0,T;L^2(\Om))},
\end{equation*}
and hence $\vr_h \widetilde u_h \weak \vr u$, $\vr_h \ov u_h$ $\weak \vr u$ in the sense of 
distributions on $(0,T)\times \Om$.

Next, since $D_t^h(\vr_h \widehat u_h) \inb L^1(0,T;W^{-1,\frac{3}{2}}(\Om))$ (from Lemma \ref{lem:weaktime}), 
another application of Lemma \ref{lemma:aubinlions}, this time with $g_h = \vr_h \widehat u_h$ and $f_h = u_h$, 
renders
\begin{equation*}
	\vr_h \ov u_h \otimes u_h \weak \vr u \otimes u \text{ in }L^2(0,T;L^{c_2}(\Om)).
\end{equation*}
Clearly, this also implies that $\vr_h u_h \otimes \ov u_h \weak \vr u\otimes u$. The 
convergence of $\vr_h \widetilde u_h \otimes \ov u_h$ then follows 
by writing
\begin{equation*}
	\vr_h \widetilde u_h \otimes \ov u_h = \vr_h u_h \otimes \ov u_h  + \vr_h(\Pi_h^\mathcal{N} u_h - u_h)\otimes \ov u_h,
\end{equation*}
and applying the interpolation error estimate on the remainder.

\end{proof}

Using the convergences we have derived thus far, 
we are able to pass to the limit in the continuity 
approximation \eqref{fem:cont}.

\begin{lemma}\label{lem:cont}
Let $(\vr_h, u_h)$ be the numerical solution 
obtained by Definition \ref{def:scheme} and \eqref{def:schemeII}. The limit $(\vr, u)$
 is a weak solution of continuity equation:
\begin{equation*}
	\vr_t + \Div(\vr u) = 0 \quad \text{in }\mathcal{D}'([0,T)\times \overline{\Om}).
\end{equation*}
\end{lemma}
\begin{proof}
Let $\phi \in C_0^\infty([0,T)\times \overline{\Om})$ be arbitrary
and set $\Pi_h^Q \phi$ as test-function in the continuity scheme \eqref{fem:cont}
to obtain
\begin{equation}\label{cont:start}
	\begin{split}
		&\int_0^T\int_\Om D_t^h(\vr_h)\phi~dxdt \\
		&\qquad = \sum_\Gamma \int_0^T \int_\Gamma \up(\vr u)\jump{\Pi_h^Q \phi} - h^{1-\eps}\jump{\vr_h}_\Gamma \jump{\Pi_h^Q \phi}~dS(x)dt \\
		&\qquad = \int_0^T\int_\Om \vr_h \tilde u_h \Grad \phi~dxdt + P_1(\phi) \\
		&\qquad \qquad - h^{1-\eps}\sum_\Gamma \int_0^T \int_\Gamma\jump{\vr_h}_\Gamma \jump{\Pi_h^Q \phi}~dS(x)dt,
	\end{split}
\end{equation}
where the last equality is Lemma \ref{lem:transport}. 

From Proposition \ref{pro:transport}, we have 
that 
\begin{equation}\label{cont:1}
	|P_1(\phi)| \leq h^\frac{1}{4}C\|\phi\|_{L^4(0,T;L^\frac{12}{5}(\Om))}.
\end{equation}

From Lemma \ref{lem:stupid}, we have the bound
\begin{equation}\label{cont:2}
	\left|\sum_\Gamma \int_0^T \int_\Gamma h^{1-\eps}\jump{\vr_h}_\Gamma \jump{\Pi_h^Q \phi}~dS(x)dt\right|
	\leq h^\frac{11-6\eps}{12}C\|\Grad \phi\|_{L^2(0,T;L^2(\Om))}.
\end{equation}

Summation by parts provides the identity
\begin{equation}\label{cont:3}
	\begin{split}
		&\int_0^{T}\int_\Om D_t^h(\vr_h)\phi~dxdt  \\
		&\qquad \qquad = - \int_0^{T-\Delta t}\int_\Om \vr_h D_t^h(\phi(+\Delta t))~dxdt - \int_\Om \vr_h^0 \phi(\Delta t)~dx.
	\end{split}
\end{equation}

Now, we apply \eqref{cont:3} in \eqref{cont:start} and send $h \rightarrow 0$
along the subsequence where $\vr_h \tilde u_h \weak \vr u$ and apply
\eqref{cont:1} and \eqref{cont:2} to obtain
\begin{equation*}
	\int_0^T\int_\Om \vr (\phi_t + u\cdot \Grad \phi)~dxdt = -\int_\Om \vr_0\phi(0, \cdot)~dx.
\end{equation*}
This concludes the proof.

\end{proof}

Next, we prove that the limit $(\vr, u)$ is almost 
a weak solution of the momentum equation \eqref{eq:moment}
and hence that it only remains to prove strong convergence of the density.

\begin{lemma}\label{lem:weakmoment}
Let $(\vr_h, u_h)$ be as in the previous lemma. The limit $(\vr, u)$
satisfies
\begin{equation}\label{eq:limbar}
	\begin{split}
		(\vr u)_t + \Div(\vr u \otimes u) + \Grad \overline{p(\vr)}- \Delta u = 0\quad \text{in }\mathcal{D}'([0,T)\times \Om),
	\end{split}
\end{equation}
where $\overline{p(\vr)}$ is the weak limit of $p(\vr_h)$.
\end{lemma}
\begin{proof}
Let $v \in [C_0^\infty([0,T)\times \Om)]^3$ be arbitrary and set $v_h = \Pi_h^V v$
as test-function in the momentum scheme \eqref{fem:moment}. After an application 
of Lemma \ref{lem:consistency}, we obtain the identity
\begin{equation}\label{mom:start}
	\begin{split}
		&\int_0^T\int_\Om D_t^h(\vr_h \ov u_h)v - \vr_h \tilde u_h \otimes \ov u_h:\Grad v~dxdt + F(v)\\
		&\qquad + \int_0^T\int_\Om \Grad_h u_h \Grad_h v_h - p(\vr_h)\Div_h v_h~dxdt \\
		&\qquad-h^{1-\eps}\sum_\Gamma \int_0^T\int_\Gamma \jump{\vr_h}_\Gamma \left(\frac{\ov u_+ + \ov u_-}{2}\right)\jump{\ov v_h}_\Gamma~dS(x)dt = 0
	\end{split}
\end{equation} 
Lemma \ref{lem:consistency} also provides the bound
\begin{equation}\label{mom:1}
	|F(v)| \leq h^\frac{1}{4}C\|\Grad v\|_{L^\infty(0,T;L^3(\Om))}.
\end{equation}
Moreover, from Lemma \ref{lem:stupid}, we have the bound
\begin{equation}\label{mom:2}
	\begin{split}
		&\left|h^{1-\eps}\sum_\Gamma \int_0^T\int_\Gamma \jump{\vr_h}_\Gamma \left(\frac{\ov u_+ + \ov u_-}{2}\right)\jump{\ov v_h}_\Gamma~dS(x)dt \right| \\
		&\qquad \leq h^\frac{13-6\eps}{12}C\|\Grad v\|_{L^\infty(0,T;L^3(\Om))}.
	\end{split}
\end{equation}
By applying summation by parts to the first term in \eqref{mom:start}, sending $h \rightarrow 0$ along
the subsequence for which $\vr_h \tilde u_h\otimes \ov u_h \weak \vr u\otimes u$ (Lemma \ref{lem:product}), 
and applying \eqref{mom:1} and \eqref{mom:2} we obtain
\begin{equation*}
	\int_0^T\int_\Om -\vr u v_t - \vr u \otimes u:\Grad v + \Grad u\Grad v - \overline{p(\vr)}\Div v~dxdt
	= \int_\Om m_0 v(0, \cdot)~dx,
\end{equation*}
which concludes our proof.

\end{proof}

\section{The discrete Laplace operator}\label{sec:hodge}
In the next section we  establish 
the most important ingredient in our
proof of compactness of the density approximation, namely 
weak continuity of the effective viscous flux.  
However, before we can embark on this proof, we will 
need to establish some properties related to our 
numerical Laplace operator.

In contrast to a standard continuous approximation 
scheme, our discrete  Laplace operator 
does not respect Hodge decompositions. 
More specifically, to prove the upcoming Proposition \ref{pro:effectiveflux}
we shall need to use test-functions  of the form
\begin{equation*}
	v = \mathcal{A}^\Grad[\vr],
\end{equation*}
where we for the purpose of this discussion does not require 
$v$ to satisfy boundary conditions. 
At the continuous level, testing with this $v$ is equivalent to 
applying $\mathcal{A}^\text{div}[\cdot]$ to the equation. In particular, 
since $\Curl \mathcal{A}^\Grad = 0$, $v$ satisfies
\begin{equation*}
	\int_\Om \Grad u \Grad v~dx = \int_\Om \Div u \vr~ dx.
\end{equation*}
This property is not true for our discrete Laplace operator. However, 
in the upcoming analysis it is essential that, at least,
\begin{equation}\label{eq:prop}
	\int_\Om \Grad_h u_h \Grad_h \left(\Pi_h^V \mathcal{A}^\Grad[\vr_h]\right)~dx = \int_\Om \Div_h u_h \vr_h~ dx + \mathcal{O}(h^\alpha), \quad \alpha > 0.
\end{equation}
This is the result that we will prove in this section. 

As will become evident, the property \eqref{eq:prop} is not trivially satisfied 
for our discretization. It is the extra "artifical" stabilization terms in 
the scheme (adding diffusion in all directions) that will provide 
the needed ingredient. In fact, the property \eqref{eq:prop} 
is the only reason for the presence of these terms. 
This  discretization strategy was devised by Eymard et. al for the stationary compressible 
Stokes equations \cite{Gallouet2}. In fact, most of the 
material contained in this section can be found there with only slight modifications 
to fit the present case.

In our proof, we shall need the
 operator $\Pi_h^L: Q_h(\Om) \mapsto \mathcal{P}_h(\Om)$ interpolating piecewise 
constant functions in the space of continuous piecewise linears $\mathcal{P}_h$ (Lagrange element space). 
This operator is defined by
\begin{equation*}
	\left(\Pi_h^L q_h\right)(\tau) = \frac{1}{\operatorname{card}(N_{\tau})}\sum_{E \in N_E}q_h|_E, 
\end{equation*}
for all vertices $\tau$ in the discretization, where $N_\tau$ is the set of elements 
having $\tau$ as a vertex. Note that shape regularity of $E_h$ renders the cardinality of $N_\tau$ bounded.
The following result can be found in \cite[Lemma 5.8]{Gallouet2}:
\begin{lemma}\label{lem:}
Let $q_h \in Q_h(\Om)$. There exists a constant $C$, depending only on 
the shape-regularity of $E_h$ such that
	\begin{align}
		\|\Grad \left(\Pi_h^L q_h\right)\|_{L^2(\Om)} &\leq C\left(\sum_\Gamma\int_\Gamma \frac{\jump{q_h}_\Gamma^2}{h}~dS(x) \right)^\frac{1}{2},\label{eq:pil1}\\
		\|\Pi_h^Lq_h - q_h\|_{L^2(\Om)} &\leq hC\left(\sum_\Gamma\int_\Gamma \frac{\jump{q_h}_\Gamma^2}{h}~dS(x) \right)^\frac{1}{2}.\label{eq:pil2}
	\end{align}
\end{lemma}

We shall need the following auxiliary result.
\begin{lemma}\label{lem:ortho}
Let $u_h \in V_h(\Om)$ and $v \in W^{2,2}(\Om)$ be arbitrary. Then, 
\begin{equation}\label{eq:curlid}
	\begin{split}
		&\int_\Om \Grad_h u_h \Grad v ~dx 
		 = \int_\Om \Curl_h u_h \Curl_h v + \Div_h u_h \Div v~dx + E(v, u_h).
	\end{split}
\end{equation}
Furthermore, there is a constant $C > 0$, depending only on the shape-regularity of $E_h$
such that
\begin{equation*}
	\left|E(v,u_h)\right| \leq h C\|\Grad_h u_h\|_{L^2(\Om)}\|\Grad^2 v\|_{L^{2}(\Om)}.
\end{equation*}

\end{lemma}
\begin{proof}
Using the Stoke's theorem and the identity 
$-\Delta = \Curl \Curl -\Grad \Div$,
\begin{equation*}
	\begin{split}
		& \int_\Om \Grad_h u_h \Grad v~dx\\
		&\quad = \sum_E \int_E - u_h \Delta v~dxdt + \int_0^T\int_{\partial E}(\Grad v\cdot \nu)u_h~dS(x)\\
		&\quad =\sum_E \int_E  u_h \Curl \Curl v - u_h \Grad \Div v~dxdt + \int_0^T\int_{\partial E}(\Grad v\cdot \nu)u_h~dS(x)\\
		&\quad =\int_\Om \Curl_h u_h \Curl v + \Div u_h \Div v~dx \\
		&\quad \quad + \sum_E\int_{\partial E}(\Curl v\times \nu)u_h \times \nu - (\Div v)u_h \cdot \nu + (\Grad v\cdot \nu)u_h~dS(x),
	\end{split}
\end{equation*}
which is \eqref{eq:curlid} with 
$$
E(v)=\sum_E\int_{\partial E}(\Curl v\times \nu)u_h \times \nu - (\Div v)u_h \cdot \nu + (\Grad v\cdot \nu)u_h~dS(x).
$$.

Next, we use that the trace of $\Grad v$ is continuous across edges to deduce
\begin{equation*}
	\begin{split}
		|E(v)| &\leq \sum_E\int_{\partial E}|\Grad v|u_h~dS(x) \\
		&= \sum_E \int_{\partial E}|\Grad v| (u_h - \Pi_h^Q u_h)~dS(x) \\
		&= \sum_E \int_{\partial E}|\Grad v - \Pi_h^Q(\Grad v)| (u_h - \Pi_h^Q u_h)~dS(x) \\
		&\leq h^{-\frac{1}{2}}C\left(\|\Grad v - \Pi_h^Q(\Grad v)\|_{L^2(\Om)} + h\|\Grad^2 v\|_{L^2(\Om)}\right) \\
		&\qquad \times h^{-\frac{1}{2}}\left(\| u_h - \Pi_h^Qu_h\|_{L^2(\Om)} + h\|\Grad_h u_h\|_{L^2(\Om)}\right) \\
		&\leq h C\|\Grad^2 v\|_{L^2(\Om)}\|\|\Grad u_h\|_{L^2(\Om)},
	\end{split}
\end{equation*}
where we have applied the trace theorem (Lemma \ref{lemma:toolbox}) and the 
interpolation error estimate (Lemma \ref{lemma:interpolationerror}). This concludes the proof.
\end{proof}

As in \cite{Gallouet2}, we obtain the following identity.

\begin{lemma}\label{lem:yes}
Let $(\vr_h, u_h)$ be the numerical solution obtained by Definition \ref{def:scheme} 
and let $\mathcal{A}^\Grad[\cdot]$ be given by \eqref{def:aop}. For any $\psi \in C^\infty([0,T)\times \Om)$
and $\phi \in C_0^\infty(\Om)$, 
\begin{equation*}
	\begin{split}
		& \int_0^T\int_\Om \Grad_h u_h \Grad_h \Pi_h^V \left(\psi \mathcal{A}^\Grad[\phi \vr_h] \right)~dxdt \\
		&\qquad = \int_0^T\int_\Om \Curl_h u_h \Grad \psi \times \mathcal{A}^\Grad[\phi \vr_h] + \Div_h u_h \Grad \psi \cdot \mathcal{A}^\Grad[\phi \vr_h]~dxdt\\
		&\quad \qquad + \int_0^T\int_\Om \phi \psi \Div_h u_h \vr_h~dxdt +  \mathcal{F}(\vr_h),
	\end{split}
\end{equation*}
where the $\mathcal{F}$ term is bounded as
\begin{equation}\label{eq:fb}
	\left|\mathcal{F}(\vr_h)\right| \leq h^{\frac{1}{2}{(\epsilon - \frac{1}{6})}}C,,
\end{equation}
with $C$ independent of $h$ and where we recall the requirement $\eps > \frac{1}{6}$.
\end{lemma}

\begin{proof}
To simplify notation, let
\begin{equation*}
	v = \psi \mathcal{A}^\Grad[\phi \vr_h], \qquad v_L = \psi \mathcal{A}^\Grad[\phi \Pi_h^L\vr_h],
\end{equation*}
and observe that linearity of $\mathcal{A}^\Grad$ provides the identity
\begin{equation*}\label{ort:id}
	v-v_L = \psi \mathcal{A}^\Grad[\phi(\vr_h - \Pi_h^L \vr_h)].
\end{equation*}
By using Lemma \ref{lem:amazing} and adding and subtracting $v_L$, we deduce
\begin{equation*}
	\begin{split}
		&\int_0^T\int_\Om \Grad_h u_h \Grad_h \Pi_h^V v~dxdt \\
		&\qquad= \int_0^T\int_\Om \Grad_h u_h \Grad  v~dxdt \\
		&\qquad = \int_0^T\int_\Om \Grad_h u_h \Grad v_L~dxdt + \int_0^T\int_\Om \Grad_h u_h \Grad (v - v_L)~dxdt \\
		&\qquad = \int_0^T\int_\Om \Curl_h u_h \Curl v_L + \Div_h u_h \Div v_L~dxdt \\
		&\qquad\qquad + \int_0^T\int_\Om \Grad_h u_h \Grad (v - v_L)~dxdt + E(v_L, u_h),
	\end{split}
\end{equation*}
where the last equality is Lemma \ref{lem:ortho}. Next, we once more add and subtract $v_l$
to obtain
\begin{equation}\label{ort:start}
	\begin{split}	
		&\int_0^T\int_\Om \Grad_h u_h \Grad_h \Pi_h^V v~dxdt \\
		&\qquad = \int_0^T\int_\Om \Curl_h u_h \Curl v + \Div_h u_h \Div v~dxdt \\
		&\qquad\qquad +\int_0^T\int_\Om \Curl_h u_h \Curl (v_L-v) + \Div_h u_h \Div (v_L-v)~dxdt \\
		&\qquad\qquad + \int_0^T\int_\Om \Grad_h u_h \Grad (v - v_L)~dxdt + E(v_L, u_h)\\
		&\qquad  = \int_0^T\int_\Om \Curl_h u_h \Grad \psi \times \mathcal{A}^\Grad[\phi \vr_h] + \Div_h u_h \Grad \psi \cdot \mathcal{A}^\Grad[\phi \vr_h]~dxdt\\
		&\quad \qquad + \int_0^T\int_\Om \phi \psi \Div_h u_h \vr_h~dxdt +  \mathcal{F}(\vr_h),
	\end{split}
\end{equation}
where we have used the definition of $v_l$ and introduced the quantity
\begin{equation*}
	\begin{split}
		\mathcal{F}(\vr_h) 
		&:= \int_0^T\int_\Om \Curl_h u_h \Curl (v_L-v) + \Div_h u_h \Div (v_L-v)~dxdt \\
		&\qquad \int_0^T\int_\Om \Grad_h u_h \Grad (v - v_L)~dxdt + E(v_L, u_h).
	\end{split}
\end{equation*}
In view of \eqref{ort:start}, it only remains to prove \eqref{eq:fb}.

Some applications of the Cauchy-Schwartz inequality together with 
the properties of $\mathcal{A}^\Grad[\cdot]$ gives
\begin{equation*}
	\begin{split}
		|\mathcal{F}(\vr_h)| &\leq C\|\Grad_h u_h\|_{L^2(\Om)}\|\phi(\vr_h - \Pi_h^L \vr_h)\|_{L^2(0,T;L^2(B))} + |E(v_L, u_h)| \\
		&\leq hC\left(\sum_{\Gamma \subset B} \int_0^T\int_{\Gamma} \frac{\jump{\vr_h}^2_\Gamma}{h}~dS(x)dt \right)^\frac{1}{2}+ h\|\Grad^2 v\|_{L^2(0,T;L^2(\Om))} \\
		&\leq \sqrt{h}2C\left(\sum_{\Gamma \subset B} \int_0^T\int_\Gamma \jump{\vr_h}^2_\Gamma~dS(x)dt \right)^\frac{1}{2}.
	\end{split}
\end{equation*}
where the second last inequality is \eqref{eq:pil2} and Lemma \ref{lem:ortho} and the last inequality is 
and application of \eqref{eq:pil2}. To bound the last term, we invoke \eqref{stupid:2} yielding
\begin{equation*}
	|\mathcal{F}(\vr_h)| \leq h^{\frac{1}{2}{(\epsilon - \frac{1}{6})}}C,
\end{equation*}
which is \eqref{eq:fb}

\end{proof}

\section{The effective viscous flux}\label{sec:flux}

The main tool that will allow 
us to conclude strong convergence 
of the density is a remarkable 
result  discovered by P.~L.~Lions \cite{Lions:1998ga}
for a continuous approximation scheme. 
The result states that the effective viscous flux
\begin{equation*}
	\Div u - p(\vr),
\end{equation*}
behaves as if it is converging strongly.
More specific, in our numerical setting, the result 
goes as follows:

\begin{proposition}\label{pro:effectiveflux}
Let $(\vr_h, u_h)$ be the numerical solution 
obtained using Definition \ref{def:scheme} and \eqref{def:schemeII}.  
Moreover, let $(\vr, u)$ be given through the 
convergences \eqref{eq:conv1}. Then,
\begin{equation*}\label{eq:effective}
	\begin{split}
		&\lim_{h \rightarrow 0}\int_0^T\int_\Om \phi \psi\left(\Div_h u_h - p(\vr_h)\right)\vr_h~dxdt \\
		&\qquad =\int_0^T\int_\Om \phi \psi\left(\Div u - \overline{ p(\vr)}\right)\vr~dxdt,
	\end{split}
\end{equation*}	
for all $\phi \in C_0^\infty(\Om)$ and $\psi \in C_0^\infty([0,T)\times \Om)$.
\end{proposition}
Hence, the product of the weakly converging effective viscous flux 
and the weakly converging density converges to the product of 
the weak limits. Our proof of this proposition 
will rely on a number of auxiliary results which we will prove 
first. The proof is concluded in Section \ref{sec:prop}.

\subsection{The numerical commutator estimate}
In the upcoming analysis, the following lemma will be essential. A proof 
based on the div-curl lemma can be found in \cite{Feireisl}.
\begin{lemma}\label{lem:divcurl}
Let $v_n$ and $w_n$ be sequences of vector valued functions such that $v_n \weak v$ in $L^p(\Om)$ and $w_n \weak w$ in $L^q(\Om)$, 
respectively,  where 
$1<p,q<\infty$ and $\frac{1}{p}+ \frac{1}{q} \leq 1$. Moreover, let $B_n \weak B$ in $L^p(\Om)$.
Then, 
\begin{equation*}
	\begin{split}
		&(1)\quad v_n\Grad \mathcal{A}^\text{div}[w_n] - w_n \Grad \mathcal{A}^\text{div}[v_n]
		\quad \longrightarrow\quad v\Grad \mathcal{A}^\text{div}[w] - w\Grad \mathcal{A}^\text{div}[v]\\
		&(2)\quad B_n\Grad \mathcal{A}^\text{div}[w_n] - w_n \Grad \mathcal{A}^\nabla[B_n]
		\quad \longrightarrow\quad B\Grad \mathcal{A}^\text{div}[w] - w\Grad \mathcal{A}^\nabla[B]\\
	\end{split}
\end{equation*}
in the sense of distributions on $\Om$.
\end{lemma}

\begin{lemma}\label{lem:fundamental}
Given the convergences \eqref{eq:conv1} - \eqref{eq:conv2},
\begin{equation*}
	\begin{split}
		&\lim_{h \rightarrow 0}\int_0^T\int_\Om \phi \vr_h \tilde u_h\Grad \left(\mathcal{A}^\text{div}[\psi \vr_h \widehat u_h] \right) 
			- \psi\vr_h \tilde u_h \otimes \widehat u_h:\Grad \left(\mathcal{A}^\Grad[\phi \vr_h]\right)~dxdt \\
		&\qquad= \int_0^T\int_\Om \phi \vr  u\Grad \left(\mathcal{A}^\text{div}[\psi \vr  u] \right) - \psi\vr  u \otimes  u:\Grad \left(\mathcal{A}^\Grad[\phi \vr]\right)~dxdt.
	\end{split}
\end{equation*}

\end{lemma}
\begin{proof}
We begin by observing the following identity
	\begin{equation}\label{eq:J11}
		\begin{split}
			 & \int_0^T\int_\Om \phi\vr_h \tilde u_h\Grad \left(\mathcal{A}^\text{div}[\psi \vr_h \widehat u_h] \right) - \psi\vr_h \tilde u_h \otimes \widehat u_h:\Grad \left( \mathcal{A}^\Grad[\phi \vr_h]\right)~dxdt \\
			&=\int_0^T\int_\Om \tilde u_h \cdot \left[\Grad \left(\mathcal{A}^\text{div}[\psi \vr_h \widehat u_h] \right)\phi\vr_h - \Grad \left(\mathcal{A}^\Grad[\phi \vr_h]\right)\cdot \psi \vr_h \widehat u_h\right]~dxdt \\
			&=:\int_0^T\int_\Om \tilde u_h \mathcal{H}^h~dxdt, 
		\end{split}
	\end{equation}
	where we have introduced $\mathcal{H}^h$ given by
	$$
	\mathcal{H}^h = \Grad \left(\mathcal{A}^\text{div}[\psi \vr_h \widehat u_h] \right)\phi\vr_h - \Grad \left(\mathcal{A}^\Grad[\phi \vr_h]\right)\cdot \psi \vr_h \widehat u_h.
	$$

	From Lemma \ref{lem:product}, we have that 
	\begin{equation}\label{spec:w1}
		\begin{split}
			\vr_h &\weak \vr \text{ in $C(0,T;L^\gamma(\Om))\cap L^2(0,TL^\gamma(\Om))$}, \\
			\vr_h \widehat u_h &\weak \vr u \text{ in  $C(0,T;L^\frac{2\gamma}{\gamma+1}(\Om))\cap L^2(0,T;L^3(\Om))$}.
		\end{split}
	\end{equation}
	Since in addition $\gamma > 3$, 
	the H\"older inequality, Lemma \ref{lem:duality}, and Corollary \ref{cor:energy}, provides the estimate
	\begin{equation}\label{spec:w2}
		\begin{split}
			&\|\mathcal{H}^h\|_{L^2(0,T;L^\frac{3}{2}(\Om))}
			\leq C\|\vr_h \widehat u_h\|_{L^2(0,T;L^3(\Om))}\|\vr_h\|_{L^\infty(0,T;L^\gamma(\Om))} \leq C,
		\end{split}
	\end{equation}
	and similarly, 
	\begin{equation}\label{spec:w3}
		\begin{split}
			&\|\mathcal{H}^h\|_{L^\infty(0,T;L^{\frac{2\gamma}{3+\gamma}}(\Om))} 
			\leq C\|\vr_h \widehat u_h\|_{L^\infty(0,T;L^\frac{2\gamma}{\gamma+1}(\Om))}\|\vr_h\|_{L^\infty(0,T;L^\gamma(\Om))} \leq C,
		\end{split}
	\end{equation}
	where $\frac{2\gamma}{\gamma+1} > 1$ since $\gamma > 3$.
	
	By virtue of \eqref{spec:w1}-\eqref{spec:w3}, we can apply Lemma \ref{lem:divcurl}, with $v_h = \psi \vr_h \widehat u_h$
	and $B_h = \phi \vr_h$, to conclude the weak convergence
	\begin{equation}\label{def:HH}
		\begin{split}
			\mathcal{H}^h &= \Grad \left(\mathcal{A}^\text{div}[\psi \vr_h \widehat u_h] \right)\phi\vr_h- \Grad \left(\mathcal{A}^\Grad[\phi \vr_h]\right)\cdot \psi \vr_h \widehat u_h \\
			&\qquad\qquad \weak \Grad \left(\mathcal{A}^\text{div}[\psi \vr  u] \right)\phi\vr- \Grad \left(\mathcal{A}^\Grad[\phi \vr]\right)\cdot \psi \vr  u 
			=: \mathcal{H},
		\end{split}
	\end{equation}
	in $L^2(0,T;L^\frac{3}{2}(\Om)) \cap C(0,T;L^\frac{2\gamma}{3+\gamma}(\Om))$ as $h \rightarrow 0$. 

	To proceed we will need the standard mollifier which we will denote by $R^\delta$.
 	It will be a standing assumption throughout that $\delta$ is sufficiently small. 
	
	Now, by adding and subtracting we write
	\begin{equation}\label{spec:1}
		\begin{split}
		&\int_0^T\int_\Om \tilde u_h \mathcal{H}^h~dxdt \\
		&\qquad = \int_0^T\int_\Om (\tilde u_h -u_h)\mathcal{H}^h +  ( u_h - R^\delta \star u_h)\mathcal{H}^h + (R^\delta \star u_h)\mathcal{H}^h~dxdt \\
		\end{split}
	\end{equation}
	The first term in \eqref{spec:1} converges to zero as
	\begin{equation}\label{spec:2}
		\begin{split}
			&\int_0^T\int_\Om (\tilde u_h -u_h)\mathcal{H}^h~dxdt  \\
			&\qquad\qquad \leq \|\tilde u_h - u_h\|_{L^2(0,T;L^3(\Om))}\|\mathcal{H}^h\|_{L^2(0,T;L^\frac{3}{2}(\Om))} \\
			&\qquad\qquad \leq h\|\Grad_h u_h\|_{L^2(0,T;L^3(\Om))}\|\mathcal{H}^h\|_{L^2(0,T;L^\frac{3}{2}(\Om))} \\
			&\qquad\qquad \leq h^\frac{1}{2}\|\Grad_h u_h\|_{L^2(0,T;L^2(\Om))}\|\mathcal{H}^h\|_{L^2(0,T;L^\frac{3}{2}(\Om))} 
			\leq h^\frac{1}{2}C.
		\end{split}
	\end{equation}
	 To bound the second term in \eqref{spec:1}, we apply the H\"older inequality and the space translation estimate of Lemma \ref{lem:translation}, 
	\begin{equation}\label{spec:3}
		\begin{split}
			&\int_0^T\int_\Om  ( u_h - R^\delta \star u_h)\mathcal{H}^h~dxdt \\
			&\qquad\qquad \leq \left\|u_h - R^\delta \star u_h\right\|_{L^2(0,T;L^3(\Om))}\|\mathcal{H}^h\|_{L^2(0,T;L^\frac{3}{2}(\Om))} \\
			&\qquad\qquad \leq C\left(h^2 + \delta^2\right)^\frac{1}{4}\|\Grad_h u_h\|_{L^2(0,T;L^2(\Om))}
			\leq  C\left(h^2 + \delta^2\right)^\frac{1}{4}.
		\end{split}
	\end{equation}
	Next, since $\mathcal{H}^h \weak \mathcal{H}$ in $C(0,T;L^\frac{2\gamma}{3+\gamma}(\Om))$, we have in particular that 
	$$
		\mathcal{H}^h \rightarrow \mathcal{H} \text{ in $L^2(0,T;W^{-1, p}(\Om))$}, \quad p < \frac{3}{2}.
	$$
	Hence, for each fixed $\delta$, we can conclude that
	\begin{equation*}
		\lim_{h \rightarrow 0}\int_0^T\int_\Om (R^\delta \star u_h)\mathcal{H}^h~dxdt = \int_0^T\int_\Om (R^\delta \star u)\mathcal{H}~dxdt.
	\end{equation*}
 	
	Finally, we pass to the limit in \eqref{spec:1} and apply \eqref{spec:2} and \eqref{spec:3} to obtain
	\begin{equation*}
		\lim_{h \rightarrow 0}\int_0^T\int_\Om \tilde u_h \mathcal{H}^h~dxdt = \int_0^T\int_\Om (R^\delta \star u)\mathcal{H}~dxdt + \mathcal{O}(\sqrt{\delta}).
	\end{equation*}
	We conclude the proof by sending $\delta \rightarrow 0$ and recalling \eqref{eq:J11} and \eqref{def:HH}.
	
\end{proof}

\begin{lemma}\label{lem:commutator}
Let $(\vr_h, u_h)$ be the numerical solution obtained
through Definition \ref{def:scheme} and \eqref{def:schemeII}. Let $(\vr, u)$ be the corresponding weak limit pair 
obtained from the convergences \eqref{eq:conv1}. Then, for any $\phi \in C_0^\infty(\Om)$
and $\psi \in C_0^\infty([0,T)\times \Om)$,
\begin{equation*}
	\begin{split}
		&\lim_{h \rightarrow 0}\left(\sum_\Gamma \int_0^T\int_\Gamma\up(\vr u)\jump{\Pi_h^Q \left[\phi\mathcal{A}^\text{div}[\psi \vr_h \widehat u_h]\right]}_\Gamma \right.\\
		&\qquad \qquad \qquad\left. - \up(\vr u \otimes u)\jump{\Pi_h^Q\Pi_h^V \left[\psi \mathcal{A}^\Grad[\phi \vr_h]\right]}_\Gamma~dS(x)dt \right)	\\
		&= \int_0^T \int_\Om \vr u \Grad (\phi \mathcal{A}^\text{div}(\psi \vr u)) - \vr u \otimes u:\Grad \left(\psi \mathcal{A}^\Grad[\phi \vr]\right)~dxdt,
	\end{split}
\end{equation*}
where the operators $\mathcal{A}^\Grad[\cdot]$ and $\mathcal{A}^\text{div}[\cdot]$ are given by \eqref{def:aop}.

\end{lemma}
\begin{proof}
Our starting point is provided by Lemma \ref{lem:transport} which in 
this case gives
\begin{equation}\label{com:start}
	\begin{split}
		&\sum_\Gamma \int_0^T\int_\Gamma\up(\vr u)\jump{\Pi_h^Q \left[\mathcal{A}^\text{div}[\psi \vr_h \widehat u_h]\right]}_\Gamma \\
		&\qquad \qquad \qquad - \up(\vr u \otimes u)\jump{\Pi_h^Q\Pi_h^V \left[\psi \mathcal{A}^\Grad[\phi \vr_h]\right]}_\Gamma~dS(x)dt \\
		& = \int_0^T\int_\Om \vr_h \widetilde u_h\Grad \left(\phi\mathcal{A}^\text{div}[\psi \vr_h \widehat u_h] \right) - \vr_h \widetilde u_h \otimes \widehat u_h:\Grad \left(\psi \mathcal{A}^\Grad[\phi \vr_h]\right)~dxdt \\
		&\qquad \qquad + P_1\left(\phi\mathcal{A}^\text{div}[\psi \vr_h \widehat u_h] \right) + \sum_{i=2}^4P_i\left(\psi \mathcal{A}^\Grad[\phi \vr_h]]\right).
	\end{split}
\end{equation}
Let us first prove that the $P_i$ terms are converging to zero 
as $h \rightarrow 0$. 

\vspace{0.2cm}
\noindent
\textit{1. Convergence of $P_i$ to zero:}
From \eqref{eq:p1}, we have the following bound 
\begin{equation}\label{com:p11}
	\begin{split}
		&\left|P_1\left(\phi\mathcal{A}^\text{div}[\psi \vr_h \widehat u_h] \right)\right|  \\
		&\qquad \leq h^{\frac{1}{2} - 3\frac{4-\gamma}{4\gamma}}C\left\|\Grad \left(\phi\mathcal{A}^\text{div}[\psi \vr_h \widehat u_h] \right)\right\|_{L^4(0,T;L^3(\Om))} \\
		&\qquad \leq h^{\frac{1}{2} - 3\frac{4-\gamma}{4\gamma}}C\|\phi\|_{W^{1,\infty}(\Om)}\left\||\mathcal{A}^\text{div}[\psi \vr_h \widehat u_h]\right\|_{L^4(0,T;W^{1,3}(\Om))}.
	\end{split}
\end{equation}
From the regularization property of $\mathcal{A}^\text{div}$ (Lemma \ref{lem:duality}) 
together with an inverse estimate in time (Lemma \ref{lemma:inverse}) we deduce
\begin{equation}\label{com:p12}
	\begin{split}
		&\left\||\mathcal{A}^\text{div}[\psi \vr_h \widehat u_h]\right\|_{L^4(0,T;W^{1,3}(\Om))} \\
		&\qquad \leq C\|\psi \vr_h \widehat u_h\|_{L^4(0,T;L^3(\Om))} \\
		&\qquad \leq C\|\psi\|_{L^\infty((0,T)\times \Om)}\|\vr_h \widehat u_h\|_{L^4(0,T;L^3(\Om))} \\
		&\qquad \leq (\Delta)^{-\frac{1}{4}}C\|\psi\|_{L^\infty((0,T)\times \Om)}\|\vr_h \widehat u_h\|_{L^2(0,T;L^3(\Om))},
	\end{split}
\end{equation}
where the last term is bounded by the energy (Corollary \ref{cor:energy}).
By setting \eqref{com:p12} in \eqref{com:p11} and using that $\Delta t = a h$, 
we obtain
\begin{equation}\label{com:P1}
	\left|P_1\left(\phi\mathcal{A}^\text{div}[\psi \vr_h \widehat u_h] \right)\right|
	\leq h^{\frac{1}{4} - 3\frac{4-\gamma}{4\gamma}}C,
\end{equation}
where the exponent for $h$ is strictly positive as $\gamma > 3$.

For the remaining $P^i$ terms, we apply \eqref{eq:pcombined} in Lemma \ref{lem:transport}
to obtain
\begin{equation}\label{com:P2}
	\begin{split}
		&\sum_{i=2}^4 \left|P_i\left(\psi \mathcal{A}^\Grad[\phi \vr_h]]\right)\right| \\
		&\qquad \leq h^\frac{1}{4}C\left\|\Grad\left(\psi \mathcal{A}^\Grad[\phi \vr_h]]\right)\right\|_{L^\infty(0,T;L^\gamma(\Om))} \\
		&\qquad \leq h^\frac{1}{4}C\|\psi\|_{L^\infty(0,T;W^{1,\infty}(\Om))}\|\phi\|_{L^\infty(\Om)}\|\vr_h\|_{L^\infty(0,T;L^\gamma(\Om))} \leq h^\frac{1}{4}C,
	\end{split}
\end{equation}
where we in the second last inequality have applied Lemma \ref{lem:duality}. The last inequality follows 
from $p(\vr_h) \inb L^\infty(0,T;L^1(\Om))$.

Consequently, \eqref{com:P1} and \eqref{com:P2} allow us to conclude that 
\begin{equation}\label{com:Pconv}
\lim_{h \rightarrow 0}\left(\left|P_1\left(\phi\mathcal{A}^\text{div}[\psi \vr_h \widehat u_h] \right)\right| +	\sum_{i=2}^4 \left|P_i\left(\psi \mathcal{A}^\Grad[\phi \vr_h]]\right)\right|\right)
= 0.
\end{equation}

\vspace{0.3cm}
\noindent
\textit{2. Convergence of the commutator term:}
To prove convergence of the first term 
after the equality in \eqref{com:start} we will need 
to rewrite this term on a form for which Lemma \ref{lem:fundamental} 
is applicable:
\begin{align}\label{com:com1}
		&\int_0^T\int_\Om \vr_h \tilde u_h\Grad \left(\phi\mathcal{A}^\text{div}[\psi \vr_h \widehat u_h] \right) 
			- \vr_h \tilde u_h \otimes \widehat u_h:\Grad \left(\psi \mathcal{A}^\Grad[\phi \vr_h]\right)~dxdt \nonumber\\
		&=\int_0^T\int_\Om \phi \vr_h \tilde u_h\Grad \left(\mathcal{A}^\text{div}[\psi \vr_h \widehat u_h] \right) 
			- \psi\vr_h \tilde u_h \otimes \widehat u_h:\Grad \left(\mathcal{A}^\Grad[\phi \vr_h]\right)~dxdt \\
		& \quad + \int_0^T\int_\Om\vr_h \tilde u_h \Grad \phi \cdot \mathcal{A}^\text{div}[\psi \vr_h \widehat u_h] -
		 \vr_h \tilde u_h \otimes \widehat u_h:\Grad \psi \otimes \mathcal{A}^\Grad[\phi \vr_h]~dxdt.\nonumber
\end{align}
Now,  observe that the first term after the equality is precisely  the one covered by Lemma \ref{lem:fundamental}. Hence,
\begin{equation}\label{com:com2}
	\begin{split}
		&\lim_{h \rightarrow 0}\int_0^T\int_\Om \phi \vr_h \widetilde u_h\Grad \left(\mathcal{A}^\text{div}[\psi \vr_h \widehat u_h] \right)
		 	- \psi\vr_h \widetilde u_h \otimes \widehat u_h:\Grad \left(\mathcal{A}^\Grad[\phi \vr_h]\right)~dxdt \\
		&\qquad= \int_0^T\int_\Om \phi \vr  u\Grad \left(\mathcal{A}^\text{div}[\psi \vr  u] \right) - \psi\vr  u \otimes  u:\Grad \left(\mathcal{A}^\Grad[\phi \vr]\right)~dxdt.
	\end{split}
\end{equation}
Moreover, in view of the convergences in Lemma \ref{lem:product}, we have that
\begin{equation*}
	\begin{split}
				\mathcal{A}^\Grad[\phi \vr_h] &\rightarrow \mathcal{A}^\Grad[\phi \vr]~ \text{in $C(0,T;L^p(\Om))$ for any $p < \infty$}\\
			\mathcal{A}^\text{div}[\psi \vr_h \widehat u_h] &\rightarrow \mathcal{A}^\text{div}(\psi \vr u) ~ \text{ in $C(0,T;L^3(\Om))$}.
	\end{split}
\end{equation*}
As a consequence, there is no problems with concluding that
\begin{equation}\label{com:com3}
	\begin{split}
		&\lim_{h \rightarrow 0}
		\int_0^T\int_\Om\vr_h \widetilde u_h \Grad \phi \cdot \mathcal{A}^\text{div}[\psi \vr_h \widehat u_h] 
			- \vr_h \widetilde u_h \otimes \widehat u_h:\Grad \psi \otimes \mathcal{A}^\Grad[\phi \vr_h]~dxdt \\
		&= \int_0^T\int_\Om \vr u \Grad \phi \cdot \mathcal{A}^\text{div}[\psi \vr u] - \vr \tilde u \otimes  u:\Grad \psi \otimes \mathcal{A}^\Grad[\phi \vr]~dxdt.
	\end{split}
\end{equation}
Finally, we send $h \rightarrow 0$ in \eqref{com:com1} and apply \eqref{com:com2} and \eqref{com:com3} to conclude
\begin{align}\label{com:com4}
		&\lim_{h \rightarrow 0}\int_0^T\int_\Om \vr_h \widetilde u_h\Grad \left(\phi\mathcal{A}^\text{div}[\psi \vr_h \widehat u_h] \right) - \vr_h \widetilde u_h \otimes \widehat u_h:\Grad \left(\psi \mathcal{A}^\Grad[\phi \vr_h]\right)~dxdt \nonumber\\
		&=\int_0^T\int_\Om \phi \vr  u\Grad \left(\mathcal{A}^\text{div}[\psi \vr  u] \right) - \psi\vr  u \otimes  u:\Grad \left(\mathcal{A}^\Grad[\phi \vr]\right)~dxdt \nonumber \\
		&\quad \qquad \int_0^T\int_\Om \vr u \Grad \phi \cdot \mathcal{A}^\text{div}[\psi \vr u] - \vr \tilde u \otimes  u:\Grad \psi \otimes \mathcal{A}^\Grad[\phi \vr]~dxdt \\
		&= \int_0^T \int_\Om \vr u \Grad (\psi \mathcal{A}^\text{div}(\phi \vr u)) - \vr u \otimes u:\Grad \left(\psi \mathcal{A}^\Grad[\phi \vr]\right)~dxdt.\nonumber
\end{align}

\vspace{0.3cm}
\noindent
\textit{3. Conclusion:}
We conclude the proof  by sending $h \rightarrow 0$ in \eqref{com:start}
and applying \eqref{com:Pconv} and \eqref{com:com4}.

\end{proof}

\subsection{Proof of Proposition \ref{pro:effectiveflux}}\label{sec:prop}
Define the test-function 
\begin{equation*}
	v = \psi \mathcal{A}^\Grad[\phi \vr_h], \qquad v_h = \Pi_h^V v.
\end{equation*}
By setting $v_h$ as test-function in the momentum scheme \eqref{fem:moment}, 
integrating in time,  applying Lemma \ref{lem:yes}
for the term involving $\Grad_h u_h\Grad_h v_h$,
the calculation \eqref{eq:thecalc} for the term involving the pressure,
and reordering terms, we obtain
\begin{equation}\label{eff:start}
	\begin{split}
		\int_0^T\int_\Om \phi \psi(p(\vr_h) -  \Div_h u_h )\vr_h~dxdt = \sum_{i=1}^4 J_i^h + \mathcal{F}(\vr_h),
	\end{split}
\end{equation}
where $\mathcal{F}(\vr_h)$ is given by \eqref{eq:fb} and
\begin{align*}
	J^h_1 &= - \int_0^T\int_\Om p(\vr_h)\Grad \phi \cdot \mathcal{A}^\Grad[\phi \vr_h]~dxdt,\\
	J^h_2 &= \int_0^T\int_\Om \Curl_h u_h \Grad \psi \times \mathcal{A}^\Grad[\phi \vr_h] + \Div_h u_h \Grad \psi \cdot \mathcal{A}^\Grad[\phi \vr_h]~dxdt, \\
	J^h_3 &= \int_0^T\int_\Om D_t^h(\vr_h \ov u_h)v_h~dxdt-\sum_\Gamma \int_0^T \int_{\Gamma}\up(\vr  u \otimes \ov u)\jump{\ov v_h}_\Gamma~dS(x)dt, \\
	J^h_4 &= h^{1-\eps}\sum_\Gamma \int_0^T\int_\Gamma \left(\frac{\ov u_+ +\ov u_-}{2} \right)\jump{\vr_h}_\Gamma\jump{\ov v_h}_\Gamma~dS(x)dt, 
\end{align*}

Next, define the test-function
\begin{equation*}
	w = \psi \mathcal{A}^\Grad [\phi \vr].
\end{equation*}
Observe that this test-function is the limit of $v$. That is, 
\begin{equation*}\label{eq:v2w}
	v \rightarrow w \quad \text{a.e as $h \rightarrow 0$}.
\end{equation*}

Setting $w$ as test-function in the weak limit of the momentum scheme \eqref{eq:limbar}, 
and reordering terms, yields
\begin{equation}\label{eff:end}
	\begin{split}
		\int_0^T\int_\Om \psi \phi (\overline{p(\vr)} - \Div u)\vr~dxdt = J_1 + J_2 + J_3.
	\end{split}
\end{equation}
\begin{align*}
	J_1 &= - \int_0^T\int_\Om \overline{p(\vr)}\Grad \phi \cdot \mathcal{A}^\Grad[\phi \vr]~dxdt,\\
	J_2 &= \int_0^T\int_\Om \Curl u \Grad \psi \times \mathcal{A}^\Grad[\phi \vr] + \Div u \Grad \psi \cdot \mathcal{A}^\Grad[\phi \vr]~dxdt \\
	J_3 &= -\int_0^T\int_\Om \vr u w_t + \vr u \otimes u:\Grad w~dxdt - \int_\Om m_0 w(0,\cdot)~dx.
\end{align*}

Now, observe that the proof of Proposition \ref{pro:effectiveflux} is complete once we prove
$$
J_1^h \rightarrow J_1, \quad J_2^h \rightarrow J_2, \quad J_3^h \rightarrow J_3, \quad J_4^h \rightarrow 0.
$$

Since we have already established the convergences
\begin{equation*}
	\begin{split}
				\mathcal{A}^\Grad[\phi \vr_h] &\rightarrow \mathcal{A}^\Grad[\phi \vr]~ \text{in $C(0,T;L^p(\Om))$ for any $p < \infty$}\\
			\mathcal{A}^\text{div}[\psi \vr_h \widehat u_h] &\rightarrow \mathcal{A}^\text{div}(\psi \vr u) ~ \text{ in $C(0,T;L^3(\Om))$},
	\end{split}
\end{equation*}
we can immediately conclude that
\begin{equation}\label{eq:J1J2}
	J_1^h \rightarrow J_1 \qquad J_2^h \rightarrow J_2.
\end{equation}
For the $J_4^h$ term, the calculation \eqref{high:i2} gives
\begin{equation}\label{eff:j4}
	|J_4^h| \leq h^\frac{1}{2}C \rightarrow 0.
\end{equation}
Hence, the only remaining ingredient in the proof of Proposition \ref{pro:effectiveflux}, 
is to establish $J_3^h \rightarrow J_3$. 
Indeed, let us for the moment take this result as granted (Lemma \ref{lem:j3} below).
Then, we can send $h \rightarrow 0$ 
in \eqref{eff:start}, using \eqref{eq:fb}, \eqref{eq:J1J2}, 
\eqref{eff:j4}, to discover
\begin{equation*}\label{eff:startis}
	\begin{split}
		&\lim_{h \rightarrow 0}\int_0^T\int_\Om \phi \psi(p(\vr_h) -  \Div_h u_h )\vr_h~dxdt \\
		&\qquad = \lim_{h \rightarrow 0}\sum_{i=1}^4 J_i^h + \mathcal{F}(\vr_h) 
		 = J_1 + J_2 + J_3 \\
		&\qquad=\int_0^T\int_\Om \phi \psi(\overline{p(\vr)} -  \Div_h u )\vr~dxdt,
	\end{split}
\end{equation*}
where the last equality is \eqref{eff:end}. This is precisely Proposition \ref{pro:effectiveflux}. 

Hence, the proof is complete once we establish 
the following lemma.

\begin{lemma}\label{lem:j3}
Under the conditions of Proposition \ref{pro:effectiveflux},
\begin{equation*}
	\lim_{h \rightarrow 0}J_3^h = J_3,
\end{equation*}
where $J_3^h$ and $J_3$ are given by \eqref{eff:start} and \eqref{eff:end}, respectively.
\end{lemma}

\begin{proof}[Proof of Lemma \ref{lem:j3}]
To prove convergence of $J_3^h$, we shall need to rewrite the 
time derivative term in $J_3^h$
using the continuity scheme.
By adding and subtracting, and applying summation by parts, we deduce
\begin{align}
		&\int_0^{T}\int_\Om D_t^h(\vr_h \ov u_h)v_h~dxdt \nonumber \\
		& \quad = - \int_0^{T-\Delta t}\int_\Om \vr_h\ov u_h D_t^h(v(+\Delta t))~dxdt \nonumber\\
		&\qquad \quad \qquad - \int_\Om \vr_h^0\ov u_h^0 v(\Delta t)~dx + \int_0^{T}\int_\Om D_t^h(\vr_h \ov u_h)(v_h-v)~dxdt\nonumber\\
		&\quad  = - \int_{\Delta t}^{T}\int_\Om (\vr_h\ov u_h)(\cdot -\Delta t)~ D_t^h\left(\psi\mathcal{A}^\Grad\left[\phi \vr_h\right]\right)~dxdt\nonumber\\
		&\qquad \quad \qquad - \int_\Om \vr_h^0\ov u_h^0 v(\Delta t)~dx + \int_0^{T}\int_\Om D_t^h(\vr_h \ov u_h)(v_h-v)~dxdt\nonumber\\
		&\quad = - \int_{\Delta t}^{T}\int_\Om (\vr_h\ov u_h\psi)(\cdot -\Delta t)~ \mathcal{A}^\Grad\left[\phi D_t^h\vr_h\right]~dxdt\nonumber\\
		&\qquad \quad \qquad -\int_{\Delta t}^{T}\int_\Om (\vr_h\ov u_h)(\cdot -\Delta t)~ \mathcal{A}^\Grad\left[\phi\vr_h\right]D_t^h\psi ~dxdt \nonumber\\
		&\qquad \quad \qquad- \int_\Om \vr_h^0\ov u_h^0 v(\Delta t)~dx + \int_0^{T}\int_\Om D_t^h(\vr_h \ov u_h)(v_h-v)~dxdt\nonumber
\end{align}
Next, we apply the integration by parts formula for $\mathcal{A}^\Grad[\cdot]$ (Lemma \ref{lem:duality}) 
\begin{align*}
		&\int_0^{T}\int_\Om D_t^h(\vr_h \ov u_h)v_h~dxdt \nonumber \\
		&  \qquad =  \int_{\Delta t}^{T}\int_\Om \mathcal{A}^\text{div}[(\vr_h\ov u_h\psi)(\cdot -\Delta t)] \phi D_t^h\vr_h \nonumber\\
		&\qquad\qquad   -\int_{\Delta t}^{T}\int_\Om (\vr_h\ov u_h)(\cdot -\Delta t)~ \mathcal{A}^\Grad\left[\phi\vr_h\right]D_t^h\psi ~dxdt \nonumber\\
		&\qquad\qquad  - \int_\Om \vr_h^0\ov u_h^0 v(\Delta t)~dx + \int_0^{T}\int_\Om D_t^h(\vr_h \ov u_h)(v_h-v)~dxdt\nonumber.
\end{align*}
We then rewrite the first term using the continuity scheme \eqref{fem:cont} with 
$$
q_h =\Pi_h^Q\left[\phi\mathcal{A}^\text{div}[(\vr_h\ov u_h\psi)(\cdot -\Delta t)]\right].
$$
The resulting expression reads
\begin{align}\label{eff:j31}
		&\int_0^{T}\int_\Om D_t^h(\vr_h \ov u_h)v_h~dxdt \nonumber \\
		& \qquad = \sum_\Gamma \int_{\Delta t}^T\int_\Gamma \up(\vr u)\jump{\Pi_h^Q\left[\phi\mathcal{A}^\text{div}[(\vr_h\ov u_h\psi)(\cdot -\Delta t)]\right] }_\Gamma~dS(x)dt \nonumber\\
		& \qquad \qquad - h^{1-\eps}\sum_\Gamma \int_0^T\int_\Gamma\jump{\vr_h}\jump{\Pi_h^Q\left[\phi\mathcal{A}^\text{div}[(\vr_h\ov u_h\psi)(\cdot -\Delta t)]\right]}_\Gamma~dS(x)dt \nonumber\\
		&\qquad \qquad  -  \int_{\Delta t}^{T}\int_\Om (\vr_h\ov u_h)(\cdot -\Delta t)~ \mathcal{A}^\Grad\left[\phi\vr_h\right]D_t^h\psi ~dxdt \\
		&\qquad \qquad - \int_\Om \vr_h^0\ov u_h^0 v(\Delta t)~dx + \int_0^{T}\int_\Om D_t^h(\vr_h \ov u_h)(v_h-v)~dxdt,\nonumber
\end{align}
By applying \eqref{eff:j31} in the expression for $J_3^h$, we obtain
\begin{align*}%\label{eff:j39}
		J_3^h & = \sum_\Gamma \int_{\Delta t}^T\int_\Gamma \up(\vr u)\jump{\Pi_h^Q\left[\phi\mathcal{A}^\text{div}[(\vr_h\ov u_h\psi)(\cdot -\Delta t)]\right] }_\Gamma~dS(x)dt \nonumber\\
		&\qquad -\sum_\Gamma \int_0^T \int_{\Gamma}\up(\vr  u \otimes \ov u)\jump{\Pi_h^Q\Pi_h^V\left[\psi\mathcal{A}^\Grad[\vr_h \phi]\right]}_\Gamma~dS(x)dt \nonumber\\
		&\qquad - h^{1-\eps}\sum_\Gamma \int_0^T\int_\Gamma\jump{\vr_h}_\Gamma\jump{\Pi_h^Q\left[\phi\mathcal{A}^\text{div}[(\vr_h\ov u_h\psi)(\cdot -\Delta t)]\right]}_\Gamma~dS(x)dt \nonumber\\
		& \qquad -  \int_{\Delta t}^{T}\int_\Om (\vr_h\ov u_h)(\cdot -\Delta t)~ \mathcal{A}^\Grad\left[\phi\vr_h\right]D_t^h\psi ~dxdt \\
		& \qquad- \int_\Om \vr_h^0\ov u_h^0 v(\Delta t)~dx + \int_0^{T}\int_\Om D_t^h(\vr_h \ov u_h)(v_h-v)~dxdt.\nonumber
\end{align*}
In order to apply Lemma \ref{lem:commutator} we need to remove the time shift $-\Delta t$
in the first term. By adding and subtracting, we obtain
\begin{align}\label{eff:j32}
		J_3^h & = \sum_\Gamma \int_{\Delta t}^T\int_\Gamma \up(\vr u)\jump{\Pi_h^Q\left[\phi\mathcal{A}^\text{div}[\vr_h\ov u_h\psi]\right] }_\Gamma~dS(x)dt \nonumber\\
		&\qquad -\sum_\Gamma \int_0^T \int_{\Gamma}\up(\vr  u \otimes \ov u)\jump{\Pi_h^Q\Pi_h^V\left[\psi\mathcal{A}^\Grad[\vr_h \phi]\right]}_\Gamma~dS(x)dt \nonumber\\
		&\qquad -\sum_\Gamma \int_{\Delta t}^T\int_\Gamma \up(\vr u) \Delta t\jump{\Pi_h^Q\left[\phi\mathcal{A}^\text{div}[D_t^h(\vr_h\ov u_h\psi)]\right] }_\Gamma~dS(x)dt\nonumber\\
		&\qquad - h^{1-\eps}\sum_\Gamma \int_{\Delta t}^T\int_\Gamma\jump{\vr_h}\jump{\Pi_h^Q\left[\phi\mathcal{A}^\text{div}[(\vr_h\ov u_h\psi)(\cdot -\Delta t)]\right]}_\Gamma~dS(x)dt \nonumber\\
		& \qquad -  \int_{\Delta t}^{T}\int_\Om (\vr_h\ov u_h)(\cdot -\Delta t)~ \mathcal{A}^\Grad\left[\phi\vr_h\right]D_t^h\psi ~dxdt \\
		& \qquad- \int_\Om \vr_h^0\ov u_h^0 v(\Delta t)~dx + \int_0^{T}\int_\Om D_t^h(\vr_h \ov u_h)(v_h-v)~dxdt.\nonumber \\
		&:= \sum_{i=1}^7 K_i,\nonumber
\end{align}
where the third term contains the time difference. 

Now, observe that $K_1 + K_2$ is exactly the terms 
covered by  Lemma \ref{lem:commutator}. In particular, 
\begin{equation}\label{eff:k1k2}
	\begin{split}
		&\lim_{h \rightarrow 0}K_1 + K_2  \\
		&\quad = \int_0^T\int_\Om \vr u \Grad \left(\phi \mathcal{A}^\text{div}[\psi \vr u]\right) 
						- \vr u \otimes u:\Grad (\psi \mathcal{A}^\Grad [\phi \vr])~dxdt.
	\end{split}
\end{equation}

To bound the $K_3$ term (the new term), we apply Lemma \ref{lem:transport} and Proposition \ref{pro:transport}
to deduce
\begin{align}\label{eq:fuck}
		|K_3| &= \left|\sum_\Gamma \int_0^T\int_\Gamma \up(\vr u) \Delta t\jump{\Pi_h^Q\left[\phi\mathcal{A}^\text{div}[D_t^h(\vr_h\ov u_h\psi)]\right] }_\Gamma~dS(x)dt\right| \nonumber\\
		&= \Delta t\int_{\Delta t}^T\int_\Om \vr_h \tilde u_h: \Grad \left(\phi \mathcal{A}^\text{div}[D_t^h(\vr_h\ov u_h\psi)]\right)~dxdt \nonumber\\
		&\qquad \qquad + (\Delta t)P_1\left(\phi \mathcal{A}^\text{div}[D_t^h(\vr_h\ov u_h\psi)]\right)\\
		&\leq \|\vr_h \tilde u_h\|_{L^2(0,T;L^{\frac{6\gamma}{3+\gamma}}(\Om))}\|\vr_h \ov u_h(t) -\vr_h \ov u_h(t-\Delta t) \|_{L^2(\Delta t,T;L^\frac{6\gamma}{5\gamma - 3}(\Om))} \nonumber\\
		&\qquad + h^{\frac{1}{2}-\max\{3\frac{4-\gamma}{4\gamma},0\}}C\|\vr_h \ov u_h(t) -\vr_h \ov u_h(t-\Delta t) \|_{L^4(\Delta t,T;L^\frac{12}{5}(\Om))} \nonumber\\
		&\leq C\|\vr_h \ov u_h(t) -\vr_h \ov u_h(t-\Delta t) \|_{L^2(\Delta t,T;L^\frac{6\gamma}{5\gamma - 3}(\Om))} \nonumber\\
		&\qquad + h^{\alpha(\delta, \gamma)}C\|\vr_h \ov u_h(t) -\vr_h \ov u_h(t-\Delta t) \|_{L^{2-\delta}(\Delta t,T;L^\frac{12}{5}(\Om))}\nonumber
\end{align}
where $\alpha(\gamma, \delta) = {\frac{1}{4} -\max\{3\frac{4-\gamma}{4\gamma},0\} - \frac{\delta}{2(2-\delta)}}$ 
and $\delta>0$ is a small number.
To achieve the previous inequality, we have applied the inverse inequality in time (Lemma \ref{lemma:inverse}).

Now, from Lemma \ref{lem:timederiv}, we know what 
$$
	\vr_h \ov u_h(t) -\vr_h \ov u_h(t-\Delta t) \rightarrow 0 \quad \text{a.e on $(0,T)\times \Om$}.
$$
Hence, since $\vr_h \ov u_h \in L^\infty(0,T;L^\frac{3}{2}(\Om)) \cap L^2(0,T;L^3(\Om))$,
$$
\vr_h \ov u_h(t) -\vr_h \ov u_h(t-\Delta t) \rightarrow 0 \quad \text{ in $L^{p_1}(0,T; L^{q_1}(\Om))\cap L^{p_2}(0,T; L^{q_2}(\Om))$},
$$
for any $p_1 < \infty$, $q_1 < \frac{2\gamma}{\gamma + 1}$, $p_2 < 2$, and $q_2 < \frac{6\gamma}{3+\gamma}$. Moreover, 
since $\gamma > 3$, we can choose $\delta$ small such that $\alpha(\gamma, \delta)>0$. 
As a consequence, we can pass to the limit in \eqref{eq:fuck} to obtain
\begin{equation}\label{eff:k3}
	\lim_{h \rightarrow 0}|K_3| = 0.
\end{equation}

The $K_4$ term is directly bounded by Lemma \ref{lem:stupid};
\begin{equation}\label{eff:k4}
	\begin{split}
		&|K_4| = \left|h^{1-\eps}\sum_\Gamma \int_{\Delta t}^T\int_\Gamma\jump{\vr_h}_\Gamma\jump{\Pi_h^Q\left[\phi\mathcal{A}^\text{div}[\vr_h\ov u_h\psi(\cdot -\Delta t)]\right]}_\Gamma~dS(x)dt\right| \\
		&\qquad  \leq h^\frac{11-4\eps}{8}C\left\|\Grad \left(\phi\mathcal{A}^\text{div}[\vr_h\ov u_h\psi(\cdot -\Delta t)] \right)\right\|_{L^2(\Delta t,T;L^2(\Om))}\\
		&\qquad  \leq h^\frac{11-4\eps}{8}C\|\vr_h \ov u_h\|_{L^2(0,T;L^2(\Om))},
	\end{split}
\end{equation}
where the norm is bounded by Corollary \ref{cor:energy}.

Next, since 
$$
\vr_h \ov u_h \weak \vr u \quad \text{in $L^2(0,T;L^{m_2}(\Om))$, $m_2 > 3$,}
$$ 
and $\mathcal{A}^\Grad\left[\phi\vr_h\right] \rightarrow \mathcal{A}^\Grad[\phi \vr]$ in $C(0,T;L^p(\Om))$, for 
any $p < \infty$, we see that
\begin{equation}\label{eff:k5k6}
	\begin{split}
		&\lim_{h \rightarrow 0} \left(K_5 + K_6 \right) \\
		&=  - \lim_{h\rightarrow 0} \int_{\Delta t}^{T}\int_\Om (\vr_h\ov u_h)(\cdot -\Delta t)~ \mathcal{A}^\Grad\left[\phi\vr_h\right]D_t^h\psi ~dxdt \\
		&\quad - \lim_{h\rightarrow 0} \int_\Om \vr_h^0\ov u_h^0 v(\Delta t)~dx \\
		&= -\int_0^T\int_\Om \vr u\mathcal{A}^\Grad[\phi \vr]\psi_t~dxdt - \int_\Om m_0w_0~dx.
	\end{split}
\end{equation}

To bound $K_7$, we apply Lemma \ref{lem:timederiv} and the interpolation error for $\Pi_h^V$ 
and obtain
\begin{equation}\label{eff:k7}
	\begin{split}
		|K_7| = \left|\int_0^{T}\int_\Om D_t^h(\vr_h \ov u_h)(v_h-v)~dxdt\right| 
		&\leq h^\frac{1}{\gamma}C\|\Grad v\|_{L^\infty(0,T;L^3(\Om))} \\
		&\leq h^\frac{1}{\gamma}C\|\vr_h\|_{L^\infty(0,T;L^\gamma(\Om))},
	\end{split}
\end{equation}
which is bounded by Corollary \ref{cor:energy}.

At this stage, we can send $h \rightarrow 0$ in \eqref{eff:j32},
using \eqref{eff:k1k2}, \eqref{eff:k3}, \eqref{eff:k4}, \eqref{eff:k5k6}, 
and \eqref{eff:k7}, to obtain
\begin{align}\label{eff:j35}
		\lim_{h \rightarrow 0} J_3^h
		& = \int_0^T\int_\Om \vr u \Grad \left(\phi \mathcal{A}^\text{div}[\psi \vr u]\right) 
						- \vr u \otimes u:\Grad (\psi \mathcal{A}^\Grad [\phi \vr])~dxdt\nonumber\\
		&\quad -\int_0^T\int_\Om \vr u\mathcal{A}^\Grad[\phi \vr]\psi_t~dxdt - \int_\Om m_0w_0~dx. 
\end{align}

Finally, since the limit $(\vr, u)$ is a solution to the continuity equation 
(Lemma \ref{lem:cont}), we have that (recall that $\phi$ does not depend on time)
\begin{equation*}
	\begin{split}
		&\int_0^T\int_\Om \vr u \Grad \left(\phi \mathcal{A}^\text{div}[\psi \vr u]\right) \\ 
		&\qquad= -\int_0^T\int_\Om \phi\vr\left( \mathcal{A}^\text{div}[\psi \vr u]\right)_t~dxdt -\int_\Om\phi \vr_0 \mathcal{A}^\text{div}[\psi(0,\cdot) m_0]~dx\\
		&\qquad =-\int_0^T\int_\Om \left(\mathcal{A}^\Grad[\vr \phi]\right)_t\psi \vr u~dxdt \\
		&\qquad = -\int_0^T\int_\Om \vr u w_t~dxdt+\int_{0}^{T}\int_\Om \vr u~ \mathcal{A}^\Grad\left[\phi\vr\right]\psi_t ~dxdt.
	\end{split}
\end{equation*}
Thus, by using this identity in \eqref{eff:j35}, we obtain
\begin{equation*}\label{eff:J3}
	\begin{split}
		\lim_{h \rightarrow 0} J_3^h &= -\int_0^T\int_\Om \vr u w_t + \vr u\otimes u :\Grad w~dxdt -\int_\Om m_0 w(0)~dx \\
		&= J_3, 
	\end{split}
\end{equation*}
which was the desired result.
\end{proof}

The proof of Proposition \ref{pro:effectiveflux} is now complete.

\section{Strong convergence and proof of Theorem \ref{thm:main}}\label{sec:strong}
Equipped with Proposition \ref{pro:effectiveflux}
we can now obtain strong convergence of 
the density approximation.  
The following argument  uses only the continuity approximation 
and is very similar to the corresponding argument in \cite{Karlsen1, Karlsen2, Karlsen3}. 
We include it here for the sake of completeness.

\begin{lemma}\label{lem:strong}
Let $(\vr_h, u_h)$ be the numerical solution obtain through Definition \ref{def:scheme} 
and \eqref{def:schemeII}. The density approximation converges a.e
	\begin{equation*}
		\vr_h \rightarrow \vr \quad \text{a.e on $(0,T)\times \Om$}.
	\end{equation*}
\end{lemma}

\begin{proof}
From Lemma \ref{lem:renorm} with $B(\vr) = \vr \log \vr$, 
we have the inequality
\begin{equation*}
	\begin{split}
		\int_\Om \vr_h \log \vr_h~ dx~ (t) - \int_\Om \vr^0_h \log \vr^0_h~ dx  & \leq -\int_0^{t}\int_\Om \vr_h\Div u_h~dxdt.
	\end{split}
\end{equation*}
From Lemma \ref{lemma:feireisl}, we know that the limit $(\vr, u)$ is a renormalized solution 
to the continuity equation. In particular, the following identity holds
\begin{equation*}
	\begin{split}
		\int_\Om \vr \log \vr~ dx~ (t) - \int_\Om \vr_0 \log \vr_0~ dx  & = -\int_0^{t}\int_\Om \vr\Div u~dxdt.
	\end{split}
\end{equation*}
Note that there is no problems with integrability of $\vr \Div u$. Indeed, 
since $\vr \in L^\infty(0,T;L^\gamma(\Om ))$
with $\gamma > 3$ and $\Div u \in L^2(0,T;L^2(\Om))$, H\"olders inequality 
provides  $\vr \Div u \in L^2(0,T;L^\frac{6}{5}(\Om))$.

By subtracting the two previous identities, we obtain
\begin{align*}\label{strong:start}
		&\int_\Om \vr_h \log \vr_h - \vr \log \vr~dx~ (t) \nonumber \\
		&\quad \leq \int_0^{t}\int_\Om \vr \Div u - \vr_h \Div_h u_h~dxdt
		+ \int_\Om \vr_h^0 \log \vr_h^0 - \vr_0 \log \vr_0~dx~ \nonumber \\
		&\quad = \int_0^{t^k}\int_\Om \phi^2(\vr \Div u - \vr_h \Div_h u_h)~dxdt \\
		&\qquad \qquad \quad+ \int_0^{t^k}\int_\Om (1-\phi^2)(\vr \Div u - \vr_h \Div_h u_h)~dxdt \nonumber\\
		&\qquad\qquad  \quad+ \int_\Om \vr_h^0 \log \vr_h^0 - \vr_0 \log \vr_0~dx\nonumber .
\end{align*}
From Proposition \ref{pro:effectiveflux}, we have that
\begin{equation*}
	\begin{split}
		&\lim_{h \rightarrow 0}\int_0^{t^k}\int_\Om \phi^2(\vr \Div u - \vr_h \Div u_h)~dxdt \\
		&\qquad = \lim_{h \rightarrow 0}\int_0^{t^k}\int_\Om \phi^2(\vr p(\vr) - \vr_h p(\vr_h))~dxdt \leq 0, 
	\end{split}
\end{equation*}
where the last inequality follows from the  convexity of $p(\vr)$. Hence, 
this, together with the strong convergence of the initial density, allow us to conclude 
\begin{equation*}
	\begin{split}
		\int_\Om \overline{\vr \log \vr} - \vr \log \vr~dx~ (t)
		\leq \int_0^{t}\int_\Om (1 - \phi^2)Q~dxdt, 
	\end{split}
\end{equation*}
where $Q = \vr \Div u - \overline{\vr \Div u} \in L^2(0,T;L^\frac{6}{5}(\Om))$.
Hence, we see that
\begin{equation*}
	\begin{split}
		0\leq \int_\Om \overline{\vr \log \vr} - \vr \log \vr~dx~ (t)
		\leq C\left\|1- \phi^2\right\|_{L^6(\Om)},
	\end{split}
\end{equation*}
where the first inequality follows by convexity of $z \mapsto z \log z$.
Now, for any $\epsilon > 0$, we can choose $\phi \in W^{1,\infty}(\Om)$ such 
that 
\begin{equation*}
	\begin{split}
		0\leq \int_\Om \overline{\vr \log \vr} - \vr \log \vr~dx~ (t)
		\leq \epsilon.
	\end{split}
\end{equation*}
As a consequence, $\overline{\vr \log \vr} - \vr \log \vr = 0$ a.e and hence 
$\vr_h \rightarrow \vr$ a.e in $(0,T) \times \Om$.
\end{proof}

\subsection{Proof of the main result (Theorem \ref{thm:main})}\label{sec:proof}
To conclude the main result, the only remaining part is to prove that $(\vr, u)$ is a weak solution 
of the momentum equation \eqref{eq:moment} and that it satisfies the energy inequality \eqref{eq:cont-energy}. The other 
statements in Theorem \ref{thm:main} are all covered by \eqref{eq:conv1}, Lemma \ref{lem:product}, Lemma \ref{lem:cont}, 
 and Lemma \ref{lem:strong}. 
Since we now know that $\vr_h \rightarrow \vr$ a.e, we 
have in particular
$$
\overline{p(\vr)} = p(\vr) \quad \text{a.e in $(0,T)\times \Om$}.
$$
By applying this information in \eqref{eq:limbar}, we immediately see
that $(\vr, u)$ is a weak-solution of the momentum equation. 

By passing to the limit $h \rightarrow 0$ in the numerical energy inequality 
\eqref{eq:energy} (Lemma \ref{prop:energy}), using convexity, we 
discover that the limit $(\vr, u)$ satisfies the energy inequality \eqref{eq:cont-energy}.

\appendix

\section{Existence of a numerical solution}
Since the numerical method in Definition \ref{def:scheme} 
is nonlinear and implicit it is not trivial that it 
is actually well-defined (i.e admits a solution). 
In addition, the discretization of the momentum transport 
is posed using element averages of the velocity.  
As we will see the only part of the discretization 
that provides sufficient number of equations to 
determine all the degrees of freedom of $u_h$ is 
the discretization of the diffusion operator. 
Hence, in it's present form, our discretization is not
suitable for the Euler equations. 

The purpose of this section, is to prove the following
the existence result which we have relied on in our analysis.

\vspace{0.2cm}
\noindent
\textbf{Proposition \ref{pro:defined}.} 
{\it
For each fixed $h > 0$, there exists 
a solution 
\begin{equation*}
	(\vr_h^k, u_h^k) \in Q_h(\Om)\times V_h(\Om),\quad \vr^k_h(\cdot) > 0, \quad k=1, \ldots, M,
\end{equation*}
to the numerical method posed in Definition \ref{def:scheme}.}
\vspace{0.2cm}

To prove this result, we shall use a topological degree argument. 
The argument is strongly inspired by a very similar argument in the paper \cite{Gallouet:2006lr}.
We will argue the existence of solutions 
to the following finite element map.
\begin{definition}\label{def:H}
Let the finite element map \\$H:Q_h^+(\Om) \times V_h(\Om) \times [0,1] ~\mapsto~ Q_h(\Om) \times V_h(\Om)$
be given by
	\begin{equation*}
		\begin{split}
			&H(\vr_h, u_h, \alpha) = (f_h(\alpha), g_h(\alpha)),
		\end{split}
	\end{equation*}
	where $(f_h(\alpha), g_h(\alpha))$ are obtained through the mappings:
	\begin{equation}\label{fix:cont}
		\begin{split}
			\int_\Om f_h(\alpha)q_h ~dx &= \int_\Om \frac{\vr_h - \vr_h^{k-1}}{\Delta t} q_h~dx \\
			&\qquad - \alpha\sum_\Gamma \int_{\Gamma} \up(\vr u)\jump{q_h}_\Gamma~dS(x) \\
			&\qquad + \alpha h^{1-\epsilon}\sum_\Gamma \int_{\Gamma}\jump{\vr_h}_\Gamma\jump{q_h}_\Gamma~dS(x), 
		\end{split}
	\end{equation}
	for all $q_h \in Q_h(\Om)$ and
	\begin{equation}\label{fix:moment}
		\begin{split}
			\int_\Om g_h(\alpha)v_h~dx &= \int_\Om \frac{\vr_h\ov u_h - \vr_h^{k-1}\ov u^{k-1}_h}{\Delta t} v_h~dx + \int_\Om \Grad_h u_h \Grad_h v_h~dx \\
			&\qquad- \alpha\sum_{ \Gamma} \int_{\Gamma} \up(\vr u\otimes \ov u)\jump{\ov v}_\Gamma ~dS(x) \\
			& \qquad   - \alpha\int_\Om p(\vr_h)\Div v_h~dx \\
			&\qquad- \alpha h^{1-\epsilon}\sum_E \int_{\partial E} \left(\frac{\ov u_- + \ov u_+}2\right) \jump{\vr_h}\ov v_h~dS(x),
			\end{split}
	\end{equation}	
	for all $v_h \in V_h(\Om)$.
\end{definition}
Observe that a solution of $H(\vr_h, u_h, 1) = (0,0)$ is a solution 
to our numerical method as posed in Definition \ref{def:scheme}.

Before proceeding, let us make clear what we mean by topological degree in the 
present finite element context
and denote by $d(F,\Om,y)$ the $\mathbb{Z}$--valued (Brouwer) degree of a continuous function 
$F:\bar{O} \rightarrow \mathbb{R}^M$ at a point $y\in \mathbb{R}^M\backslash F(\partial O)$ 
relative to an open and bounded set $O \subset \mathbb{R}^M$. 

\begin{definition}\label{def:femtop}
Let $S_{h}$ be a finite element space, $\|\cdot \|$ be a norm on this space, and 
introduce the bounded set
$$
\tilde{S}_{h} = \Set{q_{h} \in S_{h}; \|q_{h}\| \leq C},
$$
where $C>0$ is a constant. Let $\{\sigma_{i}\}_{i=1}^M$ be
a basis such that $\operatorname{span}\{\sigma_{i}\}_{i=1}^M = S_{h}$ and 
define the operator $\Pi_{\mathcal{B}}:S_{h}  \rightarrow \mathbb{R}^M$ by 
$$
\Pi_{\mathcal{B}}q_{h} = (q_{1},q_{2}, \ldots, q_{M}),
\qquad q_{h} = \sum_{i=1}^M q_{i}\sigma_{i}.
$$ 
The degree $d_{S_{h}}(F,\tilde{S}_{h},q_{h})$ of a continuous 
mapping $F:\tilde{S}_{h} \rightarrow S_{h}$ at $q_{h} 
\in S_{h}\backslash F(\partial \tilde{S}_h)$
relative to $\tilde{S}_{h}$ is defined as
\begin{equation*}
	d_{S_{h}}(F,\tilde{S}_{h},q_{h}) 
	= d\left( \Pi_{\mathcal{B}} F(\Pi_{\mathcal{B}}^{-1}), \Pi_{\mathcal{B}}
	\tilde{S}_{h}, \Pi_{\mathcal{B}}q_{h}\right).
\end{equation*}
\end{definition}

The next lemma is a consequence of some basic properties 
of the degree, cf.~\cite{Deimling:1985rt}.
\begin{lemma}\label{lem:degree}
Fix a finite element space $S_{h}$, and let  $d_{S_{h}}(F, \tilde{S}_{h},q_{h})$ 
be the associated degree of Definition \ref{def:femtop}. The following properties hold:
\begin{enumerate}
	\item{} $d_{S_{h}}(F,\tilde{S}_{h},q_{h})$ does not depend on the choice of basis for $S_{h}$.
	\item{} $d_{S_{h}}(\mathrm{Id},\tilde{S}_{h}, q_{h}) = 1$.
	\item{} $d_{S_{h}}(H(\cdot,\alpha),\tilde{S}_{h},q_{h}(\alpha))$ is 
	independent of $\alpha \in J:=[0,1]$ for $H\!:\!\tilde{S}_{h}\times J \rightarrow S_{h}$ 
	continuous, $q_h:J \rightarrow S_{h}$ continuous, and 
	$q_{h}(\alpha) \notin H(\partial \tilde{S}_{h},\alpha)$ $\forall \alpha \in [0,1]$.
	\item{}$d_{S_{h}}(F,\tilde{S}_{h},q_{h}) \neq 0 \Rightarrow F^{-1}(q_{h}) \neq \emptyset$.
\end{enumerate}
\end{lemma}

To prove Proposition \ref{pro:defined}, we shall apply 
Lemma \ref{lem:degree} with $q_h = 0$ and mapping $H$ given by 
Definition \ref{def:H}. 
Let us first prove that our mapping $H$ satisfies (3) in Lemma \ref{lem:degree}.

\begin{lemma}\label{lem:}
Let $H:Q_h^+(\Om) \times V_h(\Om) \times [0,1] ~\mapsto~ Q_h(\Om) \times V_h(\Om)$
be the finite element mapping of Definition \ref{def:H}. 
There is a subset $\tilde S_h \subset Q_h^+(\Om) \times V_h(\Om)$
for which $H\!:\!\tilde{S}_{h}\times J \rightarrow S_{h}$ is continuous 
and the zero solution $(0,0) \not \in H(\partial \tilde S, \alpha)$ for all $\alpha \in [0, 1]$.
\end{lemma}
\begin{proof}
For any subset $\tilde S \subset Q_h^+(\Om) \times V_h(\Om)$
bounded independently of $\alpha$,
the corresponding mapping $H(\tilde S_h, \alpha, 0)$ 
is clearly continuous. This follows directly 
from \eqref{fix:cont} and \eqref{fix:moment}
using the equivalence of finite dimensional norms.
The more involved part  is to determine a subset $\tilde S_h$
for which $(0,0) \not \in H(\partial \tilde S, \alpha)$ independently of $\alpha$.

Now, let us for the moment assert the existence of
of $(\vr, u)$ satisfying
\begin{equation*}
	H(\vr_h, u_h, \alpha) = (0, 0).
\end{equation*}
Then, from Lemma \ref{lemma:vrho-props}, we have that $\vr_h > 0$ and moreover 
\eqref{fix:cont} yields
\begin{equation*}
	\int_\Om \vr_h~dx = \int_\Om \vr_h^{k-1}~dx.
\end{equation*}
Consequently, we can conclude that
\begin{equation}\label{ap:vri}
	\|\vr_h\|_{L^\infty(\Om)} \leq C_\dagger,
\end{equation}
independently of $\alpha$.

To derive a bound on the velocity $u_h$, we can repeat the steps in the proof of 
Proposition \ref{prop:energy} (the energy estimate) while
keeping track of $\alpha$ to obtain
\begin{equation*}
	\begin{split}
		&\int_\Om \frac{\vr_h |\ov u_h|^2}{2} + \frac{\alpha}{\gamma-1}p(\vr_h)~dx + \Delta t\int_\Om |\Grad u_h|^2~dx \\
		&\qquad \leq \int_\Om \frac{\vr^{k-1}_h |\ov u^{k-1}_h|^2}{2} + \frac{\alpha}{\gamma-1}p(\vr^{k-1}_h)~dx \\
		&\qquad \leq \int_\Om \frac{\vr^{k-1}_h |\ov u^{k-1}_h|^2}{2} + \frac{1}{\gamma-1}p(\vr^{k-1}_h)~dx \leq C,
	\end{split}
\end{equation*}
where $C$ is independent of $\alpha$. Together with \eqref{ap:vri}, this allow us to conclude
\begin{equation*}
	\|\vr_h\|_{L^\infty(\Om)} + \|u_h\|_{L^\infty(\Om)} \leq C_\dagger.
\end{equation*}
We can now define the subspace
\begin{equation*}
	\tilde S_h = \Set{(\vr_h, u_h) \in Q_h^+\times V_h;  \|\vr_h\|_{L^\infty(\Om)} + \|u_h\|_{L^\infty(\Om)} \leq C_\dagger},
\end{equation*}
which by definition has the property that $(0,0) \not \in H(\partial \tilde S, \alpha)$ for all $\alpha \in [0,1]$.
This concludes the proof.
\end{proof}

\begin{lemma}\label{lem:}
Let $\tilde S$ be the subspace obtained by the previous lemma.
Then, the topological degree of $H(\tilde S_h, 0)$ at $q_h = 0$ 
is non-zero: 
\begin{equation}\label{eq:degstate}
	d_{S_h}(H(\cdot, 0), \tilde S_h, 0) \neq 0.
\end{equation}
As a consequence, there exists $(\vr_h, u_h) \in \tilde S_h$ such that
\begin{equation*}
	H(\vr_h, u_h, 1) = (0,0),
\end{equation*}
and hence Proposition \ref{pro:defined} holds true.
\end{lemma}
\begin{proof}
First, we note that proving \eqref{eq:degstate} is equivalent 
to proving the existence of $(\vr_h, u_h) \in Q^+_h \times V_h$ satisfying, 
for all $(q_h, v_h) \in Q_h\times V_h$,
\begin{equation}\label{eq:final}
	\begin{split}
		&\int_\Om \vr_h q_h~dx = \int_\Om \vr^{k-1}_h q_h~dx, \\
		&\int_\Om \vr_h \ov u_h v_h~dx + \Delta t \int_\Om \Grad_h u_h \Grad_h v_h~dx = \int_\Om \vr_h^{k-1}\ov u_h^{k-1} v_h~dx. 
	\end{split}
\end{equation}
The first equation has the solution $\vr_h = \vr_h^{k-1}$. 
Setting this into the
second equation in \eqref{eq:final},  we see that the resulting linear system is 
a sum of a positive matrix $\vr^{k-1}_h \ov u_h v_h$
and a symmetric positive definite matrix $\Delta t\Grad_h u_h \Grad_h v_h$. 
Since the Laplace problem with the Crouzeix-Raviart element space and 
dirichlet conditions is well-defined, there is no problems with 
concluding the existence of $u_h$ satisfying the second 
equation in \eqref{eq:final}.

\end{proof}

\end{document}